\documentclass{article}
\usepackage[english]{babel}
\usepackage[utf8]{inputenc}
\usepackage{amsmath}
\usepackage{graphicx}
\usepackage{amsfonts}
\usepackage{enumitem}
\usepackage{amssymb}
\usepackage[colorinlistoftodos]{todonotes}
\usepackage{geometry}
\usepackage{amsthm}
\usepackage{comment}
\usepackage{hyperref}
\usepackage{tikz}
\usetikzlibrary{math}

\newcommand{\norm}[1]{\left\lVert#1\right\rVert}

\newcommand{\ang}[1]{\left\langle #1 \right\rangle}
\newcommand{\floor}[1]{\left\lfloor #1 \right\rfloor}
\newcommand{\ceil}[1]{\left\lceil #1 \right\rceil}
\newcommand{\paren}[1]{\left( #1 \right)}
\newcommand{\sqb}[1]{\left[ #1 \right]}
\newcommand{\set}[1]{\left\{ #1 \right\}}
\newcommand{\setcond}[2]{\left\{ #1 \;\middle\vert\; #2 \right\}}

\newcommand{\CC}{\mathbb{C}}

\newcommand{\RR}{\mathbb{R}}

\newcommand{\NN}{\mathbb{N}}
\newcommand{\ZZ}{\mathbb{Z}}
\newcommand{\QQ}{\mathbb{Q}}

\newcommand{\cD}{\mathcal{D}}

\newcommand{\cF}{\mathcal{F}}
\newcommand{\cG}{\mathcal{G}}

\newcommand{\cI}{\mathcal{I}}

\newcommand{\cL}{\mathcal{L}}
\newcommand{\cM}{\mathcal{M}}

\newcommand{\cO}{\mathcal{O}}
\newcommand{\cP}{\mathcal{P}}
\newcommand{\cQ}{\mathcal{Q}}

\newcommand{\fraka}{\mathfrak{a}}
\newcommand{\frakb}{\mathfrak{b}}
\newcommand{\frakc}{\mathfrak{c}}
\newcommand{\frakD}{\mathfrak{D}}
\newcommand{\frakp}{\mathfrak{p}}

\newtheorem{thm}{Theorem}[section]
\newtheorem{lem}[thm]{Lemma}
\newtheorem{cor}[thm]{Corollary}
\newtheorem{conj}[thm]{Conjecture}
\newtheorem{claim}[thm]{Claim}
\newtheorem{obs}[thm]{Observation}
\newtheorem{prop}[thm]{Proposition}
\newtheorem*{rmk}{Remark}

\theoremstyle{definition}
\newtheorem{defn}[thm]{Definition}
\newtheorem{example}[thm]{Example}

\DeclareMathOperator{\GL}{GL}
\DeclareMathOperator{\Mat}{Mat}

\DeclareMathOperator{\Vol}{Vol}
\DeclareMathOperator{\LD}{LD}
\DeclareMathOperator{\diag}{diag}

\DeclareMathOperator{\Tr}{Tr}
\DeclareMathOperator{\cont}{cont}

\title{Sums of algebraic dilates}
\author{David Conlon\thanks{Department of Mathematics, Caltech, Pasadena, CA 91125, USA. Email: {\tt dconlon@caltech.edu}. Research supported by NSF Awards DMS-2054452 and DMS-2348859.} \and Jeck Lim\thanks{Department of Mathematics, Caltech, Pasadena, CA 91125, USA. Email: {\tt jlim@caltech.edu}. Research partially supported by an NUS Overseas Graduate Scholarship.}}
\date{}

\begin{document}
\maketitle

\begin{abstract}
We show that if $\lambda_1,\ldots,\lambda_k$ are algebraic numbers, then
$$|A+\lambda_1\cdot A+\dots+\lambda_k\cdot A|\geq H(\lambda_1,\ldots,\lambda_k)|A|-o(|A|)$$
for all finite subsets $A$ of $\mathbb{C}$, where $H(\lambda_1,\ldots,\lambda_k)$ is an explicit constant that is best possible. The proof combines several ingredients, including a lower bound estimate on the measure of sums of linear transformations of compact sets in $\mathbb{R}^d$, 
a variant of Freiman's theorem tuned specifically to sums of dilates and the analysis of what we call lattice density, which succinctly captures how a subset of $\mathbb{Z}^d$ is arranged relative to a given flag of lattices. 
As an application, we revisit the study of sums of linear transformations of finite sets, in particular proving an asymptotically best possible lower bound for sums of two linear transformations. 
\end{abstract}

\section{Introduction}

For any subset $A$ of $\CC$ and $\lambda_1, \dots, \lambda_k \in \CC$, the sum of dilates $A + \lambda_1 \cdot A + \cdots + \lambda_k \cdot A$ is given by
\[A + \lambda_1 \cdot A + \cdots + \lambda_k \cdot A := \{a_0 + \lambda_1 a_1 + \cdots + \lambda_k a_k : a_0, a_1, \dots, a_k \in A\}.\]
Our concern in this paper will be with estimating the minimum size of $|A + \lambda_1 \cdot A + \cdots + \lambda_k \cdot A|$ in terms of $|A|$. For $\lambda_1, \dots, \lambda_k \in \QQ$, this problem was essentially solved by Bukh~\cite{B08}, from whose results it follows that if $\lambda_i = p_i/q$ for $q$ as small as possible for such a common denominator, then 
\[|A + \lambda_1 \cdot A + \cdots + \lambda_k \cdot A| \geq (|p_1| + \cdots + |p_k| + |q|)|A| - o(|A|)\]
for all finite subsets $A$ of $\CC$, which is best possible up to the lower-order term. This result was later sharpened by Balog and Shakan~\cite{BS14} when $k = 1$ and then Shakan~\cite{S16} in the general case, improving the $o(|A|)$ term to a constant depending only on $\lambda_1, \dots, \lambda_k$. 

When at least one of the $\lambda_i$ is transcendental, it was shown by Konyagin and \L aba~\cite{KL06} that 
\[|A + \lambda_1 \cdot A + \cdots + \lambda_k \cdot A| = \omega(|A|).\]
The problem of giving more precise lower bounds for $|A + \lambda \cdot A|$ when $\lambda$ is transcendental was studied in some depth by Sanders~\cite{S08, S12} and Schoen~\cite{Sch11}, with progress tied to advances in quantitative estimates for Freiman's theorem on sets of small doubling.
Using quite different techniques, Conlon and Lim~\cite{CL22} recently resolved this problem, 
showing that there is a constant $c$ such that
\[|A+ \lambda \cdot A| \geq e^{c \sqrt{\log |A|}}|A|,\]
which, by a construction of Konyagin and \L aba, is best possible up to the value of $c$.

Our focus here will be on the complementary case, where each of $\lambda_1, \dots, \lambda_k$ is algebraic. Early results in this direction were proved by Breuillard and Green~\cite{BG13} and Chen and Fang~\cite{CF18}, with the latter showing that, for any fixed $\lambda \ge 1$, $|A + \lambda \cdot A| \ge (1 + \lambda)|A| - o(|A|)$ for all finite subsets $A$ of $\RR$. The problem of giving more precise lower bounds for $|A + \lambda \cdot A|$ when $\lambda$ is algebraic was raised explicitly by Shakan~\cite{S16} and by Krachun and Petrov~\cite{KP23}, with the latter authors conducting the first systematic study and making the first concrete conjectures. 

To state their conjecture, suppose that $f(x) \in \ZZ[x]$ is the minimal polynomial of $\lambda$, assumed to have coprime coefficients, and $f(x) = \prod_{i=1}^d (a_i x + b_i)$ is a full complex factorisation of $f$. If we set $H(\lambda) := \prod_{i=1}^d (|a_i| + |b_i|)$, the conjecture of Krachun and Petrov~\cite{KP23} is then as follows.

\begin{conj} \label{conj:KP}
For any algebraic number $\lambda$,
\[|A + \lambda \cdot A| \ge H(\lambda)|A| - o(|A|)\]
for all finite subsets $A$ of $\CC$.
\end{conj}

Krachun and Petrov~\cite{KP23} gave some evidence for their conjecture by proving it in the special case where $\lambda = \sqrt{2}$. Subsequently, as a consequence of their work~\cite{CL23} on a conjecture of Bukh regarding sums of linear transformations, Conlon and Lim verified the conjecture for all $\lambda$ of the form $(p/q)^{1/d}$ with $p, q, d \in \NN$. Assuming all of $p$, $q$ and $d$ are as small as possible for such a representation, their result, which includes that of Krachun and Petrov, says that
\[|A + \lambda \cdot A| \ge (p^{1/d} + q^{1/d})^d |A| - o(|A|)\]
for all finite subsets $A$ of $\CC$. 
Their results also imply a general lower bound for sums of algebraic dilates, though this bound only matches the conjectured one in some special cases.

More recently, Krachun and Petrov~\cite{KP24} have revisited the problem, proving their conjecture in full whenever $\lambda$ is an algebraic integer. This is somewhat incomparable to the result of Conlon and Lim, since $(p/q)^{1/d}$, when written in lowest terms, is only an algebraic integer when $q = 1$. Here we again revisit the problem, proving Conjecture~\ref{conj:KP} in full for all algebraic numbers. Our method also extends to longer sums of algebraic dilates, so we will state our results in that level of generality. 

To state the result, given algebraic numbers $\lambda_1, \dots, \lambda_k$, recall that if the field extension $K := \QQ(\lambda_1,\ldots,\lambda_k)$ of $\QQ$ is of degree $d=\deg(K/\QQ)$, then there are exactly $d$ different complex embeddings $\sigma_1,\ldots,\sigma_d:K\to\CC$. We also need to define the \emph{denominator ideal} (see, for example,~\cite{SL22}), which is the ideal in the ring of integers $\cO_K$ given by 
$$\frakD_{\lambda_1,\ldots,\lambda_k;K}:=\setcond{x\in\cO_K}{x\lambda_l\in\cO_K \text{ for } l=1,\ldots,k}.$$
The key quantity $H(\lambda_1,\ldots,\lambda_k)$ that plays the role of $H(\lambda)$ for sums of many algebraic dilates is then
$$H(\lambda_1,\ldots,\lambda_k) := N_{K/\QQ}(\frakD_{\lambda_1,\ldots,\lambda_k;K})\prod_{i=1}^d (1+|\sigma_i(\lambda_1)|+|\sigma_i(\lambda_2)|+\cdots+|\sigma_i(\lambda_k)|),$$
where $N_{K/\QQ}(\frakD_{\lambda_1,\ldots,\lambda_k;K})$ is the ideal norm of $\frakD_{\lambda_1,\ldots,\lambda_k;K}$, equal to  $[\cO_K:\frakD_{\lambda_1,\ldots,\lambda_k;K}]$. 
To see that this indeed generalises $H(\lambda)$, observe that we can write the minimal polynomial $f(x)\in \ZZ[x]$ of $\lambda$ as $f(x)=D(x-\lambda_1)(x-\lambda_2)\cdots (x-\lambda_d)$ for some integer $D$ and $\lambda_1,\ldots,\lambda_d$ the conjugates of $\lambda$. Then $H(\lambda)=|D|(1+|\lambda_1|)\cdots (1+|\lambda_d|)$ and it can be shown that $|D|=N_{K/\QQ}(\frakD_{\lambda;K})$. 
With this definition in place, our main result, which is best possible up to the behaviour of the lower-order term, is as follows.

\begin{thm} \label{thm:main}
For any algebraic numbers $\lambda_1,\ldots,\lambda_k$,
$$|A+\lambda_1\cdot A+\cdots+\lambda_k\cdot A|\geq H(\lambda_1,\ldots,\lambda_k)|A|-o(|A|)$$
for all finite subsets $A$ of $\CC$.
\end{thm}

In practice, we will view the problem of estimating sums of algebraic dilates as one about estimating sums of linear transformations. More precisely, if we consider the number field $K = Q(\lambda_1, \dots, \lambda_k)$ as a vector space over $\QQ$, that is, as $\QQ^d$ with $d=\deg(K/\QQ)$, then multiplication by $\lambda_i$ becomes a linear map $\cM_i$ from $\QQ^d$ to itself, so the problem of giving a lower bound for $|A+\lambda_1\cdot A+\cdots+\lambda_k\cdot A|$ for $A \subset \RR$ becomes equivalent to the analogous problem for $|A+\cM_1 A+\cdots+\cM_k A|$ for $A \subset \QQ^d$. A further reduction (see Section~\ref{sec:mapZd} for  details) then recasts the problem in terms of estimating $|\cL_0A+\cL_1 A+\cdots+\cL_k A|$ for $\cL_0, \cL_1,\ldots,\cL_k\in \Mat_d(\ZZ)$ and $A \subset \ZZ^d$.

Such sums of linear transformations have been studied before~\cite{BJM24, CL23, M19}, with much of the motivation coming from a conjecture of Bukh asking whether a discrete Brunn--Minkowski-type inequality holds for sums of linear transformations. A corrected version of his original conjecture, first stated in~\cite{CL23}, is as follows.

\begin{conj} \label{conj:bukh}
Suppose that $\cL_0,\ldots,\cL_k\in \Mat_d(\ZZ)$ are irreducible and coprime. Then
$$|\cL_0 A+\cdots +\cL_k A|\geq \paren{|\det(\cL_0)|^{1/d}+\cdots +|\det(\cL_k)|^{1/d}}^d|A| - o(|A|)$$
for all finite subsets $A$ of $\ZZ^d$.
\end{conj}

The conditions on $\cL_0,\ldots,\cL_k$, that they be irreducible and coprime, are necessary, with irreducibility guaranteeing that the problem does not reduce to one of lower dimension and coprimeness that it cannot be restated in terms of matrices with smaller determinants. The formal definitions are as follows, though we refer the reader to~\cite{CL23} for a more complete discussion and some illustrative examples. 

\begin{defn}
We say that $\cL_0,\ldots,\cL_k\in \Mat_d(\ZZ)$ are \textit{irreducible} if there are no non-trivial subspaces $U$, $V$ of $\QQ^d$ of the same dimension such that $\cL_i U\subseteq V$ for all $i$.
\end{defn}

\begin{defn}
We say that $\cL_0,\ldots,\cL_k\in\Mat_d(\ZZ)$ are \textit{coprime} if there are no 
$\cP,\cQ\in\GL_d(\QQ)$ with $0<|\det(\cP)\det(\cQ)|<1$ such that
$$\cP\cL_0\cQ,\cP\cL_1\cQ,\ldots, \cP\cL_k\cQ\in \Mat_d(\ZZ).$$
In particular, $\cL_0\ZZ^d+\cdots+\cL_k\ZZ^d=\ZZ^d$.
\end{defn}

For $k = 1$, Conjecture~\ref{conj:bukh} was fully resolved in~\cite{CL23} and the lower bound for $|A + \lambda \cdot A|$ when $\lambda$ is of the form $(p/q)^{1/d}$ followed as a corollary. Here we work in the opposite direction, showing that our main result, Theorem~\ref{thm:main}, on sums of algebraic dilates implies a lower bound for sums of certain linear transformations. To state this result requires some further definitions, the first of which is an additional condition beyond irreducibility and coprimeness that a collection of linear transformations must satisfy for our methods to apply. 

\begin{defn}
We say that $\cL_0,\ldots,\cL_k\in \Mat_d(\QQ)$ are \emph{pre-commuting} if there is some $\cP\in \GL_d(\QQ)$ such that $\cP\cL_0,\ldots,\cP\cL_k$ pairwise commute.
\end{defn}

Suppose now that $\cL_0,\ldots,\cL_k\in \Mat_d(\ZZ)$ are non-zero, pre-commuting, irreducible and coprime. Let $G$ be the polynomial
\[G(x_0,\ldots,x_k):=\det(x_0\cL_0+\cdots+x_k\cL_k).\]
If $\cP\in \GL_d(\QQ)$ is such that $\cP\cL_0,\ldots,\cP\cL_k$ are pairwise commuting, then, by a folklore result, these matrices are simultaneously upper-triangularisable over $\CC$, so $G$ factorises into linear terms
$$G(x_0,\ldots,x_k)=\prod_{i=1}^d (a_{0i}x_0+\cdots+a_{ki}x_k).$$
We may then define $H(\cL_0,\ldots,\cL_k)$ to be the quantity
$$H(\cL_0,\ldots,\cL_k):=\prod_{i=1}^d (|a_{0i}|+\cdots+|a_{ki}|).$$ 
With this definition in place, our main result about sums of linear transformations, which is again best possible up to the lower-order term, is as follows. 

\begin{thm} \label{thm:linearsums}
    Suppose that $\cL_0,\ldots,\cL_k\in \Mat_d(\ZZ)$ are pre-commuting, irreducible and coprime. Then 
    $$|\cL_0 A+\cdots+\cL_k A|\geq H(\cL_0,\ldots,\cL_k)|A|-o(|A|)$$
    for all finite subsets $A$ of $\ZZ^d$.
\end{thm}

The first step in establishing our results is to prove a continuous analogue of this statement, a sharpening of the estimate coming from the Brunn--Minkowski inequality standing in the same relation to our results as that inequality does to Conjecture~\ref{conj:bukh}. For $k = 1$, such a continuous result 
was proved by Krachun and Petrov~\cite{KP23}. We extend it here to sums of several pre-commuting linear transformations. Our main contribution, Theorem~\ref{thm:main}, whose proof occupies the bulk of this paper, 
says that a discrete version of this continuous estimate holds for sums of pre-commuting linear transformations corresponding to sums of algebraic dilates. The proof of Theorem~\ref{thm:linearsums} then involves showing that this seemingly special case really encapsulates all pre-commuting families.

Before moving on to a more in-depth discussion of our proof and what is novel about it, we note that for $k = 1$ the condition that the matrices $\cL_0$ and $\cL_1$ be pre-commuting is always true, since the irreducibility condition implies that $\cL_0$ and $\cL_1$ are invertible, so we may take $\cP = \cL_0^{-1}$. As such, we have the following corollary of Theorem~\ref{thm:linearsums}, resolving another conjecture of Krachun and Petrov~\cite{KP23} (see also~\cite{CL23}) and, unlike the $k=1$ case of Conjecture~\ref{conj:bukh} verified in~\cite{CL23}, best possible up to the lower-order term for all $\cL_0$ and $\cL_1$.

\begin{cor}
    Suppose that $\cL_0,\cL_1\in \Mat_d(\ZZ)$ are irreducible and coprime. Then
    $$|\cL_0 A+\cL_1 A|\geq H(\cL_0,\cL_1)|A|-o(|A|)$$
     for all finite subsets $A$ of $\ZZ^d$.
\end{cor}

\subsection{A sketch of the proof}

Our overall strategy is similar to that used by Krachun and Petrov~\cite{KP24} to treat the case of algebraic integers. After rephrasing the problem in terms of sums of linear transformations, their approach can be summarised as having the following three steps:
\begin{enumerate}
    \item Establishing a continuous version of the required estimate.
    \item Reducing to the case where $A$ is a dense subset of a box.
    \item Representing the discrete set $A$ by a continuous set $\overline{A}$.
\end{enumerate}
The idea is that Step 2 guarantees that the $\overline{A}$ obtained in Step 3 is well-behaved. One can then apply the continuous variant from Step 1 to $\overline{A}$ and the required result in the discrete world follows. However, despite having the same general outline, our proof is significantly more complex. We now discuss each step in turn, though we refer the reader to the relevant sections for more details and precise definitions.

Step 1 is the part which is most similar to that of Krachun and Petrov. After lifting the problem to one about sums of certain pre-commuting linear transformations, we prove a tight lower bound for the continuous analogue of this problem by partitioning the underlying space into eigenspaces and then symmetrising our set along these eigenspaces.

For their Step 2, Krachun and Petrov make use of Freiman's theorem on sets of small doubling, one version of which says that any finite set of reals $A$ with $|A + A| \le C |A|$ is contained in a small generalised arithmetic progression (or GAP, for brevity). For our result, we need to prove a novel variant of Freiman's theorem for sets $A$ with a small sum of dilates, that is, with 
$$|A+\lambda_1\cdot A+\cdots+\lambda_k\cdot A|\leq C|A|.$$
It is not hard to show that any such $A$ also has small doubling, so, by the usual version of Freiman's theorem, it must be contained in a small GAP. However, a GAP does not necessarily have a small sum of dilates. Our variant of Freiman's theorem, stated below, says instead that $A$ is contained in what we call an $\cO_K$-GAP, which shares with $A$ the property that it has a small sum of dilates.

\begin{thm}
    For every $C>0$ and $p \in \NN$, there are constants $n$ and $F$ such that for any $A\subset K$ satisfying
    $$|A+\lambda_1\cdot A+\cdots+\lambda_k\cdot A|\leq C|A|,$$
    there exists a $p$-proper $\cO_K$-GAP $P\subset K$ containing $A$ of dimension at most $n$ and size at most $F|A|$.
\end{thm}

To prove this result, we first need to extend several results in additive geometry, a term we borrow from Tao and Vu~\cite[Chapter 3]{TV06}, to the ring of integers $\cO_K$. Once the theorem is in place, we can map $A$ to a dense subset of the box $[0,N)^d$ via a Freiman isomorphism of the surrounding $\cO_K$-GAP, reducing the problem, as promised, to the case of a dense set.

Step 3 is the main and most difficult step. To say something about it, we first describe the method used by Krachun and Petrov~\cite{KP24} to estimate $|A+\lambda\cdot A|$ when $\lambda$ is an algebraic integer. As indicated earlier, they viewed this problem in terms of estimating the size of $|A+\cL A|$ where $A$ is a dense subset of the box $[N]^d$ and $\cL\in \Mat_d(\ZZ)$ is a linear transformation corresponding to multiplication by $\lambda$ and we will discuss it in these terms here. 

A naive way of representing $A$ by a continuous set $\overline{A}$ is to divide the box $[N]^d$ into small cubes and set $\overline{A}\subset \RR^d$ to be the union of the cubes which intersect $A$. However, this is not a good representation, since the volume of $\overline{A}$ can be very different from $|A|$. Indeed, if $A$ consists of all the points in $[N]^d$ with even coordinates, its representation $\overline{A}$ would be the same as if $A$ contained all the points of $[N]^d$.

Krachun and Petrov's solution is to introduce a new dimension to encode the ``local density'' of $A$ at a point $x$, which is, roughly speaking, the relative density of $A$ within a small box containing $x$. More precisely, their continuous representation is a (compact) set $\overline{A}\subset \RR^{d+1}$, which can be seen as having a base in $\RR^d$ resembling $A$, as described in the naive way above, and fibres in $\RR$, with 
the fibre at the point $x\in \RR^d$ being the interval $[0,r]$, where $r$ is the local density of $A$ at $x$. 
In particular, the volume of $\overline{A}$ matches the size $|A|$. The key to their approach is the following simple observation.

\begin{obs} \label{obs:localdens}
    The local density of $B=A+\cL A$ at $x+\cL y$ is at least the local density of $A$ at $x$.
\end{obs}
If $\overline{B}$ is the continuous representation of $B$, then this observation is equivalent to saying that $\overline{B}$ contains $\overline{A}+\cL'(\overline{A})$, where $\cL':\RR^{d+1}\to \RR^{d+1}$ is given by $\cL'(x,y)=(\cL x, 0)$ for $x\in \RR^d$ and $y\in \RR$. Therefore, $\Vol(\overline{B})\geq \Vol(\overline{A}+\cL'(\overline{A}))$. One can then apply the continuous version of sums of dilates to obtain a tight lower bound for $\Vol(\overline{A}+\cL'(\overline{A}))$ in terms of $\Vol(\overline{A})$ and translate the result back to the discrete world using the fact that $\Vol(\overline{A})=|A|$.

The problem with extending this approach to general algebraic $\lambda$ is that Observation~\ref{obs:localdens} is too weak. Indeed, if $\lambda$ is not integral, estimating $|A+\lambda\cdot A|$ is equivalent to estimating $|\cL_1 A+\cL_2 A|$ for some $\cL_1,\cL_2\in \Mat_d(\ZZ)$ and $A$ a dense subset of $[N]^d$. The analogue of the observation in this situation is that the local density of $\cL_1 A+\cL_2 A$ at $\cL_1 x+\cL_2 y$ is at least $\frac{1}{|\det \cL_1|}$ times the local density of $A$ at $x$. However, this is not tight, since if $A$ contains all the lattice points in some convex region, then the local density of $A$ is $1$ uniformly and, by coprimeness, we also expect the local density of $\cL_1 A+\cL_2 A$ to be $1$ uniformly. But the observation only guarantees that the local density of $\cL_1 A+\cL_2 A$ is at least $\frac{1}{|\det \cL_1|}$, which is less than 1 if $\lambda$ is not integral.

Since the local structure of $\cL_1A+\cL_2A$ at $\cL_1x+\cL_2y$ depends on the local structure of $A$ at $x$ and $y$, which can look completely different, we need to consider asymmetric sums $\cL_1A_1+\cL_2A_2$. As an example, consider the periodic sets $A_1,A_2\subset \ZZ$ given by $A_1=\{0,3\}+6\ZZ$ and $A_2=\{0,4\}+6\ZZ$ and the sums $2\cdot A_1+3\cdot A_2$ and $2\cdot A_2+3\cdot A_1$ (for technical reasons we work with periodic sets when defining our notion of local density). Both $A_1,A_2$ have density $1/3$, whereas $2\cdot A_1+3\cdot A_2=6\ZZ$ has density $1/6$ and $2\cdot A_2+3\cdot A_1=\{0,2,3,5\}+6\ZZ$ has density $2/3$. Notice that although $2\cdot A_1+3\cdot A_2$ is less dense than $A_1$ and $A_2$, the swapped sum $2\cdot A_2+3\cdot A_1$ is more dense. In fact, this observation holds more generally.

\begin{obs} \label{obs:localdens2}
    Suppose $A_1$ has local density $\sigma_1$ at $x$ and $A_2$ has local density $\sigma_2$ at $y$. If $\cL_1A_1+\cL_2A_2$ has local density $\eta$ at $\cL_1x+\cL_2y$ and $\cL_1A_2+\cL_2A_1$ has local density $\eta'$ at $\cL_1y+\cL_2x$, then $\eta\eta'\geq \sigma_1\sigma_2$.
\end{obs}

This somewhat resolves the previous issue: although $B=\cL_1A+\cL_2A$ can be locally less dense than $A$ in some places, it must be more dense in others. To justify this observation, consider the sets $A_1,A_2$ as before. Draw the elements of $\ZZ/6\ZZ$ as a $2\times 3$ grid according to their residues modulo 2 and 3 and scale the grid to be a square of side length 1. Denote by $\LD(A_1)$ the union of the cells representing the residues modulo 6 contained in $A_1$ and similarly for $\LD(A_2)$ (see Figure~\ref{fig:mod6ld}). Note that the density of $A_i$ is equal to the volume of $\LD(A_i)$. Then do the same for $2\cdot A_1+3\cdot A_2$ and $2\cdot A_2+3\cdot A_1$ (see Figure~\ref{fig:mod6ldsum}).

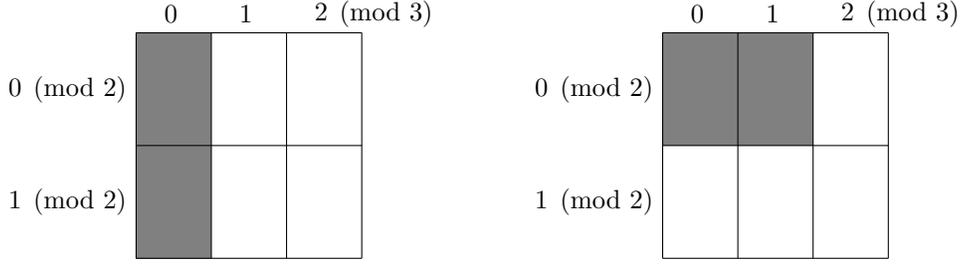
\begin{figure}
    \centering
    \begin{tikzpicture}
        \foreach \x/\y in {0/0, 0/1} {
            \fill[gray] (\x,1.5*\y) rectangle (\x+1,1.5*\y+1.5);
        }
        
        \foreach \i in {0,1,2,3} {
            \draw (\i,0) -- (\i,3);
        }
        \foreach \i in {0,1.5,3} {
            \draw (0,\i) -- (3,\i);
        }
        \node[anchor=east] at (0,0.75) {$1\pmod{2}$};
        \node[anchor=east] at (0,2.25) {$0\pmod{2}$};
        \node[anchor=west] at (0.25,3.25) {$0$};
        \node[anchor=west] at (1.25,3.25) {$1$};
        \node[anchor=west] at (2.25,3.25) {$2\pmod{3}$};

        \pgfmathsetmacro{\xoffset}{7}
        \foreach \x/\y in {0/1, 1/1} {
            \fill[gray] (\xoffset+\x,1.5*\y) rectangle (\xoffset+\x+1,1.5*\y+1.5);
        }
        
        \foreach \i in {0,1,2,3} {
            \draw (\xoffset+\i,0) -- (\xoffset+\i,3);
        }
        \foreach \i in {0,1.5,3} {
            \draw (\xoffset,\i) -- (\xoffset+3,\i);
        }
        \node[anchor=east] at (\xoffset+0,0.75) {$1\pmod{2}$};
        \node[anchor=east] at (\xoffset+0,2.25) {$0\pmod{2}$};
        \node[anchor=west] at (\xoffset+0.25,3.25) {$0$};
        \node[anchor=west] at (\xoffset+1.25,3.25) {$1$};
        \node[anchor=west] at (\xoffset+2.25,3.25) {$2\pmod{3}$};
        
    \end{tikzpicture}
    \caption{The shaded regions are $\LD(A_1)$ and $\LD(A_2)$.}
    \label{fig:mod6ld}
\end{figure}

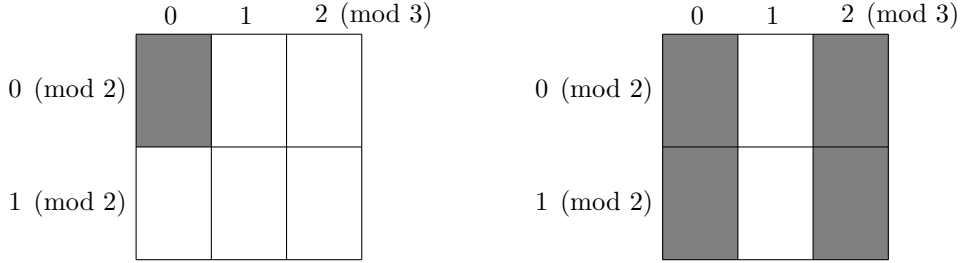
\begin{figure}
    \centering
    \begin{tikzpicture}
        \foreach \x/\y in {0/1} {
            \fill[gray] (\x,1.5*\y) rectangle (\x+1,1.5*\y+1.5);
        }
        
        \foreach \i in {0,1,2,3} {
            \draw (\i,0) -- (\i,3);
        }
        \foreach \i in {0,1.5,3} {
            \draw (0,\i) -- (3,\i);
        }
        \node[anchor=east] at (0,0.75) {$1\pmod{2}$};
        \node[anchor=east] at (0,2.25) {$0\pmod{2}$};
        \node[anchor=west] at (0.25,3.25) {$0$};
        \node[anchor=west] at (1.25,3.25) {$1$};
        \node[anchor=west] at (2.25,3.25) {$2\pmod{3}$};

        \pgfmathsetmacro{\xoffset}{7}
        \foreach \x/\y in {0/0, 0/1, 2/0, 2/1} {
            \fill[gray] (\xoffset+\x,1.5*\y) rectangle (\xoffset+\x+1,1.5*\y+1.5);
        }
        
        \foreach \i in {0,1,2,3} {
            \draw (\xoffset+\i,0) -- (\xoffset+\i,3);
        }
        \foreach \i in {0,1.5,3} {
            \draw (\xoffset,\i) -- (\xoffset+3,\i);
        }
        \node[anchor=east] at (\xoffset+0,0.75) {$1\pmod{2}$};
        \node[anchor=east] at (\xoffset+0,2.25) {$0\pmod{2}$};
        \node[anchor=west] at (\xoffset+0.25,3.25) {$0$};
        \node[anchor=west] at (\xoffset+1.25,3.25) {$1$};
        \node[anchor=west] at (\xoffset+2.25,3.25) {$2\pmod{3}$};
        
    \end{tikzpicture}
    \caption{The shaded regions are $\LD(2\cdot A_1+3\cdot A_2)$ and $\LD(2\cdot A_2+3\cdot A_1)$.}
    \label{fig:mod6ldsum}
\end{figure}

Writing $\pi_1$ and $\pi_2$ for the projections onto the mod $3$ and mod $2$ axes, respectively, we see that $\LD(2\cdot A_1+3\cdot A_2)$ is a $\pi_1(\LD(A_1))\times \pi_2(\LD(A_2))$ rectangle, while, if we allow a permutation of the columns, $\LD(2\cdot A_2+3\cdot A_1)$ is a $\pi_1(\LD(A_2))\times \pi_2(\LD(A_1))$ rectangle. This justifies Observation~\ref{obs:localdens2}, that $\Vol(\LD(2\cdot A_1+3\cdot A_2))\cdot \Vol(\LD(2\cdot A_2+3\cdot A_1))\geq \Vol(\LD(A_1))\cdot \Vol(\LD(A_2))$, in this case. We also note that $|\pi_1(\LD(A_i))|$ is preserved under $\cL_1 = 2 \times$, while $|\pi_2(\LD(A_i))|$ is preserved under $\cL_2 = 3 \times$.

The set $\LD(A)$, which records how $A$ is arranged relative to certain lattices, is roughly what we call the ``lattice density''. More precisely, for a (periodic) set $A\subseteq \ZZ^d$ and a flag of lattices $\cF=\set{L_0\subseteq L_1\subseteq\cdots\subseteq L_k}$, the lattice density $\LD(A;\cF)$ is a compact downset in $[0,1]^{k+1}$ which encodes information about the density of $A$ relative to the lattices $L_l$. Our continuous representation $\overline{A}$ is then a compact subset of $\RR^{d+k+1}$ with a base in $\RR^d$ resembling $A$ and fibres in $\RR^{k+1}$ equal to the local lattice density at each point of $A$. 

Once again, estimating $|A+\lambda_1\cdot A+\cdots+\lambda_k\cdot A|$ is equivalent to estimating $|\cL_0 A+\cdots+\cL_k A|$ for some $\cL_0,\ldots,\cL_k\in \Mat_d(\ZZ)$. The key now, proved through a refinement argument, is that one can find two flags $\cF,\cG$ such that $\pi_{i+1}(\LD(\cL_i A;\cG))\approx \pi_{i+1}(\LD(A;\cF))$ for each $i=0,\ldots,k$. 
In turn, this allows us to show that if $A_0,\ldots,A_k \subset A$ are (periodic) sets, then $\LD(\cL_0 A_0+\cdots+\cL_k A_k;\cG)$ roughly contains the cuboid with side lengths 
\[|\pi_1(\LD(A_0;\cF))|, \ldots, |\pi_{k+1}(\LD(A_k;\cF))|.\]
By making appropriate choices of $A_0, \dots, A_k \subset A$ locally, 
this implies that if $B=\cL_0 A+\cdots+\cL_k A$, then $\overline{B}\subset \RR^{d+k+1}$ contains the sumset $\cL_0'\overline{A}+\cdots+\cL_k'\overline{A}$, where $\cL_i':\RR^{d+k+1}\to \RR^{d+k+1}$ is given by $\cL_i'(x,y)=(\cL_i x, \pi_{i+1}(y))$ for $x\in \RR^d$ and $y\in \RR^{k+1}$. In line with the application of Observation~\ref{obs:localdens} in the algebraic integer case, we can then apply our continuous estimate to $\cL_0'\overline{A}+\cdots+\cL_k'\overline{A}$, which in turn yields the desired result in the discrete case.

\subsection{Notation}

Throughout the paper, we will use the following notation:
\begin{itemize}
    \item $\lambda_0, \lambda_1,\ldots,\lambda_k$ are algebraic numbers with $\lambda_0=1$.
    \item $K:=\QQ(\lambda_1,\ldots,\lambda_k)$ is the number field generated by $\lambda_1,\ldots,\lambda_k$.
    \item The degree of $K$ over $\QQ$ is $d:=\deg(K/\QQ)$, so $K\cong \QQ^d$. 
    \item The ring of integers over $K$ is denoted by $\cO_K$, so $\cO_K\cong \ZZ^d$.
    \item We write $K_\RR:=K\otimes_\QQ \RR=\cO_K\otimes_\ZZ \RR \cong \RR^d$ and $K_\CC:=K\otimes_\QQ \CC\cong \CC^d$. 
    \item We will generally use $i$ to index $1,\ldots,d$, $j$ to index $1,\ldots,n$ and $l$ to index $1,\ldots,k$ (possibly starting at 0). However, this is not strict and the usage can depend on context.
\end{itemize}

\subsection{Organisation of the paper}

The remainder of the paper is laid out as follows. We begin, in Section~\ref{sec:mapZd}, by formally describing how to rephrase the problem of estimating sums of algebraic dilates in terms of estimating sums of linear transformations, in particular observing that the linear transformations corresponding to taking various algebraic dilates are simultaneously diagonalisable. In Section~\ref{sec:cts}, we prove the continuous analogue of our estimate, extending a result of Krachun and Petrov to sums of arbitrarily many simultaneously diagonalisable linear transformations of compact sets. It is also here, in Section~\ref{sec:low}, that we show that Theorem~\ref{thm:main} is best possible. We extend several results from additive geometry, including Minkowski's second theorem and John's theorem, to rings of integers in Section~\ref{sec:algaddgeom}, culminating in our Freiman-type structure theorem for small sums of dilates. We then use this result in Section~\ref{sec:reduction} to reduce the proof of Theorem~\ref{thm:main} to the special case where $A$ is a dense subset of a box. In Section~\ref{sec:ld}, we introduce lattice densities and prove some of their basic properties, while in Section~\ref{sec:families} we introduce two key families of flags of lattices and establish some relations between lattice densities taken relative to these flags. Using these results, we conclude the proof of Theorem~\ref{thm:main} in Section~\ref{sec:dense} by verifying it in the dense case. In Section~\ref{sec:linearsums}, we discuss the general problem of estimating sums of pre-commuting linear transformations, showing how it reduces to Theorem~\ref{thm:main}. Finally, in Section~\ref{sec:conc}, we point towards some further possible research directions.

\subsection{Acknowledgements}

The authors thank Deepesh Singhal for helpful discussions on the algebraic number theory aspects of the paper. We would also like to note that this work was supported by NSF grant DMS-1928930 while the authors were in residence at the Simons--Laufer Mathematical Sciences Institute in Berkeley, California during the Spring 2025 semester on Extremal Combinatorics.

\section{Mapping to $\mathbb{Z}^d$} \label{sec:mapZd}

In this section, we show how the problem of estimating sums of algebraic dilates can be recast in terms of estimating sums of linear transformations. 
Our first lemma, generalising~\cite[Lemma~3.1]{KP24}, will allow us to assume that $A$ is a subset of $K$.

\begin{lem} \label{lem:aink}
Suppose that $\lambda_1,\ldots,\lambda_k\in\CC$ and $A\subset\CC$ is finite. Then there exists a finite set $B\subset K=\QQ(\lambda_1,\ldots,\lambda_k)$ such that $|B| = |A|$ and $|B + \lambda_1 \cdot B+\cdots+\lambda_k \cdot B| \leq |A + \lambda_1 \cdot A+\cdots+\lambda_k \cdot A|$.
\end{lem}
\begin{proof}
    Let $L$ be the field extension of $K$ generated by $A$. Pick any $K$-linear map $f:L\to K$ which is injective on $A$. Such a map exists since $A$ is finite. Set $B=f(A)$. Then $|B|=|A|$ and, for any $a_0,\ldots,a_k\in A$, 
    $$f(a_0+\lambda_1 a_1+\cdots+\lambda_k a_k)=f(a_0)+\lambda_1 f(a_1)+\cdots+\lambda_k f(a_k).$$
    Hence, $|B + \lambda_1 \cdot B+\cdots+\lambda_k \cdot B|=|f(A + \lambda_1 \cdot A+\cdots+\lambda_k \cdot A)|\leq |A + \lambda_1 \cdot A+\cdots+\lambda_k \cdot A|$.
\end{proof}

In light of this result, we will henceforth assume that $A\subset K$. For any $a\in K$, there exists a positive integer $n$ such that $na\in \cO_K$. In fact, this is true for any fractional ideal $\cI \subseteq \cO_K$ -- for any $a\in K$, there exists a positive integer $n$ such that $na\in \cI$. Thus, since $A$ is finite, by rescaling $A$ to $n\cdot A$ for an appropriately large $n$, we may assume that $A\subset \cI$ if we wish to without any loss of generality.

To pass to linear transformations, we fix a $\ZZ$-basis $e_1=1,e_2,\ldots,e_d$ of $\cO_K$ and let $\Phi:\cO_K\to \ZZ^d$ be the isomorphism mapping the $e_i$ to the standard basis of $\ZZ^d$. 
This map extends linearly to an isomorphism $\Phi:K\to \QQ^d$. Under this isomorphism, multiplication by $\lambda_l$ corresponds to the linear map $\cM_l\in \Mat_d(\QQ)$ defined by 
\[\cM_l(x)=\Phi(\lambda_l\cdot \Phi^{-1}(x)).\] 
The problem of estimating $|A+\lambda_1\cdot A+\cdots+\lambda_k\cdot A|$ for $A\subset K$ is then equivalent to estimating $|A+\cM_1 A+\cdots+\cM_k A|$ for $A\subset \QQ^d$.

One further step allows us to convert the problem into one about sums of linear transformations with \emph{integer} entries. 
Recall, from the introduction, that the denominator ideal of $\lambda_1,\ldots,\lambda_k$ is the non-zero ideal $\frakD = \cO_K\cap \lambda_1^{-1}\cO_K\cap\cdots \cap\lambda_k^{-1}\cO_K$ with the property that $\lambda_l \frakD\subseteq \cO_K$ for all $l=0,\ldots,k$.
If we fix an isomorphism $\Phi':\frakD\to \ZZ^d$, then multiplication of the elements of $\frakD$ by $\lambda_l$ corresponds to the linear map $\cL_l:\ZZ^d\to \ZZ^d$ defined by 
\[\cL_l(x)=\Phi(\lambda_l\cdot \Phi'^{-1}(x)).\]
By rescaling, we may assume that $A\subset \frakD$, 
so that Theorem~\ref{thm:main} becomes equivalent to the following result, whose proof will now be our principal goal.

\begin{thm} \label{thm:main2}
    For all finite subsets $A$ of $\ZZ^d$,
    $$|\cL_0 A+\cdots+\cL_k A|\geq H(\lambda_1,\ldots,\lambda_k)|A|-o(|A|).$$
\end{thm}

The next lemma determines all the (simultaneous) eigenvalues of the $\lambda_l$, when they are viewed as $\QQ$-linear maps on $K$. In the statement and proof, we will use the fact that there are exactly $d$ different complex embeddings (that is, injective field homomorphisms) of $K$ in $\CC$, which we denote by  $\sigma_1,\ldots,\sigma_d$ with $\sigma_1$ the identity. 

\begin{lem} \label{lem:diag}
    Viewing $K\cong \QQ^d$, multiplication by $\lambda_l$ induces a $\QQ$-linear map $\cM_l:\QQ^d\to\QQ^d$. Then the maps $\cM_0,\ldots,\cM_k$ are simultaneously diagonalisable over $\CC$ into the diagonal matrices $\cD_0,\ldots,\cD_k$, where $\cD_l$ has diagonal entries $(\sigma_1(\lambda_l),\ldots,\sigma_d(\lambda_l))$ for $l=0,\ldots,k$.
\end{lem}
\begin{proof}
    Let $K_\CC=K\otimes_\QQ \CC$ and define $\sigma:K_\CC \to \CC^d$ to be the $\CC$-linear map defined by $\sigma(\alpha\otimes c)=(c\sigma_1(\alpha),\ldots,c\sigma_d(\alpha))$. We claim that $\sigma$ is an isomorphism. Indeed, let $\alpha\in K$ be a generator of $K$, i.e., $K=\QQ(\alpha)$. Then $(1,\alpha,\ldots,\alpha^{d-1})$ is a $\QQ$-basis for $K$ and $\sigma_1(\alpha),\ldots,\sigma_d(\alpha)$ are all distinct. Under this basis, which is also a basis for $K_\CC$, $\sigma$ is represented by the matrix
    \[
    \begin{pmatrix}
        1 & \sigma_1(\alpha) & \sigma_1(\alpha)^2 & \cdots & \sigma_1(\alpha)^{d-1}\\
        1 & \sigma_2(\alpha) & \sigma_2(\alpha)^2 & \cdots & \sigma_2(\alpha)^{d-1}\\
        \vdots & \vdots & \vdots & \ddots & \vdots\\
        1 & \sigma_d(\alpha) & \sigma_d(\alpha)^2 & \cdots & \sigma_d(\alpha)^{d-1}
    \end{pmatrix},
    \]
    which is non-singular, since it is a Vandemonde matrix. 
    Let $e_1,\ldots,e_d\in \CC^d$ be the standard basis of $\CC^d$ and $v_i=\sigma^{-1}(e_i)$. Then $v_1,\ldots,v_d$ form a basis for $K_{\CC}$. We claim that, in this basis, $\cM_l$ diagonalises into the desired form. It suffices to check that $\cM_l(v_i)=\sigma_i(\lambda_l)v_i$.

    Let $x_1,\ldots,x_d\in K$ be a $\QQ$-basis for $K$. Then $v_i$ can be written in the form $v_i=x_1\otimes c_{i1}+\cdots+x_k\otimes c_{ik}$ for some $c_{il}\in \CC$, so that $\sigma(v_i)=e_i$ says that 
    $$\sum_{l=1}^k c_{il}\sigma_j(x_l)=\delta_{ij}.$$
    But then
    \begin{align*}
        \sigma_j(\cM_l(v_i)) &= \sigma_j\paren{\cM_l\paren{\sum_{m=1}^k x_m\otimes c_{im}}}
        = \sigma_j\paren{\sum_{m=1}^k (\lambda_l x_m)\otimes c_{im}}\\
        &= \sum_{m=1}^k c_{im}\sigma_j(\lambda_l x_m)
        = \sigma_j(\lambda_l)\sum_{m=1}^k c_{im}\sigma_j(x_m)\\
        &= \sigma_j(\lambda_l)\delta_{ij}=\sigma_i(\lambda_l)\delta_{ij}.
    \end{align*}
    It follows that $\cM_l(v_i)=\sigma_i(\lambda_l)v_i$, as required.
\end{proof}

\section{The continuous version} \label{sec:cts}

We now come to the first part of our argument, which is to extend an estimate of Krachun and Petrov~\cite[Theorem~2]{KP24} on sums of linear transformations of compact sets to more than two variables. We will need to assume that the linear transformations are simultaneously diagonalisable. But, as we have seen in Lemma~\ref{lem:diag} above, this is exactly the situation we are concerned with.

Throughout this section, we will fix an identification $K_\RR\cong \RR^d$ and take $\mu$ to be the Lebesgue measure on $\RR^d$ and, hence, on $K_\RR$. Our main result may then be stated as follows.

\begin{thm} \label{thm:cts}
    Suppose $\cL_1,\ldots,\cL_k\in \Mat_d(\RR)$ are simultaneously diagonalisable over $\CC$ into the diagonal matrices $\cD_1,\ldots,\cD_k$, where $\cD_l=\diag(\lambda_{l1},\ldots,\lambda_{ld})$ with each $\lambda_{li}\in\CC$. Then, for any compact $A\subset \RR^d$,
    $$\mu(\cL_1 A+\cL_2 A+\dots+\cL_k A)\geq \paren{\prod_{i=1}^d \sum_{l=1}^k |\lambda_{li}|}\mu(A).$$
    Moreover, equality holds for some $A$ with $\mu(A)>0$.
\end{thm}

\begin{proof}
    Let $\Lambda=\set{(\lambda_{1i},\lambda_{2i},\ldots,\lambda_{ki})}_{i=1}^d$. Since complex conjugation preserves each $\cL_l$, it permutes the elements of $\Lambda$. Thus, we can split $\Lambda$ into two parts $\Lambda_1$ and $\Lambda_2$, where $\Lambda_1\subset \RR^k$ consists of those tuples fixed by conjugation and $\Lambda_2$ consists of conjugate pairs of tuples. Then we may decompose $\RR^d$ into the eigenspaces
    $$\RR^d=\bigoplus_{\lambda\in \Lambda_1} E_\lambda\oplus \bigoplus_{(\lambda,\overline{\lambda})\in \Lambda_2} E_{\lambda,\overline{\lambda}},$$
    where each $E_{\lambda}$ is 1-dimensional and each $E_{\lambda,\overline{\lambda}}$ is 2-dimensional. For $\lambda=(\lambda_1,\ldots,\lambda_k)\in \Lambda_1$, each $\cL_l$ acts on $E_{\lambda}$ by $\lambda_l$. For $(\lambda,\overline{\lambda})\in \Lambda_2$, each $\cL_l$ acts on $E_{\lambda,\overline{\lambda}}$ by $|\lambda_l|R_{\arg(\lambda_l)}$, where $R_{\theta}$ is the rotation map on $\RR^2$ by $\theta$.

    We prove the theorem in the following more general form. Suppose we have a decomposition
    $$\RR^d=\bigoplus_{j=1}^n E_j,$$
    where $\dim E_j=d_j$ and $\cL_l$ acts on $E_j$ by $r_{lj}P_{lj}$, where $r_{lj}\geq 0$ and $P_{lj}$ is an orthogonal matrix acting on $E_j$. In other words, for any vector $v\in \RR^d$, if we decompose it into $v=v_1+\cdots+v_n$ with $v_j\in E_j$ for all $1 \le j \le n$, then $\cL_l v=r_{l1}P_{l1}v_1+\cdots+r_{ln}P_{ln}v_n$. 
    We will show that
    $$\mu(\cL_1 A+\cL_2 A+\cdots+\cL_k A)\geq \paren{\prod_{j=1}^n \paren{\sum_{l=1}^k r_{lj}}^{d_j}}\mu(A).$$
    
    We perform Steiner symmetrisation, a continuous analogue of compression introduced by Steiner in his classical work on the isoperimetric problem,    
    along each of the eigenspaces $E_j$ as follows. Write $\RR^d=E_j\oplus E$, where $E$ is the direct sum of the remaining spaces. Let $\pi_1:\RR^d\to E_j$ and $\pi_2:\RR^d\to E$ be the projections onto $E_j$ and $E$, respectively. For a compact $A\subset \RR^d$ and $x\in E$, write $A_x:=\pi_1(\pi_2^{-1}(x))\subset E_j$ for the fibre of $A$ at $x$. Then $\mu(A)=\int_x \mu(A_x) d\mu(x)$. The \emph{Steiner symmetrisation} of $A$ along $E_j$ is the set $S_j(A)\subset \RR^d$ with the same support as $A$ on $E$ and such that, for each $x\in \pi_2(A)$, $S_j(A)_x$ is the closed ball centered at 0 with the same volume as $A_x$. 
    
    \begin{claim} \label{claim:steiner}
        The Steiner symmetrisation has the following properties:
        \begin{enumerate}
            \item $\mu(S_j(A))=\mu(A)$.
            \item $S_j(A)$ is invariant under any orthogonal transformation of $E_j$.
            \item $S_j(\cL_l A)\supseteq \cL_l (S_j(A))$ for all $l$.
            \item $S_j(A)$ is compact.
            \item If $B$ is compact, then $S_j(A+B)\supseteq S_j(A)+S_j(B)$.
            \item If $F\in \GL(E)$ and $F'\in \GL_d(\RR)$ is given by $I_{E_j}\oplus F$ and $F'(A)=A$, then $F'(S_j(A))=S_j(A)$. 
        \end{enumerate}
    \end{claim}
    \begin{proof}
        \begin{enumerate}
            \item This is true since $\mu(S_j(A)_x)=\mu(A_x)$ for all $x\in E$.
            \item This is true since $S_j(A)_x$ is a ball for all $x\in E$.
            \item Let $x\in \pi_2(A)$ and $B=S_j(A)_x$, a ball. Then $\cL_l(S_j(A))=\bigcup_{x\in \pi_2(A)} \cL_l|_{E_j}(B)\oplus \cL_l x$. Note that $\cL_l|_{E_j}(B)$ is also a ball of volume $\mu(\cL_l|_{E_j}(A_x))\leq \mu((\cL_lA)_{\cL_l x})$. 
            Thus, $\cL_l|_{E_j}(B)\oplus \cL_l x\subseteq S_j(\cL_l A)$ and the result follows.
            \item Since $A$ is bounded, so is $S_j(A)$. To show that $S_j(A)$ is closed, it is sufficient to show that for any sequence $x_1,x_2,\ldots\in E$ converging to $x\in E$, we have $\mu(A_x)\geq \limsup_{n} \mu(A_{x_n})$. 
            Since $A$ is closed, $A_x\supseteq \limsup_i A_{x_i}$, so it suffices to show that $\mu(\limsup_i A_{x_i})\geq \limsup_i \mu(A_{x_i})$. But this is true since the $A_{x_i}$ are uniformly bounded.

            \item For $x\in \pi_2(A)$ and $y\in \pi_2(B)$, let $r,r'$ be the radii of the balls $S_j(A)_x$ and $S_j(B)_y$, with volumes $V,V'$. Then $(S_j(A)+S_j(B))_{x+y}$ is a ball of radius $r+r'$, maximised over all $x,y$ with the same fixed sum. But, by the Brunn--Minkowski inequality,
            \begin{align*}
                \mu(S_j(A+B)_{x+y}) &\geq \mu(S_j(A)_x + S_j(B)_y)\\
                &\geq (\mu(S_j(A)_x)^{1/d_j}+\mu(S_j(B)_y)^{1/d_j})^{d_j}\\
                &= (V^{1/d_j}+V'^{1/d_j})^{d_j}\\
                &= \mu((S_j(A)+S_j(B))_{x+y}).
            \end{align*}
            Thus, $S_j(A+B)_{x+y}\supseteq (S_j(A)+S_j(B))_{x+y}$ and the result follows. 

            \item Let $x\in E$. Since $F'(A)=A$, we have $A_{F(x)}=A_x$. Therefore, $S_j(A)_{F(x)}=S_j(A)_x$, so we have $F'(S_j(A))=S_j(A)$. \qedhere
        \end{enumerate}
    \end{proof}

    Perform Steiner symmetrisation on $A$ successively along $E_1,\ldots,E_n$ to obtain, by Claim~\ref{claim:steiner}(4), the compact set $B=S_1(S_2(\cdots S_n(A)\cdots))$. By Claim~\ref{claim:steiner}(1), (5) and (3),
    \begin{align*}
        \mu(\cL_1 A+\cdots +\cL_k A) &= \mu(S_j(\cL_1 A+\cdots +\cL_k A))\\
        &\geq \mu(S_j(\cL_1 A)+\cdots +S_j(\cL_k A))\\
        &\geq \mu(\cL_1(S_j(A))+\cdots +\cL_k(S_j(A))).
    \end{align*}
    Iterating, we see that $\mu(\cL_1 A+\cdots +\cL_k A)\geq \mu(\cL_1 B+\cdots +\cL_k B)$, where we also have $\mu(B)=\mu(A)$. 
    
    Let $\cL_l'$ be the linear map that just scales by $r_{lj}$ on each $E_j$, i.e., $\cL_l'(v_1+\cdots+v_n)=r_{l1}v_1+\cdots+r_{ln}v_n$ for any $v_j\in E_j$. By repeated applications of Claim~\ref{claim:steiner}(2) and (6), we may check that $B$ is rotationally invariant on each $E_j$, so we have $\cL_l' B=\cL_l B$. Thus,
    \begin{align*}
        \mu(\cL_1 B+\cdots +\cL_k B) &= \mu(\cL_1' B+\cdots +\cL_k' B)\\
        &\geq \mu((\cL_1'+\cdots+\cL_k')(B))\\
        &= |\det(\cL_1'+\cdots+\cL_k')|\mu(B)\\
        &= \paren{\prod_{j=1}^l \paren{\sum_{l=1}^k r_{lj}}^{d_j}}\mu(B).
    \end{align*}

    Finally, to see that equality may hold, observe that we can take $A$ to be the product of the unit balls in each $E_j$.
\end{proof}

In particular, this yields the smallest possible value of $\mu(A+\lambda_1\cdot A+\dots+\lambda_k\cdot A)$ in terms of $\mu(A)$. To see this, let $\cM_l\in \Mat_d(\QQ)$ be the matrix representing multiplication by $\lambda_l$ for $l=0,\ldots,k$, as defined in Section~\ref{sec:mapZd}. Then, by Lemma~\ref{lem:diag}, the $\cM_l$ are simultaneously diagonalisable into the diagonal matrices $\cD_l$ with entries $(\sigma_1(\lambda_l),\ldots,\sigma_d(\lambda_l))$, where $\sigma_1,\ldots,\sigma_d$ are all the complex embeddings of $K$. By Theorem~\ref{thm:cts}, we therefore have
\begin{align*}
    \mu(A+\lambda_1\cdot A+\cdots+\lambda_k\cdot A) &= \mu(\cM_0 A+\cdots+\cM_k A)\\
    &\geq \paren{\prod_{i=1}^d (1+|\sigma_i(\lambda_1)|+\cdots+|\sigma_i(\lambda_k)|)}\mu(A).
\end{align*}
Comparing this to our main result, Theorem~\ref{thm:main}, we see that the discrete version differs from the continuous one only in the factor $N_{K/\QQ}(\frakD_{\lambda_1,\ldots,\lambda_k;K})$, which is a measure of the non-integrality of $\lambda_1,\ldots,\lambda_k$. We say more below.

\subsection{Lower bound construction} \label{sec:low}

In this short subsection, we give a lower bound construction for the discrete case, showing that the constant $H(\lambda_1,\ldots,\lambda_k)$ in Theorem~\ref{thm:main} is best possible. In brief, the construction is a discretised version of the equality case in Theorem~\ref{thm:cts}.

\begin{prop}
    Let $\lambda_1,\ldots,\lambda_k\in K=\QQ[\lambda_1,\ldots,\lambda_k]$ be algebraic numbers. Then there exist arbitrarily large $A\subset \CC$ such that
    \[|A+\lambda_1\cdot A+\cdots+\lambda_k\cdot A|\leq H(\lambda_1,\ldots,\lambda_k)|A|+O(|A|^{\frac{d-1}{d}}),\]
    where $d=\deg(K/\QQ)$.
\end{prop}
\begin{proof}
    Let $\sigma_1,\ldots,\sigma_d:K\to \CC$ be the complex embeddings of $K$ and set $\frakD=\frakD_{\lambda_1,\ldots,\lambda_k;K}$. Viewing multiplication by $\lambda_l$ as a $\QQ$-linear map $\cM_l:K\to K$ for each $l$, 
    take $A'\subset K_\RR$ satisfying the equality case in Theorem~\ref{thm:cts} with $\mu(A')=1$. Then $\mu(A'+\lambda_1\cdot A'+\cdots+\lambda_k\cdot A')= \paren{\prod_{i=1}^d (1+|\sigma_i(\lambda_1)|+\cdots+|\sigma_i(\lambda_k)|)}\mu(A')$.

    Let $n$ be an arbitrarily large positive integer and let $A=nA'\cap \frakD$, so that 
    $$|A| = \mu(nA')/\Vol(K_\RR/\frakD) + O(n^{d - 1}) = n^d/\Vol(K_\RR/\frakD) + O(n^{d-1}).$$ 
    On the other hand, for each $l$, $\lambda_l\cdot A\subset \lambda_l\cdot \frakD\subseteq \cO_K$, so we have
    \[A+\lambda_1\cdot A+\cdots+\lambda_k\cdot A\subseteq n(A'+\lambda_1\cdot A'+\cdots+\lambda_k\cdot A')\cap \cO_K.\]
    Therefore,
    \begin{align*}
        |A+\lambda_1\cdot A+\cdots+\lambda_k\cdot A| &\leq \mu(n(A'+\lambda_1\cdot A'+\cdots+\lambda_k\cdot A'))/\Vol(K_\RR/\cO_K)+O(n^{d-1})\\
        &=n^d\paren{\prod_{i=1}^d (1+|\sigma_i(\lambda_1)|+\cdots+|\sigma_i(\lambda_k)|)}/\Vol(K_\RR/\cO_K)+O(n^{d-1}).
    \end{align*}
    Since $\Vol(K_\RR/\frakD)/\Vol(K_\RR/\cO_K)=[\cO_K:\frakD]=N_{K/\QQ}(\frakD)$, we obtain that
    \begin{align*}
        |A+\lambda_1\cdot A+\cdots+\lambda_k\cdot A| &\leq N_{K/\QQ}(\frakD)\paren{\prod_{i=1}^d (1+|\sigma_i(\lambda_1)|+\cdots+|\sigma_i(\lambda_k)|)}|A|+O(n^{d-1})\\
        &= H(\lambda_1,\ldots,\lambda_k)|A|+O(|A|^{\frac{d-1}{d}}). \qedhere
    \end{align*}
\end{proof}

\section{Algebraic additive geometry} \label{sec:algaddgeom}

In this section, we extend several results from additive geometry 
to  rings of integers, culminating in the version of Freiman's theorem for sums of dilates mentioned in the introduction. Along the way, we prove several results that may be of independent interest, including versions of Minkowski's second theorem and John's theorem for lattices over rings of integers.

\subsection{A norm on $\cO_K$ and $K_{\RR}$}

In this subsection, we define a norm 
on $\cO_K$ and $K_{\RR}$ and note some of its basic properties (this is not to be confused with the field norm $N_{K/\QQ}(\cdot)$ on $K$, which we also use). Recall from Section~\ref{sec:mapZd} that we have an isomorphism $\Phi:\cO_K\to \ZZ^d$ given by sending a basis $e_1,\ldots,e_d$ of $\cO_K$ to the standard basis. By pulling back $\Phi$, the $\infty$-norm on $\ZZ^d$ defines a norm $\norm{\cdot}$ on $\cO_K$, 
namely, for $l_1,\ldots,l_d\in \ZZ$,
$$\norm{l_1e_1+\cdots+l_de_d}:=\max_i |l_i|.$$
The open ball $B(L)$ of radius $L>0$ under this norm is then given by
$$B(L):=\setcond{l_1e_1+\cdots+l_de_d\in \cO_K}{|l_i|< L \mbox{ for all } i}.$$
$\norm{\cdot}$ extends linearly and continuously to a norm on $K_\RR$, which we also denote by $\norm{\cdot}$. The open ball $B_\RR(R)$ of radius $R>0$ in $K_\RR$ is then
$$B_\RR(R):=\setcond{e_1\otimes r_1+\cdots+e_d\otimes r_d\in K_\RR}{|r_i|<R \mbox{ for all } i}.$$

The following lemma may be seen as defining some constants associated to the norm $\norm{\cdot}$.

\begin{lem} \label{lem:consts}
    There exist constants $C_1,C_2,C_3\in \NN$ such that the following hold:
    \begin{enumerate}
        \item For all $x,y\in K_\RR$, $\norm{xy}\leq C_1\norm{x}\norm{y}$.
        \item For all $l=0,\ldots,k$, $C_2\lambda_l\in \cO_K$.
        \item For all $l=0,\ldots,k$ and $x\in \cO_K$, $\lambda_lx\in \frac{1}{C_2}\cdot B(C_3\norm{x})$.
    \end{enumerate}
\end{lem}
\begin{proof}
    \begin{enumerate}
        \item Let $M>0$ be the maximum of $\norm{e_ie_j}$ over all pairs $i,j\in [d]$. Now, for any $x=e_1\otimes x_1+\cdots+e_d\otimes x_d$ and $y=e_1\otimes y_1+\cdots+e_d\otimes y_d$ with $x_i,y_i\in \RR$, we have $|x_i|\leq \norm{x}$ and $|y_i|\leq \norm{y}$. Therefore,
        \[
            \norm{xy} = \|\sum_{i,j} e_ie_j\otimes x_iy_i\|
            \leq \sum_{i,j}\norm{e_ie_j\otimes x_iy_i}
            = \sum_{i,j}\norm{e_ie_j}|x_iy_i|
            \leq d^2M\norm{x}\norm{y},
        \]
        so we may pick $C_1=d^2M$.
        
        \item Since $\cO_K$ is of full rank, for any $\lambda\in K$, there is some integer $C>0$ such that $C\lambda\in \cO_K$. Thus, we may pick $C_2$ to be the lowest common multiple of the $C$'s corresponding to each $\lambda_l$.
        
        \item Pick an integer $C_3$ such that $C_3>C_1C_2\max_l \norm{\lambda_l}$. Then we have $C_2\lambda_lx\in \cO_K$ and $\norm{C_2\lambda_lx}\leq C_1C_2\norm{\lambda_l}\norm{x}<C_3\norm{x}$. Therefore, $\lambda_lx\in \frac{1}{C_2}\cdot B(C_3\norm{x})$. \qedhere
    \end{enumerate}
\end{proof}

Throughout the rest of this section, we will use the constants $C_1,C_2,C_3$ as given by this lemma.

\subsection{An algebraic Minkowski's second theorem}

In this subsection, we prove a variant of Minkowski's second theorem for lattices over rings of integers. Before we state this result, let us recall the original theorem of Minkowski. We first need a definition, noting that here a \emph{convex body} is assumed to be convex, open, non-empty and bounded.

\begin{defn}
    Let $\Gamma\subset \RR^n$ be a lattice of rank $m$ and $B\subset \RR^n$ a convex body containing 0. We define the \emph{successive minima} $\ell_j=\ell_j(B,\Gamma)$ of $B$ with respect to $\Gamma$ by
    $$\ell_j:=\inf\setcond{\ell>0}{\ell\cdot B \text{ contains $j$ linearly independent elements of $\Gamma$}}$$
    for each $1\leq j\leq m$.  
    Note that $0<\ell_1\leq\cdots\leq \ell_m<\infty$.
\end{defn}

Minkowski's second theorem (see, for example,~\cite[Theorem 3.30]{TV06}) is then as follows.

\begin{thm}[Minkowski's second theorem] \label{thm:minkowski2}
    Let $\Gamma\subset \RR^n$ be a lattice of full rank and let $B$ be a centrally symmetric convex body in $\RR^n$ with successive minima $0<\ell_1\leq\cdots\leq \ell_n$. Then there exist $n$ linearly independent vectors $v_1,\ldots,v_n\in \Gamma$ with the following properties:
    \begin{itemize}
        \item for each $1\leq j\leq n$, $v_j$ lies in the boundary of $\ell_j\cdot B$, but $\ell_j\cdot B$ itself does not contain any vectors in $\Gamma$ outside the span of $v_1,\ldots,v_{j-1}$;
        \item the octahedron with vertices $\pm v_j$ for $1 \le j \le n$ contains no elements of $\Gamma$ in its interior other than the origin;
        \item 
        one has
        $$\frac{2^n[\Gamma:\ang{v_1,\ldots,v_n}_\ZZ]}{n!}\leq \frac{\ell_1\cdots\ell_n \Vol(B)}{\Vol(\RR^n/\Gamma)}\leq 2^n.$$
    \end{itemize}
\end{thm}

To state our variant of this theorem, we need to first clarify what we mean by a lattice over a ring of integers. 

\begin{defn}
    An \emph{$\cO_K$-lattice} is a lattice $\Gamma$ in $K^n\cong \QQ^{dn}$ that is closed under multiplication by $\cO_K$. That is, for any $v\in \Gamma$ and $a\in \cO_K$, $av\in \Gamma$. Equivalently, $\Gamma$ is a discrete $\cO_K$-submodule of $K^n$. Observe that $\QQ\cdot \Gamma=K\cdot \Gamma$ is a $K$-subspace of $K^n$. The \emph{$\cO_K$-rank} of $\Gamma$ is the dimension $m$ of this subspace. Note that, when viewed as an ordinary lattice, the rank of $\Gamma$ is $md$. 
\end{defn}

For the next definition, we recall, from Section~\ref{sec:mapZd}, that we view $\cO_K$ as having a fixed $\ZZ$-basis $e_1,\ldots,e_d$.

\begin{defn}
    For a real number $r\geq 1$, a subset $B\subseteq K_\RR^n$ is said to be \emph{$r$-thick} if $e_i\cdot B\subseteq r\cdot B$ for all $i\in [d]$.
\end{defn}
For example, by Lemma~\ref{lem:consts}, $\norm{e_i x}\leq C_1\norm{e_i}\norm{x}=C_1\norm{x}$ for all $x\in K_\RR$, so that $B_\RR(L)$ is $C_1$-thick for any $L>0$.

We now redefine successive minima, but with respect to $\cO_K$-lattices.

\begin{defn}
    Let $\Gamma$ be an $\cO_K$-lattice of $\cO_K$-rank $m$ and $B$ a convex body in $K_\RR^n$ containing 0. We define the \emph{successive minima} $\ell_j=\ell_j(B,\Gamma)$ of $B$ with respect to $\Gamma$ by
    $$\ell_j:=\inf\setcond{\ell>0}{\ell\cdot B\text{ contains $j$ $K$-linearly independent elements of $\Gamma$}}$$
    for each $1\leq j\leq m$. 
    Note that we again have $0<\ell_1\leq\cdots\leq \ell_m<\infty$, since $\Gamma$ has $\cO_K$-rank $m$ and so contains $m$ $K$-linearly independent elements of $K^n$.
\end{defn}

We may now state and prove our version of Minkowski's second theorem for $\cO_K$-lattices.

\begin{lem} \label{lem:algminkowski2}
    Let $r\geq 1$ be a real number, let $\Gamma\subset K^n$ be an $\cO_K$-lattice of full rank and let $B$ be an $r$-thick centrally symmetric convex body in $K_\RR^n$ with successive minima $0<\ell_1\leq\cdots\leq \ell_n$. Then there exist $n$ $K$-linearly independent vectors $v_1,\ldots,v_n\in \Gamma$ with the following properties:
    \begin{itemize}
        \item for each $1\leq j\leq n$, $v_j$ lies in the boundary of $\ell_j\cdot B$, but $\ell_j\cdot B$ does not contain any vectors in $\Gamma$ outside the $K$-span of $v_1,\ldots,v_{j-1}$;
        \item the octahedron with vertices $\pm \frac{1}{r}e_iv_j$ for $i\in [d],j\in [n]$ contains no elements of $\Gamma$ in its interior other than the origin;
        \item if $\Gamma'$ is the $\cO_K$-lattice generated by $v_1,\ldots,v_n$, then
        \begin{equation}
            \frac{(2/r)^{nd}[\Gamma:\Gamma']}{(nd)!}\leq \frac{(\ell_1\cdots\ell_n)^d \Vol(B)}{\Vol(K_\RR^n/\Gamma)}\leq 2^{nd}. \label{eqn:algminkowskiineq}
        \end{equation}
    \end{itemize}
\end{lem}

We note that here the volume of a set $B\subset K_\RR^n$ is defined by fixing some isomorphism $K_\RR^n\cong \RR^{nd}$ and using the standard Lebesgue measure on $\RR^{nd}$. Crucially, the statement of the lemma does not depend on the particular identification $K_\RR^n\cong \RR^{nd}$, since any two volume forms differ by a scalar.

\begin{proof}[Proof of Lemma~\ref{lem:algminkowski2}]
    The proof is essentially identical to that of the original theorem given in~\cite[Theorem~3.30]{TV06}, though some care is required to differentiate between the $\QQ$-span and $K$-span.
    
    By the definition of $\ell_1$, we may find $v_1\in \Gamma$ on the boundary of $\ell_1\cdot B$, where $\ell_1\cdot B$ does not contain any non-zero elements of $\Gamma$. By the definition of $\lambda_2$, we may then find $v_2\in \Gamma$ on the boundary of $\ell_2\cdot B$ which is $K$-linearly independent of $v_1$, where $\ell_2\cdot B$ contains no elements of $\Gamma$ outside the $K$-span of $v_1$. Continuing, we have a $K$-basis $v_1,\ldots,v_n$ such that $v_j$ is on the boundary of $\ell_j\cdot B$, where $\ell_j\cdot B$ does not contain any element of $\Gamma$ outside the $K$-span of $v_1,\ldots,v_{j-1}$, as required by the first property. 

    Since $v_1,\ldots,v_n$ are $K$-linearly independent, the vectors $e_iv_j$ are $\QQ$-linearly independent. Therefore, the octahedron $S$ with vertices $\pm \frac{1}{r}e_iv_j$ is non-degenerate and spans $K^n$ over $\QQ$. 
    Suppose the interior of $S$ contains a non-zero point $v\in \Gamma$. Let $m$ be the smallest positive integer such that $v$ lies in the $K$-span of $v_1,\ldots,v_m$. Then $v$ does not lie in the $K$-span of $v_1,\ldots,v_{m-1}$. Since $\ell_m\cdot \overline{B}$ contains $v_1,\ldots,v_m$ and $B$ is $r$-thick, $r\ell_m\cdot \overline{B}$ contains $e_iv_j$ for all $i\in [d]$ and $j\leq m$. Therefore, $\ell_m\cdot \overline{B}$ contains $\pm\frac{1}{r}e_iv_j$ for all $i \in [d]$ and $j \le m$, so its interior $\ell_m\cdot B$ contains $v$. But this contradicts the definition of $\ell_m$, since $\ell_m\cdot B$ cannot contain any vector outside the $K$-span of $v_1,\ldots,v_{m-1}$, including $v$. Hence, the interior of $S$ contains no vector in $\Gamma$, verifying the second property.

    Since $e_iv_j\in r\ell_j\cdot \overline{B}$, we have that $\overline{B}$ contains the vectors $\frac{1}{r\ell_j}e_iv_j$ and, hence, the octahedron $S'$ with vertices $\pm \frac{1}{r\ell_j}e_iv_j$ for $i \in [d]$, $j \in [n]$. The volume of the simplex with vertices 0 and $e_iv_j$ for all $i \in [d]$, $j \in [n]$ is $\frac{1}{(nd)!}\Vol(K_\RR^n/\Gamma')$. Since $S'$ is the union of $2^{nd}$ scaled copies of this simplex, the volume of $S'$ is 
    $$\Vol(S')=\frac{1}{r^{nd}\ell_1^d\cdots \ell_n^d}\frac{2^{nd}}{(nd)!}\Vol(K_\RR^n/\Gamma').$$
    Therefore, since $\overline{B}$ contains $S'$, we have
    \begin{align*}
        \Vol(B) &\geq \Vol(S') = \frac{1}{r^{nd}\ell_1^d\cdots \ell_n^d}\frac{2^{nd}}{(nd)!}\Vol(K_\RR^n/\Gamma')\\
        &= \paren{\frac{(2/r)^n}{\ell_1\cdots\ell_n}}^d\frac{[\Gamma:\Gamma']}{(nd)!}\Vol(K_\RR^n/\Gamma),
    \end{align*}
    establishing the lower bound in (\ref{eqn:algminkowskiineq}).

    For the upper bound, we require the following lemma.

    \begin{lem}[Squeezing lemma {\cite[Lemma~3.31]{TV06}}]
        Let $S$ be a centrally symmetric convex body in $\RR^n$, $A$ be an open subset of $S$, $V$ be an $m$-dimensional subspace of $\RR^n$ and $0< \theta\leq 1$. Then there exists an open subset $A'$ of $S$ such that $\Vol(A')=\theta^m \Vol(A)$ and $(A'-A')\cap V\subseteq \theta\cdot (A-A)\cap V$.
    \end{lem}

    Let $V_j$ be the $\RR$-span of the $K$-span of $v_1,\ldots,v_j$, so that $V_j$ is a $jd$-dimensional real subspace of $K_\RR^n$. We apply the squeezing lemma iteratively, starting with $A_0:=\frac{\ell_n}{2}\cdot B$, to create open sets $A_1,\ldots,A_{n-1}\subseteq A_0$ such that 
    $$\Vol(A_j)=\paren{\frac{\ell_j}{\ell_{j+1}}}^{jd}\Vol(A_{j-1})$$
    and
    $$(A_j-A_j)\cap V_j\subseteq \frac{\ell_j}{\ell_{j+1}}\cdot (A_{j-1}-A_{j-1})\cap V_j$$
    for $j=1,\ldots,n-1$. Then $\Vol(A_{n-1})=(\ell_1\cdots\ell_n2^{-n})^d\Vol(B)$ and one can show by induction that
    $$(A_{n-1}-A_{n-1})\cap V_j\subseteq \frac{\ell_j}{\ell_n}\cdot (A_{j-1}-A_{j-1})\cap V_j.$$
    
    On the other hand, $A_{j-1}\subseteq A_0=\frac{\ell_n}{2}\cdot B$ and $B$ is centrally symmetric, so $A_{j-1}-A_{j-1}\subseteq \ell_n\cdot B$. It follows that
    $$(A_{n-1}-A_{n-1})\cap V_j\subseteq \lambda_j\cdot B\cap V_j$$
    for $j=1,\ldots,n$.
    By the definition of successive minima, $\lambda_j\cdot B\cap V_j$ does not contain any point in $\Gamma$ except for those in $V_{j-1}$. This implies that $A_{n-1}-A_{n-1}$ does not contain any point in $\Gamma$ other than the origin. If $\Vol(A_{n-1})>\Vol(K_\RR^n/\Gamma)$, then, by Blichfeldt's principle, one can find a translate $A_{n-1}+t$ of $A_{n-1}$ containing two distinct points of $\Gamma$. Thus, $A_{n-1}-A_{n-1}$ contains a non-zero point of $\Gamma$, a contradiction. Therefore, we have $\Vol(A_{n-1})\leq \Vol(K_\RR^n/\Gamma)$. Hence, we have
    $$(\ell_1\cdots\ell_n2^{-n})^d\Vol(B)\leq \Vol(K_\RR^n/\Gamma),$$
    giving the upper bound in (\ref{eqn:algminkowskiineq}).
\end{proof}

\subsection{$\cO_K$-GAPs and an algebraic John's theorem}

Recall that a \emph{generalised arithmetic progression (or GAP)} $P\subset \ZZ^d$ is a set of the form
$$P=\setcond{v_0+l_1v_1+\cdots+l_n v_n}{0\leq l_j< L_j \text{ for all } j}$$
for some $v_0,\ldots,v_n\in \ZZ^d$ and $L_1,\ldots,L_n\in \NN$. The \emph{dimension} of $P$ is $n$. We say that $P$ is \emph{proper} if all elements on the RHS are distinct and $k$-proper if 
$$\setcond{l_1v_1+\cdots+l_n v_n}{0\leq l_j< kL_j \text{ for all } j}$$
has all elements distinct.

Our object of study for the remainder of this section is the following algebraic analogue of a GAP, which we call an $\cO_K$-GAP.

\begin{defn}
    An \emph{$\cO_K$-GAP} is a set $P\subset K$ of the form
    \begin{equation}
        P=\setcond{v_0+l_1v_1+\cdots+l_nv_n}{l_j\in B(L_j) \text{ for all } j} \label{eqn:okgap}
    \end{equation}
    for some $v_0,\ldots,v_n\in K$ and $L_1,\ldots,L_n\in \NN$. The \emph{dimension} of $P$ is $n$. For $p\in \NN$, define
    \begin{equation}
        p\star P:=\setcond{pv_0+l_1v_1+\cdots+l_nv_n}{l_j\in B(pL_j) \text{ for all } j}. \label{eqn:okgapscale}
    \end{equation}
    We say that $P$ is \emph{proper} if all the elements on the RHS of (\ref{eqn:okgap}) are distinct and \emph{$p$-proper} if all the elements on the RHS of (\ref{eqn:okgapscale}) are distinct. Note that $p\star P$ is similar, but, because $B(L_j)$ is an \emph{open} ball, not exactly equal, to the $p$-fold sumset $pP$. 
\end{defn}

The classical John's theorem (see \cite{J48} or \cite[Theorem~3.13]{TV06}) says that any centrally symmetric convex body $A$ in $\RR^n$ can be approximated by an open centrally symmetric ellipsoid $E$ in the sense that $E \subseteq A \subseteq \sqrt{n} \cdot E$. A discrete version of this result, due to Tao and Vu~\cite{TV08}, says that the intersection of a centrally symmetric convex body with a lattice in $\RR^n$ can be approximated by a GAP. Here we prove the following algebraic analogue of this result. 

\begin{lem}
\label{lem:algjohn}
    For any real number $r\geq 1$, there are integer constants $D_1,D_2>0$ such that the following holds. Let $\Gamma\subseteq K^n$ be an $\cO_K$-lattice of full rank and $B\subset K_\RR^n$ be an $r$-thick convex centrally symmetric body. Then there exist $v_1,\ldots,v_n\in K$ and positive integers $L_1,\ldots,L_n$ such that the $\cO_K$-GAPs given by 
    \begin{align*}
        P_1 &:=\setcond{l_1v_1+\cdots+l_nv_n}{l_j\in B(L_j) \text{ for all } j},\\
        P_2 &:=\setcond{l_1v_1+\cdots+l_nv_n}{l_j\in B(D_1L_j) \text{ for all } j}
    \end{align*}
    satisfy
    \begin{equation}
        P_1\subseteq B\cap \Gamma\subseteq \frac{1}{D_2}\cdot P_2. \label{eqn:algjohn}
    \end{equation}
\end{lem}

Unlike for the discrete John's theorem for ordinary lattices, the constant $D_2$ is necessary here. 
Indeed, if $K$ has non-trivial ideal class group, then, for $\Gamma\subset \cO_K$ a non-principal ideal, we cannot hope for a one-dimensional $\cO_K$-GAP to span the same lattice as $\Gamma$, since any such $\cO_K$-GAP is generated by a single element.

\begin{proof}[Proof of Lemma~\ref{lem:algjohn}]
    Applying John's theorem to $B\subset K_\RR^n\cong \RR^{dn}$, we obtain an open centrally symmetric ellipsoid $E\subset K_\RR^n$ such that $E\subseteq B\subseteq \sqrt{dn}\cdot E$. 
    For any $x\in E$ and $i\in [d]$, $e_ix\in e_i\cdot B\subseteq r\cdot B\subseteq r\sqrt{dn}\cdot E$, so $E$ is $r_1$-thick with $r_1:=r\sqrt{dn}$. Consider the norm $\norm{\cdot}_E$ on $K_\RR^n$ whose unit ball is $E$, that is,
    $$\norm{x}_E:=\inf\setcond{\ell>0}{x\in \ell\cdot E}.$$
    Since $E$ is $r_1$-thick, for any $\ell>0$ and $x\in K_\RR^n$, $x\in \ell\cdot E$ implies that $e_ix\in r_1\ell\cdot E$. Therefore, for any $x \in  K_\RR^n$,
    \begin{equation}
        \norm{e_ix}_E\leq r_1\norm{x}_E. \label{eqn:normeix}
    \end{equation}
    Since $|a_i|\leq \norm{l}$ for any $l=a_1e_1+\cdots+a_de_d\in \cO_K$, we also have, for any $x \in  K_\RR^n$, that 
    \begin{equation}
        \norm{lx}_E\leq \sum_{i=1}^d \norm{a_ie_ix}_E\leq dr_1\norm{l}\norm{x}_E. \label{eqn:ellipsenorm}
    \end{equation}
    
    Let $v_1,\ldots,v_n\in K$ be as in Lemma~\ref{lem:algminkowski2}, when applied to the centrally symmetric convex body $E$. For each $j$, let 
    $$L_j:=\ceil{\frac{1}{nd r_1\norm{v_j}_E}}.$$
    Then, for any $l_j\in B(L_j)$, $\norm{l_j}< \frac{1}{nd r_1\norm{v_j}_E}$. Thus, by (\ref{eqn:ellipsenorm}),
    $$\norm{l_jv_j}_E\leq dr_1\norm{l_j}\norm{v_j}_E < \frac{1}{n}.$$
    Therefore, if $l_j\in B(L_j)$ for all $j$,
    $$\norm{l_1v_1+\cdots+l_nv_n}_E\leq \sum_{j=1}^n \norm{l_jv_j}_E < 1.$$
    In other words, $P_1\subseteq E\cap \Gamma\subseteq B\cap \Gamma$, giving the first inclusion in (\ref{eqn:algjohn}).

    Let $\Gamma'\subseteq \Gamma$ be the $\cO_K$-span of $v_1,\ldots,v_n$. Then, from Lemma~\ref{lem:algminkowski2}, $[\Gamma:\Gamma']\leq D:=\floor{r^{nd}(nd)!}$. 
    As a finite abelian group, $\Gamma/\Gamma'$ has order at most $D$, so every element has order dividing $D!$. Therefore, $D!\cdot \Gamma\subseteq \Gamma'$ or, equivalently, $\Gamma\subseteq \frac{1}{D!}\Gamma'$. 

    Since $E$ is a centrally symmetric ellipsoid, the norm $\norm{\cdot}_E$ arises from an inner product on $K_\RR^n$. Define a volume form on $K_\RR^n$ based on this inner product. Then $\Vol(E)=V_{nd}$, the volume of the unit ball in $\RR^{nd}$. For $u_1,\ldots,u_{nd}\in K_\RR^n$, write $u_1\wedge\cdots\wedge u_{nd}$ for the parallelotope in $K_\RR^n$ spanned by $u_1,\ldots u_{nd}$. Then $\Vol(u_1\wedge\cdots\wedge u_{nd})\leq \norm{u_1}_E\cdots \norm{u_{nd}}_E$.

    Let the successive minima of $E$ with respect to $\Gamma$ be $\ell_1,\ldots,\ell_n$, so we have $\norm{v_j}_E=\ell_j$. Let $x\in B\cap \Gamma\subseteq \sqrt{dn}\cdot E$, so that $\norm{x}_E\leq \sqrt{dn}$. Since $x\in \Gamma\subseteq \frac{1}{D!}\Gamma'$, we can find unique integers $l_{ij}$ for $i=1,\ldots,d$ and $j=1,\ldots,n$ such that
    $$x=\frac{1}{D!}(l_{11}e_1v_1+\cdots+l_{dn}e_dv_n).$$
    Using Cramer's rule, we can solve for $|l_{ij}|$. This gives 
    \begin{align*}
        |l_{ij}| &= D!\frac{\Vol(e_1v_1\wedge \cdots\wedge x\wedge\cdots\wedge e_dv_n)}{\Vol(e_1v_1\wedge\cdots\wedge e_dv_n)} & \text{here $x$ is in place of $e_iv_j$}\\
        &=D!\frac{\Vol(e_1v_1\wedge \cdots\wedge x\wedge\cdots\wedge e_dv_n)}{\Vol(K_\RR^n/\Gamma')}\\
        &\leq D!\frac{\norm{x}_E\prod_{(i',j')\neq (i,j)}\norm{e_{i'}v_{j'}}_E}{\Vol(K_\RR^n/\Gamma')}\\
        &\leq D!\frac{r_1^{nd-1}\norm{x}_E\prod_{(i',j')\neq (i,j)}\norm{v_{j'}}_E}{\Vol(K_\RR^n/\Gamma')} & \text{by (\ref{eqn:normeix})}\\
        &= D!r_1^{nd-1}\frac{(\ell_1\cdots\ell_n)^d\norm{x}_E}{\ell_j\Vol(K_\RR^n/\Gamma')}.
    \end{align*}
    From Lemma~\ref{lem:algminkowski2}, we have 
    $$\Vol(K_\RR^n/\Gamma') \geq \Vol(K_\RR^n/\Gamma)\geq \paren{\frac{\ell_1\cdots\ell_n}{2^n}}^d\Vol(E)= \paren{\frac{\ell_1\cdots\ell_n}{2^n}}^dV_{nd}.$$
    Therefore, using that $\norm{x}_E\leq \sqrt{dn}$ and $L_j\geq \frac{1}{ndr_1\norm{v_j}_E}$, we have
    \begin{align*}
        |l_{ij}|\leq \frac{D!r_1^{nd-1}2^{nd}\norm{x}_E}{\ell_j V_{nd}}< \frac{D!r_1^{nd}2^{nd+1}nd\sqrt{nd}}{V_{nd}}L_j.
    \end{align*}
    We obtain the second inclusion in (\ref{eqn:algjohn}) by setting $D_2=D!$ and $D_1=\ceil{D!r_1^{nd}2^{nd+1}nd\sqrt{nd}/V_{nd}}$.
\end{proof}

We now come to a key lemma, for which we need our algebraic version of John's lemma, saying that if $P$ is an $\cO_K$-GAP that is not $p$-proper, then there is an $\cO_K$-GAP of smaller dimension which contains and is not too much larger than $P$.

\begin{lem} \label{lem:properreduce}
    If $P$ is an $\cO_K$-GAP of dimension $n$ that is not $p$-proper, then there is an $\cO_K$-GAP $Q$ of dimension $n-1$ containing $P$ with $|Q|\ll_{n,p} |P|$.
\end{lem}
\begin{proof}
    Assume that $P$ is centered and of the form
    $$P=\setcond{l_1v_1+\cdots+l_nv_n}{l_j\in B(L_j)}$$
    with $L_j>1$ for all $j$. Since $P$ is not $p$-proper, there exist $l_j,l_j'\in B(pL_j)$ for all $j$ such that $l_j\neq l_j'$ for some $j$ and
    $$l_1v_1+\cdots+l_nv_n=l_1'v_1+\cdots+l_n'v_n.$$
    Setting $a_j=l_j-l_j'\in B(2pL_j)$, we have that the $a_j$ are not all 0 and $a_1v_1+\cdots+a_nv_n=0$. We may assume without loss of generality that $a_n\neq 0$. Then we have the relation
    \begin{align}
        v_n=-\frac{a_1v_1}{a_n}-\cdots-\frac{a_{n-1}v_{n-1}}{a_n}. \label{eqn:lindep}
    \end{align}

    Let $w=(-\frac{a_1}{a_n},\ldots,-\frac{a_{n-1}}{a_n})\in K^{n-1}$. Let $\Gamma:=\cO_K^{n-1}+\cO_K\cdot w\subset K^{n-1}$. Then $\Gamma$ is a discrete lattice which is invariant under multiplication by $\cO_K$ and so is an $\cO_K$-lattice. $\Gamma$ is also of full rank, since it contains $\cO_K^{n-1}$. Consider the homomorphism $f:\Gamma\to K$ given by
    $$f((x_1,\ldots,x_{n-1})+x_nw):=x_1v_1+\cdots+x_nv_n.$$
    Then $f$ is well-defined because of the relation (\ref{eqn:lindep}). Note also that $f$ is $\cO_K$-linear, that is, $f$ is linear and $f(ax)=af(x)$ for any $a\in \cO_K,x\in \Gamma$. We may also extend $f$ $\cO_K$-linearly to a $K$-linear map $f:K^{n-1}\to K$.

    Let $B_0\subset K_\RR^{n-1}$ be the convex centrally symmetric body
    $$B_0:= \setcond{(x_1,\ldots,x_{n-1})\in K_\RR^{n-1}}{x_i\in B_\RR(L_i)}.$$
    Let $B=B_0+B_\RR(L_n)\cdot w$, which is also a convex centrally symmetric body. Since $B_0$ and $B_\RR(L_n)\cdot w$ are $C_1$-thick, so is $B$. Indeed, if $x\in B_0$ and $y\in B_\RR(L_n)\cdot w$, then $e_i\cdot(x+y)=e_i\cdot x+e_i\cdot y\in C_1\cdot B_0+C_1\cdot (B_\RR(L_n)\cdot w)=C_1\cdot(B_0+B_\RR(L_n)\cdot w)$. 

    \begin{claim}
        One has the inclusions
        $$P\subseteq f(B\cap \Gamma)\subseteq (2pC_1+1)\star P.$$
    \end{claim}
    \begin{proof}
        For the first inclusion, let $v=l_1v_1+\cdots+l_nv_n\in P$ with $l_j\in B(L_j)$. Then $v=f((l_1,\ldots,l_{n-1})+l_nw)$ with $\norm{l_j}<L_j$, so that $(l_1,\ldots,l_{n-1})+l_nw\in B\cap \Gamma$.

        For the second inclusion, let $(l_1,\ldots,l_{n-1})+l_nw\in B\cap \Gamma$ with $l_j\in \cO_K$. Since $(l_1,\ldots,l_{n-1})+l_nw\in B$, there exist $x_1,\ldots,x_n\in K_\RR$ with $\norm{x_j}<L_j$ such that $(l_1,\ldots,l_{n-1})+l_nw=(x_1,\ldots,x_{n-1})+x_nw$. In other words, $l_j-\frac{a_jl_n}{a_n}=x_j-\frac{a_jx_n}{a_n}$ for $j=1,\ldots,n-1$. Let $z=\frac{l_n-x_n}{a_n}\in K_\RR$, so we have
        \begin{equation}
            l_j-x_j=a_jz \label{eqn:lj-xj}
        \end{equation}
        for all $j=1,\ldots,n$. Let $x\in \cO_K$ be the closest element to $z$ according to the metric $\norm{\cdot}$. Recall that this is the $\infty$-norm, so we have $\norm{x-z}\leq 1$. Let $l_j'=l_j-a_jx\in \cO_K$. Then $l_1v_1+\cdots+l_nv_n=l_1'v_1+\cdots+l_n'v_n$, so we have $f((l_1,\ldots,l_{n-1})+l_nw)=l_1'v_1+\cdots+l_n'v_n$. It suffices to show that $\norm{l_j'}< (2pC_1+1)L_j$ for all $j$. 
        But we have 
        \begin{align*}
            \norm{l_j'} &=\norm{l_j-a_jx}\\
            &\leq \norm{l_j-a_jx-x_j}+\norm{x_j}\\
            &< \norm{a_j(z-x)}+L_j & \text{by (\ref{eqn:lj-xj})}\\
            &\leq C_1\norm{a_j}\norm{z-x}+L_j & \text{by Lemma~\ref{lem:consts}}\\
            &\leq (2pC_1+1)L_j, 
        \end{align*}
    as required.    
    \end{proof}

    By Lemma~\ref{lem:algjohn}, we can find constants $D_1,D_2=O_{n}(1)$ and $\cO_K$-GAPs $P_1,P_2$ of dimension $n-1$ such that $P_2=D_1\star P_1$ and $P_1\subseteq B\cap \Gamma\subseteq \frac{1}{D_2}\cdot P_2$. In particular, $P_2$ can be covered by $D_1^{n-1}$ translates of $P_1$.
    
    Applying the homomorphism $f$, we obtain 
    $$f(P_1)\subseteq f(B\cap \Gamma)\subseteq \frac{1}{D_2}f(P_2).$$
    Since $f$ is $\cO_K$-linear, $f(P_1)$ and $f(P_2)$ are also $\cO_K$-GAPs of dimension $n-1$. Setting $Q=\frac{1}{D_2}f(P_2)$, which is again an $\cO_K$-GAP of dimension $n-1$, we have, by the claim above,  that $P\subseteq f(B\cap \Gamma)\subseteq Q$, so it suffices to show that $Q$ is small. Since $P_2$ can be covered by $D_1^{n-1}=O_n(1)$-many translates of $P_1$, $f(P_2)$ can also be covered by $O_n(1)$-many translates of $f(P_1)$. But then
    \[    |f(P_2)| \ll_n |f(P_1)|
        \leq |f(B\cap \Gamma)|
        \leq |(2pC_1+1)\star P|
        \ll_{n,p} |P|,\]
    as required.
\end{proof}

\subsection{Freiman's theorem for sums of dilates} \label{sec:freiman}

One version of Freiman's fundamental theorem on sets of small doubling is as follows.

\begin{thm}[Freiman~\cite{Fr73}] \label{thm:freiman}
    For every $C>0$, there are constants $n$ and $F$ such that for any $A\subset \ZZ^d$ satisfying $|A+A|\leq C|A|$, there exists a proper GAP $P\subset \ZZ^d$ containing $A$ of dimension at most $n$ and size at most $F|A|$.
\end{thm}

We have now built up sufficient background to prove the promised Freiman-type structure theorem for sets with small sums of dilates, which we restate for the reader's convenience.

\begin{thm} \label{thm:algfreiman}
    For every $C>0$ and $p \in \NN$, there are constants $n$ and $F$ such that for any $A\subset K$ satisfying
    $$|A+\lambda_1\cdot A+\cdots+\lambda_k\cdot A|\leq C|A|,$$
    there exists a $p$-proper $\cO_K$-GAP $P\subset K$ containing $A$ of dimension at most $n$ and size at most $F|A|$.
\end{thm}

Recall, from Lemma~\ref{lem:consts}, that we have constants $C_2,C_3\in \NN$ with the property that $\lambda_lx\in \frac{1}{C_2}\cdot B(C_3\norm{x})$ for all $l=0,\ldots,k$ and $x\in \cO_K$. Thus, if $P$ is an $\cO_K$-GAP, then $\lambda_l\cdot P$ lies in a translate of $\frac{1}{C_2}\cdot (C_3\star P)$. Indeed, if $x=v_0+l_1v_1+\cdots+l_mv_m\in P$, then $\lambda_lx=\lambda_lv_0+(\lambda_l l_1)v_1+\cdots+(\lambda_l l_m)v_m$ with $\lambda_l l_i\in \frac{1}{C_2}B(C_3\norm{l_i})$.  Therefore,
\begin{align*}
    |P+\lambda_1\cdot P+\cdots+\lambda_k\cdot P| &\leq |(k+1)C_3\star P|\\
    &\leq ((k+1)C_3)^{nd}|P|.
\end{align*}
In other words, $P$ has a small sum of dilates. That is, Theorem~\ref{thm:algfreiman} embeds a set $A$ with a small sum of dilates into another, more structured set which, unlike an ordinary GAP, also has a small set of dilates. We now proceed to the proof of this statement.

\begin{proof}[Proof of Theorem~\ref{thm:algfreiman}]
    By translating, we may assume that $0\in A$. By the Ruzsa triangle inequality,
    \[|A+A||\lambda_1\cdot A+\cdots+\lambda_k\cdot A|\leq |A+\lambda_1\cdot A+\cdots+\lambda_k\cdot A|^2\leq C^2|A|^2.\]
    Using the trivial bound $|\lambda_1\cdot A+\cdots+\lambda_k\cdot A|\geq |A|$, we obtain $|A+A|\leq C^2|A|$. By the Pl\"unnecke--Ruzsa inequality, $|A+A+A|\leq C^6|A|$. By the Ruzsa triangle inequality again,
    \[|(A+A)+\lambda_1\cdot A+\cdots+\lambda_k\cdot A||A|\leq |A+A+A||A+\lambda_1\cdot A+\cdots+\lambda_k\cdot A|\leq C^7|A|^2,\]
    so $|(A+A)+\lambda_1\cdot A+\cdots+\lambda_k\cdot A|\leq C^7|A|$. Similar repeated applications of the triangle inequality gives $|(A+A)+\lambda_1\cdot (A+A)+\cdots+\lambda_k\cdot (A+A)|\leq C^{7+6k}|A|$. Thus, 
    $A+\lambda_1\cdot A+\cdots+\lambda_k\cdot A$ has small doubling constant. Therefore, by Freiman's theorem, $A+\lambda_1\cdot A+\cdots+\lambda_k\cdot A$ is contained in a proper GAP 
    $$P_0=\setcond{l_1v_1+\cdots+l_{n_0}v_{n_0}}{-L_i< l_i< L_i}$$
    of dimension $n_0$ with $|P_0|\ll |A|$. Note that since $0\in A\subseteq P_0$, we are free to assume that $P_0$ is centered.

    Now let $P_1$ be the $\cO_K$-GAP given by
    $$P_1=\setcond{l_1v_1+\cdots+l_{n_0}v_{n_0}}{l_i\in B(L_i)}.$$
    Then $P_1$ contains $P_0$. At first glance, it might seem that the size of $P_1$ could be as large as $|P_0|^d$. However, we now show that this is not the case.

    \begin{claim}
        $|P_1|\ll |P_0|$. 
    \end{claim}

    \begin{proof}
        For a subset $X\subseteq K$ and $c>0$, we say that $X$ is \emph{$(c,P_0)$-small} if $X$ can be covered by $c$-many translates of $P_0$. For brevity, we will simply say that $X$ is $P_0$-small if $c$ is a bounded constant independent of $X,P_0$. 
        Thus, if $X,Y$ are $P_0$-small, so is their sumset $X+Y$. Indeed, if $X,Y$ can be covered by $x,y$-many translates of $P_0$, respectively, then $X+Y$ can be covered by $xy$-many translates of $P_0+P_0$, which itself can be covered by $2^{n_0}$-many translates of $P_0$.

        We shall show that for each $i\in [d],j\in [n_0]$, the set $S_{ij}:=\set{e_iv_j,2e_iv_j,\ldots,L_je_iv_j}$ is $P_0$-small. Then we would have proved the claim, since the sets $\set{-L_je_iv_j,\ldots,L_je_iv_j}$ are then $P_0$-small, $P_1$ is the sum of these sets and there are only a bounded number of them. 

        Since $\lambda_1,\ldots,\lambda_k$ generate $K$, there exist (fixed) integers $b,a_1,\ldots,a_k$ with $b>0$ such that $be_i=a_1\lambda_1+\cdots+a_k\lambda_k$. It will suffice to show that the set $S:=\set{be_iv_j,2be_iv_j,\ldots,L_jbe_iv_j}$ is $P_0$-small, since $S_{ij}$ can be covered by $b$ translates of it. But then it suffices to show that $S_l':=\set{a_l\lambda_lv_j,2a_l\lambda_lv_j,\ldots,L_ja_l\lambda_lv_j}$ is $P_0$-small for each $l$, since $S$ is contained in $S_1'+\cdots+S_k'$. But then, finally, it suffices to show that $S_l:=\set{\lambda_lv_j,2\lambda_lv_j,\ldots,L_j\lambda_lv_j}$ is $P_0$-small for each $l$, since $S_l'$ is covered by $|a_l|$-many translates of $S_l$.

        Suppose $|P_0+P_0|< c|A|$, where $c=O(1)$ is a positive integer. Let $s$ be an arbitrary positive integer with $s < L_j/c$. Consider the sets
        $$A,A+sv_j,A+2sv_j,\ldots,A+csv_j.$$
        All these sets have size $|A|$ and are contained in $P_0+P_0$. But $|P_0+P_0|<c|A|$, so two of these sets intersect, say $(A+msv_j)\cap (A+m'sv_j)\neq\emptyset$ for $0\leq m<m'\leq c$. Thus, $(m'-m)sv_j\in A-A$. Therefore, $c!s\lambda_lv_j\in c!(\lambda_l\cdot A)-c!(\lambda_l\cdot A)\subseteq c!P_0-c!P_0$. Since $1\leq s < L_j/c$ was arbitrary, we have that the set
        $$\set{c!\lambda_lv_j,2c!\lambda_lv_j,\ldots,\floor{L_j/c}c!\lambda_lv_j}\subseteq c!P_0-c!P_0$$
        is $P_0$-small. Thus, the set $T:=\set{c!\lambda_lv_j,2c!\lambda_lv_j,\ldots,L_jc!\lambda_lv_j}$ is $P_0$-small. Finally, $S_l$ is $P_0$-small since it can be covered by $c!$-many translates of $T$.
    \end{proof}

    If $P_1$ is $p$-proper, then we are done. Otherwise, by Lemma~\ref{lem:properreduce}, we can find an $\cO_K$-GAP $P_2$ of one dimension smaller containing $P_1$ with $|P_2|\ll |P_1|$. If $P_2$ is also not $p$-proper, we invoke Lemma~\ref{lem:properreduce} again to obtain $P_3$ and so on. Note that we can only do this at most $n_0$ times, since any $\cO_K$-GAP of dimension 1 is necessarily $p$-proper. Thus, we will eventually find a $p$-proper $\cO_K$-GAP $P$ containing $A$ of dimension $O(1)$ with $|P|\ll |A|$.
\end{proof}

\section{Reduction to a dense subset of the box} \label{sec:reduction}

With the results of the last section in hand, we are now able to complete the second part of our plan, reducing the proof of our main result, in the form of Theorem~\ref{thm:main2}, to the case where $A$ is a dense subset of the box $[0,N)^d$.

\begin{lem} \label{lem:dense}
    For any $\varepsilon>0$, there exists $N_0$ such that if $N \ge N_0$  and $A\subseteq [0,N)^d$ with $|A|\geq \varepsilon N^d$, then
    $$|\cL_0 A+\cdots+\cL_k A|\geq H(\lambda_1,\ldots,\lambda_k)|A|-o_{\varepsilon}(|A|).$$
\end{lem}

The proof of Lemma~\ref{lem:dense}, which is the heart of this paper, will occupy us for the next few sections. 
Before moving on to this, we first show that, together with our version of Freiman's theorem for sums of dilates, Lemma~\ref{lem:dense} completes the proof of Theorem~\ref{thm:main2}. 

\begin{proof}[Proof of Theorem~\ref{thm:main2} assuming Lemma~\ref{lem:dense}]
    Let $A\subset \ZZ^d$ be finite and suppose that
    $$|\cL_0 A+\cdots+\cL_k A|\leq H|A|,$$
    where $H=H(\lambda_1,\ldots,\lambda_k)$. Let $\Phi,\Phi',\frakD$ be as in Section~\ref{sec:mapZd}. Setting $A'=\Phi'^{-1}(A)\subseteq \frakD\subseteq \cO_K$, we have
    $$|A'+\lambda_1\cdot A'+\cdots+\lambda_k\cdot A'|\leq H|A'|.$$
    Let $C_3$ be as in Lemma~\ref{lem:consts}. By Theorem~\ref{thm:algfreiman}, our version of Freiman's theorem for sums of dilates applied with $p = (k+1)C_3$, $A'$ is contained in a $(k+1)C_3$-proper $\cO_K$-GAP $P\subset K$ of dimension $n=O(1)$ and size $|P|=O(|A'|)$. Suppose $P$ is of the form
    $$\setcond{v_0+l_1v_1+\cdots+l_nv_n}{l_j\in B(L_j)}.$$
    Then $|P|\sim (\prod_{j=1}^n L_j)^d$, where the notation $A \sim B$ indicates that the quantities $A$ and $B$ are equal up to a constant multiplicative factor depending only on $\lambda_1, \dots, \lambda_k$. By translating $A'$, we may assume that $v_0=0$. By Lemma~\ref{lem:consts}, we have $\lambda_l\cdot B(L_j)\subseteq \frac{1}{C_2}\cdot B(C_3L_j)$ for all $j,l$. Thus, $\lambda_l\cdot A'\subseteq \lambda_l\cdot P\subseteq \frac{1}{C_2}\cdot (C_3\star P)$ for all $l$. 
    
    We will now map $P$ to a dense subset of a box via a Freiman isomorphism. Let $v_1^*=1$ and $v_l^*=3(k+1)C_3L_{l-1}v_{l-1}^*$ for $l=2,\ldots,n$. Let $P^*$ be the $\cO_K$-GAP 
    $$P^*:=\setcond{l_1v_1^*+l_2v_2^*+\cdots+l_nv_n^*}{l_j\in B(L_j)}.$$
    Then $P^*$ is $(k+1)C_3$-proper. Indeed, if $l_1v_1^*+l_2v_2^*+\cdots+l_nv_n^*=l_1'v_1^*+l_2'v_2^*+\cdots+l_n'v_n^*$ for some $l_j,l_j'\in B((k+1)C_3L_j)$, then we have
    $$(l_1-l_1')v_1^*+\cdots+(l_n-l_n')v_n^*=0.$$
    Suppose $l_t\neq l_t'$ for some $t\in [n]$. 
    Let $t$ be the largest such index, so we have
    $$(l_t'-l_t)v_t^*=(l_1-l_1')v_1^*+\cdots+(l_{t-1}-l_{t-1}')v_{t-1}^*.$$
    However, $\norm{(l_t'-l_t)v_t^*}\geq v_t^*=3(k+1)C_3L_{t-1}v_{t-1}^*$, whereas
    \begin{align*}
        \norm{(l_1-l_1')v_1^*+\cdots+(l_{t-1}-l_{t-1}')v_{t-1}^*} &\leq \norm{(l_1-l_1')v_1^*}+\cdots+\norm{(l_{t-1}-l_{t-1}')v_{t-1}^*}\\
        &\leq (\norm{l_1}+\norm{l_1'})v_1^*+\cdots+(\norm{l_{t-1}}+\norm{l_{t-1}'})v_{t-1}^*\\
        &< 2(k+1)C_3L_1v_1^*+\cdots+2(k+1)C_3L_{t-1}v_{t-1}^*\\
        &\leq 3(k+1)C_3L_{t-1}v_{t-1}^*,
    \end{align*}
    a contradiction. This proves that $P^*$ is $(k+1)C_3$-proper.
    
    Consider $\Psi:(k+1)C_3\star P\to (k+1)C_3\star P^*$, the natural bijection given by
    $$l_1v_1+\cdots+l_nv_n \longleftrightarrow l_1v_1^*+l_2v_2^*+\cdots+l_nv_n^*.$$
    Let $A^*=\Psi(A')$, so that $|A^*|=|A'|$. We claim that for $l=0,\ldots,k$, we have $\Psi(C_2\lambda_l\cdot A')=C_2\lambda_l\cdot A^*$. Indeed, first observe that the LHS is well-defined, since $\lambda_l\cdot P\subseteq \frac{1}{C_2}\cdot (C_3\star P)$, so we have $C_2\lambda_l\cdot P\subseteq C_3\star P$, 
    which is in the domain of $\Psi$. For any $a=l_1v_1+\cdots+l_nv_n\in A'$, set $l_{jl}'=C_2\lambda_l\cdot l_j$, which belongs to $B(C_3L_j)$ by Lemma~\ref{lem:consts}. Then 
    \begin{align*}
        \Psi(C_2\lambda_l\cdot a) &= \Psi(l_{1l}'v_1+\cdots+l_{nl}'v_n) = l_{1l}'v_1^*+\cdots+l_{nl}'v_n^*\\
        &= C_2\lambda_l\cdot (l_1v_1^*+\cdots+l_nv_n^*) = C_2\lambda_l\Psi(a).
    \end{align*}
    This proves the stated claim that $\Psi(C_2\lambda_l\cdot A')=C_2\lambda_l\cdot A^*$.
    
    Since $P$ is $(k+1)C_3$-proper, $C_3\star P$ is $(k+1)$-proper. Hence, $\Psi$ is a $(k+1)$-Freiman isomorphism on $C_3\star P$ and, therefore, since $C_3 > C_2$,
    \begin{align*}
        \Psi(C_2\cdot (A'+\lambda_1\cdot A'+\cdots+\lambda_k\cdot A')) &= \Psi(C_2\lambda_0\cdot A'+C_2\lambda_1\cdot A'+\cdots+C_2\lambda_k\cdot A')\\
        &= \Psi(C_2\lambda_0\cdot A')+\Psi(C_2\lambda_1\cdot A')+\cdots+\Psi(C_2\lambda_k\cdot A')\\
        &= C_2\lambda_0\cdot A^*+C_2\lambda_1\cdot A^*+\cdots+C_2\lambda_k\cdot A^*\\
        &= C_2\cdot(A^*+\lambda_1\cdot A^*+\cdots+\lambda_k\cdot A^*).
    \end{align*}
    It follows that
    $$|A^*+\lambda_1\cdot A^*+\cdots+\lambda_k\cdot A^*| = |A'+\lambda_1\cdot A'+\cdots+\lambda_k\cdot A'|.$$
    
    Note that $P^*\subseteq B(L)$ for some $L\sim \prod_{j=1}^n L_j$. Recall that $C_2$ is an integer satisyfing $C_2\lambda_l\in \cO_K$ for all $l$. In particular, $C_2\in \frakD$ and, since $P^*\subset \cO_K$, $C_2\cdot P^*\subset \frakD$. 
    Since $C_2\cdot P^*\subseteq B(C_2L)$, $\Phi'(C_2\cdot P^*)$ is contained in a box $[-N,N]^d$ with $N\sim L$. But $N^d\sim |P|\sim |A|$ and so $\Phi'(C_2\cdot A^*)$ is a dense subset of the box $[-N,N]^d$. By Lemma~\ref{lem:dense} (after translating into the box $[0,2N+1)^d$), we have
    \begin{align*}
        |\cL_0 A+\cdots+\cL_k A| &= |A'+\lambda_1\cdot A'+\cdots+\lambda_k\cdot A'|\\
        &= |A^*+\lambda_1\cdot A^*+\cdots+\lambda_k\cdot A^*|\\
        &= |\cL_0(\Phi'(C_2\cdot A^*))+\cdots+\cL_k(\Phi'(C_2\cdot A^*))|\\
        &\geq H|A^*|-o(|A^*|)\\
        &= H|A|-o(|A|),
    \end{align*}    
    as required.
\end{proof}

\section{Lattice densities} \label{sec:ld}

As already mentioned in the introduction, the key to proving Lemma~\ref{lem:dense} is to represent each discrete set $A$ by a continuous set $\overline{A}$, which we call a lattice density, to which we can apply the continuous estimate given by Theorem~\ref{thm:cts}. 
In this section, we introduce these lattice densities and prove some general facts about them. Very roughly, the lattice density of a set $A\subseteq \ZZ^d$ will encode the density of $A$ with respect to certain lattices.

\subsection{Lattice densities for periodic sets}

Let $L$ be a lattice of rank $d$, that is, $L\cong \ZZ^d$. We say that $A\subseteq L$ is \emph{$d$-periodic} if its group of translational symmetries has rank $d$. Let $\cF=\set{L_1\subseteq L_2\subseteq \cdots \subseteq L_k}$ be a flag of sublattices of $L$, each of which has rank $d$. In this section, we will define the \emph{lattice density} of any $d$-periodic set $A\subseteq L$ with respect to the flag $\cF$, denoted by $\LD(A;\cF)$, which will be a subset of $[0,1]^k$ that is a finite union of closed axis-aligned boxes.

For any affine lattice $M\subseteq L$ of rank $d$, we write $\rho_M(A)$ for the density of $A\cap M$ in $M$. Since $A$ is $d$-periodic, this density is always well-defined. In particular, $0\leq \rho_M(A)\leq 1$.
This already allows us to define the lattice density for $k=1$. 
Indeed, if $\cF=\set{L_1}$ and $A\cap L_1 \ne \emptyset$, we set $\LD(A;\cF)$ to be the interval $[0,\rho_{L_1}(A)]\subset \RR$, while if $A\cap L_1=\emptyset$, we set $\LD(A;\cF)=\emptyset$.

For $k>1$, let $a_1,\ldots,a_m\in L_k$ be any set of coset representatives of $L_k/L_{k-1}$, where $m=[L_k:L_{k-1}]$. Let $D_j=\LD(A+a_j;\cF\setminus L_k)\subseteq [0,1]^{k-1}$ for each $j \in [m]$ and 
$$D=\bigcup_{j=1}^m \paren{D_j\times \left[\frac{j-1}{m},\frac{j}{m}\right]}\subseteq [0,1]^k.$$
Finally, set $\LD(A;\cF)=C_k(D)$, where $C_k$ is the compression in the $k$-th direction, defined as follows. 

In our case, we will only be compressing sets which are finite unions of axis-aligned closed boxes. Let $X\subset \RR^d$ be such a set and $1\leq i\leq d$. Let $\pi_i:\RR^d\to \RR^{d-1}$ be the projection along the $i$-th axis. For $x\in \RR^{d-1}$, let $X_x=\pi_i^{-1}(x)$, viewed as a subset of $\RR$, and write $|X_x|$ for the measure of $X_x$.
Now define $C_i'(X)$ to be the set $Y$ such that $\pi_i(X)=\pi_i(Y)$ and, for each $x\in \pi_i(X)$, $Y_x$ is the interval $[0,|X_x|]$.
However, because of boundary issues, this is not quite the compression we want. For example, if $X=[0,1]^2\cup [1,2]^2\subset \RR^2$, then $C_2'(X)=[0,2]\times [0,1]\cup \{1\}\times [1,2]$. The artifact $\{1\}\times [1,2]$ is undesirable and only arises because the boundaries of the two squares $[0,1]^2$ and $[1,2]^2$ overlap in the projection. To remove this artifact, we formally define $C_i(X)$ to be the closure of the interior of $C_i'(X)$. Since we will only be compressing sets which are finite unions of axis-aligned closed boxes, we still enjoy the main properties of compressions, such as preservation of the measure of $X$ and that $C_i(X)$ is also a finite union of axis-aligned closed boxes. We will say that $X$ is \emph{$C_i$-compressed} if $C_i(X) = X$ and \emph{compressed} if it is $C_i$-compressed for all $i$.

Observe that, because of the compression, $\LD(A;\cF)$ is independent of the ordering $a_1,\ldots,a_m$.

\begin{example}
Suppose $d=1$, $k=2$, $L=\ZZ$, $\cF=\set{3\ZZ\subset \ZZ}$ and $A=12\ZZ\cup (12\ZZ+1)\cup (6\ZZ+3)$. Pick $a_i=-i$ for $i=1,2,3$ to be the coset representatives of $\ZZ/3\ZZ$. Let $A_i=(A-i)\cap 3\ZZ$ for $i=1,2,3$. Thus, $A_1,A_2,A_3$ are the parts of $A$ in the residue classes mod $3$, translated so they all lie in $3\ZZ$. We can easily check that 
\begin{itemize}
    \item $A_1=12\ZZ$,
    \item $A_2=\emptyset$,
    \item $A_3=12\ZZ+\set{0, 6, 9}$.
\end{itemize}

From the definition, $D_i=\LD(A_i; \set{3\ZZ})=[0,\rho_{3\ZZ}(A_i)]$, so we have $D_1=[0,1/4]$, $D_2=\emptyset$ and $D_3=[0,3/4]$. Stacking these intervals vertically and compressing, we get $\LD(A;\cF)\subset [0,1]^2$ as shown in Figure~\ref{fig:ld-ex}.

\begin{figure}
    \centering
    \begin{tikzpicture}
        \tikzmath{\x=6;}
        
        \draw[fill=lightgray] (0,0) rectangle (1,4/3) node[pos=0.5] {$D_1$};
        \draw[fill=lightgray] (0,8/3) rectangle (3,4) node[pos=0.5] {$D_3$};
        \node[anchor=east] at (0,2/3) {$1\pmod{3}$};
        \node[anchor=east] at (0,2) {$2\pmod{3}$};
        \node[anchor=east] at (0,10/3) {$0\pmod{3}$};
        
        \draw[fill=lightgray] (\x,0) -- (\x+3,0) -- (\x+3,4/3) -- (\x+1,4/3) -- (\x+1,8/3) -- (\x,8/3) -- cycle;

        \foreach \a in {0,\x}{
            \draw[->] (\a-0.5,0) -- (\a+4.5,0);
            \draw[->] (\a,-0.5) -- (\a,4.5);
            \draw (\a+4,0) -- (\a+4,0.2) node[anchor=north] at (\a+4,-0.05) {$1$};
            \draw (\a,4) -- (\a+0.2,4) node[anchor=east] at (\a,4) {$1$};
        }
    \end{tikzpicture}
    \caption{On the left, the $D_i$ are stacked, while on the right they are compressed to give the final lattice density.}
    \label{fig:ld-ex}
\end{figure}
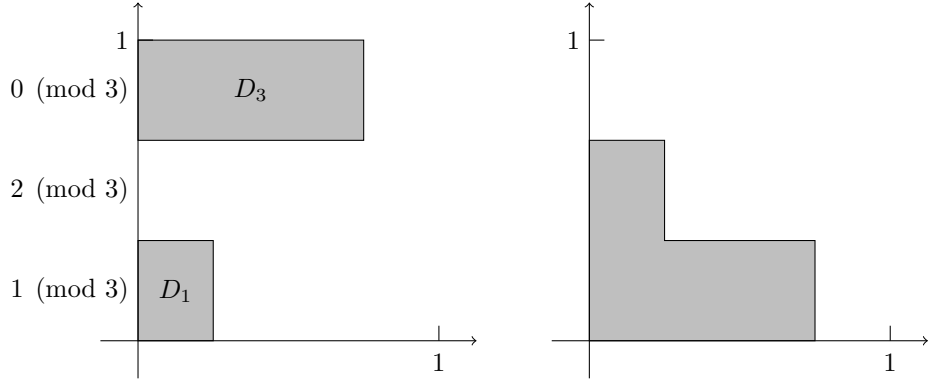

\end{example}

Throughout the rest of this section, $\cF=\set{L_1\subseteq \cdots\subseteq L_k}$ will be a flag of full-rank sublattices of a lattice $L\cong \ZZ^d$ and $A\subseteq L$ a $d$-periodic subset of $L$, our aim being to understand the properties of the lattice density $\LD(A;\cF)$. We begin with some basic observations. 

\begin{lem} \label{lem:ldbasic}
    The following are true:
    \begin{enumerate}
        \item For any $a\in L_k$, $\LD(A;\cF)=\LD(A+a;\cF)$.
        \item $\LD(A;\cF)$ is compressed. 
        \item If $B\subseteq A$ is $d$-periodic, then $\LD(B;\cF)\subseteq \LD(A;\cF)$.
        \item $\rho_{L_k}(A)=\Vol(\LD(A;\cF))$.
        \item $\LD(A;\cF)$ is a finite union of boxes of the form
        $$[0,r]\times \sqb{0,\frac{m_2}{[L_2:L_1]}}\times\cdots\times \sqb{0,\frac{m_k}{[L_k:L_{k-1}]}},$$
        where $r\in (0,1]$ and $m_2,\ldots,m_k$ are positive integers.
    \end{enumerate}
\end{lem}

\begin{proof}
    We proceed by induction on $k$. In the base case $k=1$, we have $\LD(A;\cF)=[0,\rho_{L_1}(A)]$ and it is easy to check that all of the required properties hold.

    Assume therefore that $k>1$. Let $D_1,\ldots,D_k,D$ be as defined  above. We verify each property in turn:
    \begin{enumerate}
        \item Addition by $a$ permutes the cosets $L_k/L_{k-1}$, so let $a_1',\ldots,a_m'$ be a permutation of $a_1,\ldots,a_m$ such that $a_j+a=a_j'+b_j$ for some $b_j\in L_{k-1}$. Let $D_j'=\LD(A+a+a_j;\cF\setminus L_k)$. By the induction hypothesis, $D_j'=\LD(A+a_j'+b_j;\cF\setminus L_k)=\LD(A+a_j';\cF\setminus L_k)$, so $D_1',\ldots,D_m'$ is a permutation of $D_1,\ldots,D_m$. After compression, it follows that $\LD(A;\cF)=\LD(A+a;\cF)$. 
        \item Each of the $D_j$ are $C_l$-compressed for $l=1,\ldots,k-1$. Thus, $D$ is $C_l$-compressed for $l=1,\ldots,k-1$ and, therefore, $\LD(A;\cF)=C_k(D)$ is $C_l$-compressed for $l=1,\ldots,k$.
        \item Let $D_j'=\LD(B+a_j;\cF\setminus L_k)$. By the induction hypothesis, $D_j'\subseteq D_j$, so the corresponding $D'$ satisfies $D'\subseteq D$. Therefore, $\LD(B;\cF)\subseteq \LD(A;\cF)$.
        \item By definition, $D_j=\LD(A+a_j;\cF\setminus L_k)$ and, by the induction hypothesis, we have $\rho_{L_{k-1}}(A+a_j)=\Vol(D_j)$. Therefore,
        \begin{align*}
            \Vol(\LD(A;\cF)) &=\Vol(D)
            = \frac{1}{m}\sum_{j=1}^m \Vol(D_j)
            = \frac{1}{m}\sum_{j=1}^m \rho_{L_{k-1}}(A+a_j)\\
            &= \frac{[L_k:L_{k-1}]}{m}\sum_{j=1}^m \rho_{L_k}((A+a_j)\cap L_{k-1})\\
            &= \sum_{j=1}^m \rho_{L_k}(A\cap (L_{k-1}-a_j))
            = \rho_{L_k}(A).
        \end{align*}
        \item We show by induction that $\LD(A;\cF)$ is an interior-disjoint union of boxes of the form
        $$v+[0,r]\times \sqb{0,\frac{1}{[L_2:L_1]}}\times\cdots\times \sqb{0,\frac{1}{[L_k:L_{k-1}]}},$$
        where $v$ is of the form
        $$\paren{0,\frac{m_2}{[L_2:L_1]},\ldots,\frac{m_k}{[L_k:L_{k-1}]}}$$
        with $m_2,\ldots,m_k$ non-negative integers. 
        The base case is trivial since $\LD(A;\cF)$ is an interval.
        
        By the induction hypothesis, each $D_j$ is an interior-disjoint union of boxes of the form
        $$v+[0,r]\times \sqb{0,\frac{1}{[L_2:L_1]}}\times\cdots\times \sqb{0,\frac{1}{[L_{k-1}:L_{k-2}]}}.$$
        Thus, $D$ is also the interior-disjoint union of boxes of the same kind and compressing preserves this property.

        Finally, since $\LD(A;\cF)$ is compressed, it is the finite union of boxes of the required form. \qedhere
    \end{enumerate}
\end{proof}

The next lemma fully determines $\LD(A;\cF)$ by giving a precise condition for when the lattice density contains any given point.

\begin{lem} \label{lem:ldpoint}
    Suppose $k\geq 2$, $r\in (0,1]$ is real and $m_2,\ldots,m_k$ are positive integers. Then the following are equivalent:
    \begin{enumerate}
        \item $\LD(A;\cF)$ contains the point
        $$\paren{r, \frac{m_2}{[L_2:L_1]},\frac{m_3}{[L_3:L_2]},\ldots,\frac{m_k}{[L_k:L_{k-1}]}}.$$
        \item For each $l=2,\ldots,k$ and $(i_l,i_{l+1},\ldots,i_k)\in [m_l]\times [m_{l+1}]\times\cdots \times [m_k]$, there exist $b_{i_l,\ldots,i_k}\in L_k$ such that:
        \begin{enumerate}
            \item For $l<k$, $b_{i_l,i_{l+1},\ldots,i_k}\in b_{i_{l+1},\ldots,i_k}+L_l$.
            \item $b_{i,i_{l+1},\ldots,i_k}-b_{j,i_{l+1},\ldots,i_k}\not\in L_{l-1}$ for each $i\neq j$ with $i,j\in [m_l]$.
            \item $\rho_{L_1}(A+b_{i_2,\ldots,i_k})\geq r$ for each $i_2,\ldots,i_k$.
        \end{enumerate}
    \end{enumerate}
\end{lem}
\begin{proof}
    We proceed by induction on $k$. Let $a_1,\ldots,a_m$ be any coset representatives of $L_k/L_{k-1}$ with $m=[L_k:L_{k-1}]$ and $D_i=\LD(A+a_i;\cF\setminus L_k)$. 

    \vspace{2mm}
    \noindent
    $1\Rightarrow 2$: From the construction of $\LD(A;\cF)$, $m_k$ of the $D_i$ contain the point 
    $$\paren{r, \frac{m_2}{[L_2:L_1]},\frac{m_3}{[L_3:L_2]},\ldots,\frac{m_{k-1}}{[L_{k-1}:L_{k-2}]}}.$$
    Without loss of generality, assume that they are $D_1,\ldots,D_{m_k}$. Set $b_i=a_i\in L_k$ for $i=1,\ldots,m_k$. Then $b_i-b_j\not\in L_{k-1}$ for $i\neq j$. 

    If $k=2$, then each $D_i$ with $i \in [m_k]$ contains $r$, meaning that $\rho_{L_1}(A+a_i)\geq r$. Thus, $\rho_{L_1}(A+b_i)\geq r$ for each $i \in [m_k]$, completing the proof of the base case.

    Now suppose that $k>2$. By the induction hypothesis applied to each $D_{i_k}$, there exist $b_{i_l,\ldots,i_k}'\in L_{k-1}$ for each $(i_l,i_{l+1},\ldots,i_k)\in [m_l]\times [m_{l+1}]\times\cdots \times [m_k]$ such that
    \begin{enumerate}[label=(\alph*)]
        \item For $l<k-1$, $b_{i_l,\ldots,i_k}'\in b_{i_{l+1},\ldots,i_k}'+L_l$.
        \item For $l<k$, $b_{i,i_{l+1},\ldots,i_k}'-b_{j,i_{l+1},\ldots,i_k}'\not\in L_{l-1}$ for each $i\neq j$ with $i,j\in [m_l]$.
        \item $\rho_{L_1}(A+b_{i_k}+b_{i_2,\ldots,i_k}')\geq r$ for each $i_2,\ldots,i_k$.
    \end{enumerate}
    Set $b_{i_l,\ldots,i_k}=b_{i_l,\ldots,i_k}'+b_{i_k}$. Then property (a) holds for $l<k-1$; property (b) holds for $l<k$ and property (c) holds. It remains to check that $b_{i_{k-1},i_k}\in b_{i_k}+L_{k-1}$ and $b_{i}-b_{j}\not\in L_{k-1}$ for each $i\neq j$. The former holds since $b_{i_{k-1},i_k}=b_{i_{k-1},i_k}'+b_{i_k}\in b_{i_k}+L_{k-1}$ and the latter was observed earlier.

    \vspace{2mm}
    \noindent
    $2\Leftarrow 1$: Since $a_1,\ldots,a_m$ are any coset representatives, we may pick $a_i=b_i$ for $i=1,\ldots,m_k$. 

    For $k=2$, since $\rho_{L_1}(A+b_i)\geq r$, $D_i$ contains $r$ for $i=1,\ldots,m_2$. Thus, $\LD(A;\cF)$ contains the point $(r,\frac{m_2}{|L_2/L_1|})$.

    Now assume $k>2$. Let $b_{i_l,\ldots,i_k}'=b_{i_l,\ldots,i_k}-b_{i_k}$. Then we have the following properties, inherited from the $b$:
    \begin{enumerate}[label=(\alph*)]
        \item For $l<k-1$, $b_{i_l,\ldots,i_k}'\in b_{i_{l+1},\ldots,i_k}'+L_l$.
        \item For $l<k$, $b_{i,i_{l+1},\ldots,i_k}'-b_{j,i_{l+1},\ldots,i_k}'\not\in L_{l-1}$ for each $i\neq j$ with $i,j\in [m_l]$.
        \item $\rho_{L_1}(A+b_{i_k}+b_{i_2,\ldots,i_k}')\geq r$ for each $i_2,\ldots,i_k$.
    \end{enumerate}
    By the induction hypothesis, for $i=1,\ldots,m_k$, $D_i$ contains the point
    $$\paren{r, \frac{m_2}{[L_2:L_1]},\frac{m_3}{[L_3:L_2]},\ldots,\frac{m_{k-1}}{[L_{k-1}:L_{k-2}]}}.$$
    Therefore, by the definition of $\LD(A;\cF)$, it contains the point
    $$\paren{r, \frac{m_2}{[L_2:L_1]},\frac{m_3}{[L_3:L_2]},\ldots,\frac{m_k}{[L_k:L_{k-1}]}},$$
    as required.
\end{proof}

As an application of this lemma, we now show how to compute the projections of lattice densities.

\begin{lem} \label{lem:ldproj}
    The following are true:
    \begin{enumerate}
        \item $\pi_1(\LD(A;\cF))$ is the interval $[0,r]$, where
            $$r=\max_{a\in L_k}\set{\rho_{L_1}(A+a)}.$$
            In particular, $\pi_1(\LD(A;\cF))$ depends only on $A$, $L_1$ and $L_k$.
        \item For $2\leq l\leq k$, $\pi_l(\LD(A;\cF))$ is the interval 
        $$\sqb{0,\frac{m}{[L_l:L_{l-1}]}},$$
        where $m\in \ZZ$ is the maximum number of elements $a_1,\ldots,a_m\in A\cap L_k$ such that $a_i-a_j\in L_l\setminus L_{l-1}$ for any $i\neq j$. In particular, $\pi_l(\LD(A;\cF))$ depends only on $A$, $L_{l-1}$, $L_l$ and $L_k$.
    \end{enumerate}
\end{lem}
\begin{proof}
    We first observe that the maxima are well-defined. Indeed, $\rho_{L_1}$ is invariant under translations by elements of $L_1$, so, for (1), we may take the maximum over the finitely many coset representatives of $L_k/L_1$. For (2), we see that each $a_i$ must belong to a different coset of $L_l/L_{l-1}$, so $m\leq [L_l:L_{l-1}]$.
    \begin{enumerate}
        \item If $\pi_1(\LD(A;\cF))=[0,r]$, then $\LD(A;\cF)$ contains the point
        $$\paren{r, \frac{1}{[L_2:L_1]},\frac{1}{[L_3:L_2]},\ldots,\frac{1}{[L_k:L_{k-1}]}}$$
        and $r$ is the maximum such real number. By Lemma~\ref{lem:ldpoint}, this is equivalent to the existence of some $b\in L_k$ such that $\rho_{L_1}(A+b)\geq r$. Thus,
        $$r=\max_{b\in L_k}\set{\rho_{L_1}(A+b)}.$$
        
        \item Suppose $\LD(A;\cF)$ contains the point
        $$\paren{r, \frac{1}{[L_2:L_1]},\ldots,\frac{m}{[L_{l+1}:L_l]},\ldots,\frac{1}{[L_k:L_{k-1}]}}$$
        for some $r>0$ and $m$ is the maximum such integer. By Lemma~\ref{lem:ldpoint}, this is equivalent to the existence of $b\in L_k$ and $b_1,\ldots,b_m\in b+L_l$ such that $b_i-b_j\not\in L_{l-1}$ for each $i\neq j$ and $\rho_{L_1}(A+b_i)\geq r$ for each $i$. Since we may take $r$ to be the minimum of $\rho_{L_1}(A+b_i)$ over all $i$, we are just requiring that $\rho_{L_1}(A+b_i)>0$, that is, $(A+b_i)\cap L_1\neq\emptyset$ for each $i$. 
        
        Suppose such $b,b_i$ exist. Let $a_i\in A$ be such that $a_i+b_i\in L_1$, which exists since $(A+b_i)\cap L_1\neq\emptyset$. Note that $a_i\in L_k$ since $a_i\in -b_i+L_1\subseteq L_k$. Moreover, for any $i\neq j$, $a_i-a_j\in b_j-b_i+L_1\subseteq L_l\setminus L_{l-1}$, as required.

        On the other hand, suppose we have $a_1,\ldots,a_m\in A\cap L_k$ such that $a_i-a_j\in L_l\setminus L_{l-1}$ for all $i\neq j$. Set $b=-a_1$ and $b_i=-a_i$ for each $i$. Then $b_i=b+a_1-a_i\in b+L_l$ and $b_i-b_j=a_j-a_i\not\in L_{l-1}$ for $i\neq j$. Finally, note that $(A+b_i)\cap L_1\neq\emptyset$ for each $i$, since it contains 0. \qedhere
    \end{enumerate}
\end{proof}

The next result, which again makes use of Lemma~\ref{lem:ldpoint}, describes lattice densities of sumsets.

\begin{thm} \label{thm:ldsum}
    Suppose $B\subseteq L$ is $d$-periodic. If $p=(p_1,\ldots,p_k)\in \LD(A;\cF)$ and $q=(q_1,\ldots,q_k)\in \LD(B;\cF)$, then 
    $$\max(p,q)\in \LD(A+B;\cF),$$
    where $\max(p,q)=(\max(p_1,q_1),\ldots,\max(p_k,q_k))$.
\end{thm}

\begin{proof}
    Since $\LD(A;\cF)$ and $\LD(B;\cF)$ are both unions of boxes of the form
    $$[0,r]\times \sqb{0,\frac{m_2}{[L_2:L_1]}}\times\cdots\times \sqb{0,\frac{m_k}{[L_k:L_{k-1}]}}$$
    for $r\in (0,1]$ and $m_2,\ldots,m_k$ positive integers, we may assume that $p,q$ are of the form
    $$p=\paren{r, \frac{m_2}{[L_2:L_1]},\frac{m_3}{[L_3:L_2]},\ldots,\frac{m_{k}}{[L_{k}:L_{k-1}]}},$$
    $$q=\paren{r', \frac{m_2'}{[L_2:L_1]},\frac{m_3'}{[L_3:L_2]},\ldots,\frac{m_{k}'}{[L_{k}:L_{k-1}]}}.$$
    Without loss of generality, we assume that $r\geq r'$. By Lemma~\ref{lem:ldpoint}, we obtain $b_{i_l,\ldots,i_k}, b_{i_l,\ldots,i_k}'\in L_k$ with the properties given in the lemma. Let $I=\setcond{i\in [2,k]}{m_i\geq m_i'}$ and $J=[2,k]\setminus I$. Set $c_{i_l,\ldots,i_k}=b_{i_l',\ldots,i_k'}+b_{i_l'',\ldots,i_k''}'$, where 
    $$i_j'=\begin{cases}
        i_j & \text{if } j\in I\\
        1 & \text{otherwise}
    \end{cases}\quad \text{and}\quad i_j''=\begin{cases}
        i_j & \text{if } j\in J\\
        1 & \text{otherwise}
    \end{cases}$$
    for $(i_l,\ldots,i_k)\in [\max(m_l,m_l')]\times\cdots\times [\max(m_k,m_k')]$. We wish to show that the $c_{i_l,\ldots,i_k}$ satisfy properties (a)--(c) in Lemma~\ref{lem:ldpoint} for $\LD(A+B;\cF)$. Note that we have $c_{i_l,\ldots,i_k}\in L_k$ since $b_{i_l',\ldots,i_k'},b_{i_l'',\ldots,i_k''}'\in L_k$. We now prove each of (a)--(c) in turn:
    \begin{enumerate}[label=(\alph*)]
        \item For $l<k$, we have $b_{i_l',i_{l+1}',\ldots,i_k'}\in b_{i_{l+1}',\ldots,i_k'}+L_l$ and $b_{i_l'',i_{l+1}'',\ldots,i_k''}'\in b_{i_{l+1}'',\ldots,i_k''}'+L_l$. Thus, $c_{i_l,i_{l+1},\ldots,i_k}\in c_{i_{l+1},\ldots,i_k}+L_l$.
        \item Suppose $l\in I$. Then, for $i\neq j$, $c_{i,i_{l+1},\ldots,i_k}-c_{j,i_{l+1},\ldots,i_k}=b_{i,i_{l+1}',\ldots,i_k'}-b_{j,i_{l+1}',\ldots,i_k'}\not\in L_{l-1}$. The case $l\in J$ is similar.
        \item We have $\rho_{L_1}(B+b_{i_2'',\ldots,i_k''}')\geq r'>0$. In particular, $B+b_{i_2'',\ldots,i_k''}'$ contains some element $x\in L_1$. Thus, $A+B+c_{i_2,\ldots,i_k}\supseteq A+b_{i_2',\ldots,i_k'}+x$, so we have $\rho_{L_1}(A+B+c_{i_2,\ldots,i_k})\geq \rho_{L_1}(A+b_{i_2',\ldots,i_k'}+x)=\rho_{L_1}(A+b_{i_2',\ldots,i_k'})\geq r$. \qedhere
    \end{enumerate}
\end{proof}

The final result of this subsection relates projections of lattice densities with respect to different flags.

\begin{lem} \label{lem:projineq}
    Suppose $\cF'=\set{L_1'\subseteq\cdots\subseteq L_{k-1}'\subseteq L_k}$ is a flag of full-rank sublattices of $L$. Then the following are true:
    \begin{enumerate}
        \item If $L_1'\subseteq L_1$, then
        $$|\pi_1(\LD(A;\cF))|\leq |\pi_1(\LD(A;\cF'))|.$$
        \item For $2\leq l\leq k$, if $L_l'=L_l$ and $L_{l-1}'\subseteq L_{l-1}$, then
        $$|\pi_l(\LD(A;\cF))|\geq |\pi_l(\LD(A;\cF'))|.$$
    \end{enumerate}
\end{lem}
\begin{proof}
    \begin{enumerate}
        \item Let
        $$r = \max_{a\in L_k}\set{\rho_{L_1}(A+a)} \quad\text{and}\quad r' = \max_{a\in L_k}\set{\rho_{L_1'}(A+a)}.$$
        By Lemma~\ref{lem:ldproj}, it suffices to show that $r\leq r'$. Suppose $r$ is attained by $a\in L_k$. Let $s=[L_1 : L_1']$ and $c_1,\ldots,c_s$ be coset representatives of $L_1/L_1'$. We can split $(A+a)\cap L_1$ into the disjoint union $\bigcup_{i=1}^s (A+a+c_i)\cap L_1'$, so that 
        $$\rho_{L_1}(A+a)=\frac{1}{s}\sum_{i=1}^s \rho_{L_1'}(A+a+c_i).$$
        Therefore, there is some $i$ such that $\rho_{L_1'}(A+a+c_i)\geq r$, so $r'\geq r$.

        \item Suppose $|\pi_l(\LD(A;\cF'))|=\frac{n}{[L_l':L_{l-1}']}$. By Lemma~\ref{lem:ldproj}, there are $b_1,\ldots,b_n\in A\cap L_k$ such that $b_i-b_j\in L_l'\setminus L_{l-1}'$. Let $s=[L_{l-1}:L_{l-1}']$. Define an equivalence relation by setting $b_i\sim b_j$ if $b_i-b_j\in L_{l-1}$. Then each equivalence class has at most $s$ elements, since no two elements belong to the same coset of $L_{l-1}/L_{l-1}'$. Let $a_1,\ldots,a_m$ be any representatives of the equivalence classes of $b_1,\ldots,b_n$, so that $ms\geq n$. Since the $a_i$ are in different equivalence classes, we have $a_i-a_j\not\in L_{l-1}$ for $i\neq j$. By Lemma~\ref{lem:ldproj} again, we have 
        $$|\pi_l(\LD(A;\cF))|\geq \frac{m}{[L_l:L_{l-1}]}=\frac{ms}{[L_l:L_{l-1}']}\geq \frac{n}{[L_l:L_{l-1}']}=|\pi_l(\LD(A;\cF'))|,$$
        as required. \qedhere
    \end{enumerate}
\end{proof}

\subsection{Local lattice densities}

In practice, we will make use of a local variant of lattice density. 
Intuitively, the local lattice density of $A$ at some point $x$ is the lattice density of a tiny region of $A$ around $x$. However, $A$ is a discrete set, so we cannot simply take an infinitesimally small ball around $x$. Instead, we define the local lattice density of $A$ in some small region $S\subset \RR^d$ to be the lattice density of a collection of copies of $A\cap S$, placed so as to be $d$-periodic. To make this work, we require that $S$ be tileable, which we now define. Note that we will continue to use notation from the previous subsection. In particular, $\cF=\set{L_1\subseteq \cdots\subseteq L_k}$ is a flag of full-rank sublattices of a lattice $L\cong \ZZ^d$. 

Let $L_{\RR}=L\otimes \RR\cong \RR^d$. We say that $S\subset L_{\RR}$ is \emph{tileable} if there is a sublattice $P\subseteq L_1$ of full rank such that $S\oplus P=L_{\RR}$. In this case, we say that $S$ is \emph{tiled} by $P$. For example, if $L=\ZZ^d$, the box $[0,M)^d\subset \RR^d$ is tileable as long as $M\ZZ^d\subseteq L_1$. In all of our applications, $S$ will be a half-open box of the form $[0,M)^d$ or an affine transformation of it.

Let $P\subseteq L_1$ and $S\subset L_\RR$ be such that $S$ is tiled by $P$. For any $A\subseteq L$, define the \emph{local lattice density} 
$$\LD_S(A;\cF):=\LD((A\cap S)+P; \cF),$$
noting that $(A\cap S)+P$ is $d$-periodic. Using Lemma~\ref{lem:ldpoint}, it is not hard to check that $\LD_S(A;\cF)$ is independent of the choice of $P$ as long as $P\subseteq L_1$.

Before moving on, we note some basic properties of these local lattice densities.

\begin{lem} \label{lem:localldchange}
    Let $S, T \subset L_{\RR}$ be tileable and $A\subseteq S\cap T\cap L$. Then
    $$\LD_S(A;\cF)=\Psi(\LD_T(A;\cF)),$$
    where $\Psi:\RR^d\to \RR^d$ is given by 
    $$(x_1,\ldots,x_k)\mapsto \paren{\frac{\Vol(T)}{\Vol(S)}x_1,x_2,\ldots,x_k}.$$
\end{lem}
\begin{proof}
    Suppose $T$ is tiled by $P\subseteq L_1$. Let $x=\paren{r,\frac{m_2}{[L_2:L_1]},\ldots,\frac{m_k}{[L_k:L_{k-1}]}}\in \LD_T(A;\cF)$. By Lemma~\ref{lem:ldpoint}, there exist $b_{i_l,\ldots,i_k}\in L_k$ satisfying the conditions in the lemma, one of which is that $\rho_{L_1}(A+P+b_{i_2,\ldots,i_k})\geq r$.

    Suppose $S$ is tiled by $Q\subseteq L_1$. Since $T\oplus P=L_\RR$, $\det(P)=\Vol(T)$ and, similarly, $\det(Q)=\Vol(S)$. Since $A+Q+b_{i_2,\ldots,i_k}$ is a union of translates of $A+b_{i_2,\ldots,i_k}$, one for each point of $Q$, its density within $L_1$, $\rho_{L_1}(A+Q+b_{i_2,\ldots,i_k})$, is inversely proportional to $\det(Q)$. In particular, $\rho_{L_1}(A+Q+b_{i_2,\ldots,i_k})\det(Q)=\rho_{L_1}(A+P+b_{i_2,\ldots,i_k})\det(P)$. Hence, 
    $$\rho_{L_1}(A+Q+b_{i_2,\ldots,i_k})=\frac{\det(P)}{\det(Q)}\rho_{L_1}(A+P+b_{i_2,\ldots,i_k})\geq \frac{\Vol(T)}{\Vol(S)}r.$$
    Therefore, by Lemma~\ref{lem:ldpoint}, $\Psi(x)\in \LD_S(A;\cF)$, so we have $\LD_S(A;\cF)\supseteq \Psi(\LD_T(A;\cF))$. The converse follows similarly.
\end{proof}

\begin{lem} \label{lem:sublocalld}
    Let $S,T\subset L_\RR$ be tileable with $T\subseteq S$. Then, for $2\leq l\leq k$,
    $$|\pi_l(\LD_T(A;\cF))|\leq |\pi_l(\LD_S(A;\cF))|.$$
\end{lem}
\begin{proof}
    By Lemma~\ref{lem:localldchange},
    \begin{align*}
        |\pi_l(\LD_T(A;\cF))| &= |\pi_l(\LD_T(A\cap T;\cF))|\\
        & = |\pi_l(\LD_S(A\cap T;\cF))|\\
        &\leq |\pi_l(\LD_S(A;\cF))|,
    \end{align*}
    as required.
\end{proof}

\begin{lem} \label{lem:localldvol}
    Let $S\subset L_\RR$ be tileable and $A\subseteq L_k$. Then
    $$\frac{|A\cap S|}{|L_k\cap S|}=\Vol(\LD_S(A;\cF)).$$
\end{lem}
\begin{proof}
    Suppose $S$ is tiled by $P\subseteq L_1$. Then
    \begin{align*}
        \Vol(\LD_S(A;\cF)) &= \Vol(\LD((A\cap S)+P;\cF)) = \rho_{L_k}((A\cap S)+P) = \frac{|A\cap S|}{|L_k\cap S|},
    \end{align*}
    as required.
\end{proof}

\section{Families of flags} \label{sec:families}

In this section, we construct flags such that ``the projection $\pi_{l+1}$ of the lattice density is preserved under multiplication by $\lambda_l$". 
More precisely, we want to find a flag $\cF$ in $\frakD_{\lambda_1,\ldots,\lambda_k;K}$ and a flag $\cG$ in $\cO_K$ such that, for any $d$-periodic $A\subseteq \frakD_{\lambda_1,\ldots,\lambda_k;K}$,
\begin{equation} \label{eqn:ldproj}
    \pi_{l+1}(\LD(A;\cF))\subseteq \pi_{l+1}(\LD(\lambda_l\cdot A;\cG))
\end{equation}
for $l=0,1,\ldots,k$. We can find such flags for each $l$, but, unfortunately, it may not be possible to find $\cF,\cG$ that work simultaneously for all $l$. To overcome this, we construct families of flags $\cF_{\vec{n}},\cG_{\vec{n}}$ and show that for $\vec{n}$ ``sufficiently large'' these pairs satisfy (\ref{eqn:ldproj}) approximately for all $l$. 

\subsection{Algebraic families of flags} \label{sec:algfamily}

Recall that $\lambda_1,\ldots,\lambda_k\in K=\QQ(\lambda_1,\ldots,\lambda_k)$ and $d=\deg(K/\QQ)$. Let $\fraka_l$ be the ideal $\cO_K\cap \lambda_1^{-1}\cO_K\cap\cdots\cap \lambda_l^{-1}\cO_K$ for $l=0,1,\ldots,k$. In particular, $\fraka_k=\frakD_{\lambda_1,\ldots,\lambda_k;K}$. Then $\fraka_l^{-1}$ is the fractional ideal $\cO_K+\lambda_1 \cO_K+\dots + \lambda_l\cO_K$. We also have $\cO_K=\fraka_0\mid \fraka_1\mid\cdots\mid \fraka_k$. Let $\frakb_l\subseteq \cO_K$ be the ideal such that $\fraka_l=\frakb_l\fraka_{l-1}$ for each $l=1,\ldots,k$. For each $\vec{n}=(n_1,\ldots,n_k)\in \ZZ_{\geq 0}^k$ and $l=0,1,\ldots,k$, let $\frakc_{\vec{n},l}=\frakb_{l+1}^{n_{l+1}}\cdots \frakb_k^{n_k}$. Define two flags of lattices by
\begin{align*}
    \cF_{\vec{n}}^K &:= \set{\fraka_k\frakc_{\vec{n},0}\subseteq \fraka_k\frakc_{\vec{n},1}\subseteq \cdots \subseteq \fraka_k\frakc_{\vec{n},k-1}\subseteq \fraka_k},\\
    \cG_{\vec{n}}^K &:= \set{\frakc_{\vec{n},0}\subseteq \frakc_{\vec{n},1}\subseteq \cdots \subseteq \frakc_{\vec{n},k-1}\subseteq \cO_K}.
\end{align*}

These families of flags will serve as candidates for satisfying (\ref{eqn:ldproj}). The following two lemmas make this precise. 
Note that for any two vectors $\vec{n}, \vec{m}\in \ZZ^k$, we write $\vec{n}\geq \vec{m}$ if $n_i\geq m_i$ for all $i$. We also write $\vec{n}+c$ to denote the vector $(n_1+c,\ldots,n_k+c)$.

\begin{lem} \label{lem:pi1stable}
    Let $A\subseteq \fraka_k$ be $d$-periodic. Then, for any $\vec{n}\geq 0$,
    $$\pi_1(\LD(A;\cF_{\vec{n}}^K))=\pi_1(\LD(A;\cG_{\vec{n}+1}^K)).$$
\end{lem}
\begin{proof}
    Let
    $$r = \max_{a\in \fraka_k}\set{\rho_{\fraka_k\frakc_{\vec{n},0}}(A+a)} \quad\text{and}\quad r' = \max_{a\in \cO_K}\set{\rho_{\fraka_k\frakc_{\vec{n},0}}(A+a)}.$$
    Note that $\frakb_1\cdots \frakb_k=\fraka_k$, so that $\frakc_{\vec{n}+1,0}=\fraka_k\frakc_{\vec{n},0}$. By Lemma~\ref{lem:ldproj}(1), it suffices to show that $r=r'$. Since $\fraka_k \subseteq \cO_K$, we clearly have $r'\geq r$. To see that $r\geq r'$, observe that, since $A\subseteq \fraka_k$, $(A+a)\cap \fraka_k=\emptyset$ for any $a\in \cO_K\setminus \fraka_k$. In particular, $\rho_{\fraka_k\frakc_{\vec{n},0}}(A+a)=0$.
\end{proof}

For the next lemma, recall that $\cM_l:K\to K$ is the $\QQ$-linear map corresponding to  multiplication by $\lambda_l$ and each $\cM_l$ restricts to the map $\fraka_k\to \cO_K$. 

\begin{lem} \label{lem:pilstable}
    Let $A\subseteq \fraka_k$ be $d$-periodic and $l\in [k]$. Then, for $\vec{n},\vec{m}\geq 0$ with $m_i=n_i+1$ for $i=l+1,l+2,\ldots,k$ and $m_l=n_l$, 
    $$|\pi_{l+1}(\LD(A;\cF_{\vec{n}}^K))|\leq |\pi_{l+1}(\LD(\cM_l A; \cG_{\vec{m}}^K))|.$$
\end{lem}
\begin{proof}
    Let $r$ be the maximum number of elements $a_1,\ldots,a_r\in A$ such that $a_i-a_j\in \fraka_k\frakc_{\vec{n},l}\setminus \fraka_k\frakc_{\vec{n},l-1}$ for $i\neq j$. 
    Then, by Lemma~\ref{lem:ldproj}(2), $|\pi_{l+1}(\LD(A;\cF_{\vec{n}}^K))|=r/[\fraka_k\frakc_{\vec{n},l}:\fraka_k\frakc_{\vec{n},l-1}]$.
    Since $m_l=n_l$, we have $[\fraka_k\frakc_{\vec{n},l}:\fraka_k\frakc_{\vec{n},l-1}]=[\frakc_{\vec{n},l}:\frakc_{\vec{n},l-1}]=N_{K/\QQ}(\frakb_l^{n_l})=[\frakc_{\vec{m},l}:\frakc_{\vec{m},l-1}]$. By Lemma~\ref{lem:ldproj}(2) applied to $|\pi_{l+1}(\LD(\cM_l A; \cG_{\vec{m}}^K))|$, it suffices to find $b_1,\ldots,b_r\in \cM_l A=\lambda_l\cdot A$ such that $b_i-b_j\in \frakc_{\vec{m},l}\setminus \frakc_{\vec{m},l-1}$ for $i \ne j$.
    
    Set $b_i=\lambda_l a_i$, so it is clear that $b_i\in \lambda_l\cdot A$. It suffices to show that:
    \begin{enumerate}[label=(\alph*)]
        \item $b_i-b_j\in \frakc_{\vec{m},l}$,
        \item $b_i-b_j\not\in \frakc_{\vec{m},l-1}$ for $i\neq j$.
    \end{enumerate}
    
    For (a), observe that $\lambda_l \fraka_k\frakc_{\vec{n},l}\subseteq \fraka_l^{-1}\fraka_k\frakc_{\vec{n},l}=\frakb_{l+1}\cdots\frakb_k\frakc_{\vec{n},l} 
    =\frakc_{\vec{m},l}$. Thus, $b_i-b_j=\lambda_l(a_i-a_j)\in \lambda_l \fraka_k\frakc_{\vec{n},l}\subseteq \frakc_{\vec{m},l}$. 
    
    For (b), suppose that $b_i-b_j\in \frakc_{\vec{m},l-1}$ for some $i\neq j$. Then $a_i-a_j\in \lambda_l^{-1}\frakc_{\vec{m},l-1}$. On the other hand, $a_i-a_j\in \fraka_k\frakc_{\vec{n},l}$. Together, we have $a_i-a_j\in \lambda_l^{-1}\frakc_{\vec{m},l-1}\cap \fraka_k\frakc_{\vec{n},l}$. 

    We claim that $\lambda_l^{-1}\frakc_{\vec{m},l-1}\cap \fraka_k\frakc_{\vec{n},l}\subseteq \fraka_k\frakc_{\vec{n},l-1}$, which will lead to a contradiction, since $a_i-a_j\not\in \fraka_k\frakc_{\vec{n},l-1}$. We prove the claim by proving it locally at every prime ideal $\frakp\subseteq \cO_K$, that is, we will show that $\nu_{\frakp}(\lambda_l^{-1}\frakc_{\vec{m},l-1}\cap \fraka_k\frakc_{\vec{n},l})\geq \nu_{\frakp}(\fraka_k\frakc_{\vec{n},l-1})$.

    Recall that $\fraka_l=\fraka_{l-1}\cap \lambda_l^{-1}\cO_K$, so $\nu_{\frakp}(\fraka_l)=\max(\nu_{\frakp}(\fraka_{l-1}), \nu_{\frakp}(\lambda_l^{-1}))$, which implies that $\nu_{\frakp}(\frakb_l)=\max(0,\nu_{\frakp}(\lambda_l^{-1})-\nu_{\frakp}(\fraka_{l-1}))$. We have 
    \begin{align*}
        \nu_{\frakp}(\lambda_l^{-1}\frakc_{\vec{m},l-1}\cap \fraka_k\frakc_{\vec{n},l}) &= \nu_{\frakp}(\lambda_l^{-1}\fraka_k\fraka_l^{-1}\frakc_{\vec{n},l-1}\cap \fraka_k\frakb_l^{-n_l}\frakc_{\vec{n},l-1})\\
        &= \nu_{\frakp}(\fraka_k\frakc_{\vec{n},l-1}) + \max(\nu_{\frakp}(\lambda_l^{-1})-\nu_{\frakp}(\fraka_l), -n_l\nu_{\frakp}(\frakb_l)).
    \end{align*}
    If $\nu_{\frakp}(\fraka_{l-1})\geq \nu_{\frakp}(\lambda_l^{-1})$, then $\nu_{\frakp}(\frakb_l)=0$. Otherwise, $\nu_{\frakp}(\fraka_{l})=\nu_{\frakp}(\lambda_l^{-1})$. In either case, $\max(\nu_{\frakp}(\lambda_l^{-1})-\nu_{\frakp}(\fraka_l), -\nu_{\frakp}(\frakb_l))\geq 0$, proving the claim and the lemma.
\end{proof}

Unfortunately, there are no pairs of flags $\cF_{\vec{n}}^K$ and $\cG_{\vec{m}}^K$ that simultaneously satisfy Lemmas~\ref{lem:pi1stable} and $\ref{lem:pilstable}$ for all $l$. 
Indeed, in order for $\pi_1(\LD(A;\cF_{\vec{n}}^K))=\pi_1(\LD(A;\cG_{\vec{m}}^K))$ and $|\pi_{l+1}(\LD(A;\cF_{\vec{n}}^K))|\leq |\pi_{l+1}(\LD(\cM_l A; \cG_{\vec{m}}^K))|$ to hold for all $l$ via the lemmas, we would require that $m_l=n_l$ and $m_l=n_l+1$ simultaneously. 
To overcome this, in the next subsection, we will show that for $\vec{n}$ sufficiently large the projections of the lattice densities stabilise, so we may use $\cF_{\vec{n}}^K$ and $\cG_{\vec{n}}^K$. This seems to suggest that, as $\vec{n}$ tends to infinity, the lattice densities $\LD(A;\cF_{\vec{n}})$ themselves converge as compact subsets. However, we make no attempt to formally prove this, since all we require is that their projections converge.

\subsection{Regularity}

For this subsection, we consider a more general setup, where we have, 
for each $\vec{n}=(n_1,\ldots,n_k)\in \NN^k$, two flags
\begin{align*}
    \cF_{\vec{n}} &= \set{L_{\vec{n},1}\subseteq L_{\vec{n},2}\subseteq\cdots \subseteq L_{\vec{n},k}\subseteq \ZZ^d},\\
    \cG_{\vec{n}} &=\set{M_{\vec{n},1}\subseteq M_{\vec{n},2}\subseteq\cdots \subseteq M_{\vec{n},k}\subseteq \ZZ^d},
\end{align*}
where $L_{\vec{n},l}$ depends only on $n_l,n_{l+1},\ldots,n_k$ and $L_{\vec{n},l}\subseteq L_{\vec{n}',l}$ if $\vec{n}\geq \vec{n}'$ and  similarly for $M_{\vec{n},l}$. We also fix a set $A\subseteq \ZZ^d$.

For a positive integer $R$, an \emph{$R$-cube} is a set that comes from taking the set $[0,R)^d\subset \RR^d$ and shifting it by an element of $R\ZZ^d$. 
Let $P$ be an $R$-cube for some $R$. For natural numbers $M,n_l,n_{l+1},\ldots,n_k$ with $M>0$ and a real number $\delta>0$, we say that $P$ is \emph{$(M,\delta,n_l,\ldots,n_k)$-regular} if each of the $M^d$ different $R/M$-subcubes $Q$ of $P$ satisfies
\begin{equation} \label{eqn:regularity}
    |\pi_{l+1}(\LD_{Q}(A;\cF_{\vec{n}'}))|\geq (1-\delta)|\pi_{l+1}(\LD_{P}(A;\cF_{\vec{n}}))|,
\end{equation}
where $\vec{n}=(0,\ldots,0,n_l,\ldots,n_k)$ and $\vec{n}'=(0,\ldots,0,n_l+1,\ldots,n_k)$. 

\begin{rmk}
    Here we are implicitly assuming that $R/M$ is an integer. Throughout the remainder of the paper, whenever we mention a local density $\LD_P(A;\cF)$, we will assume that $P$ is tileable. In particular, this means that $R$ and $R/M$ will always be multiples of every bounded number, so that the lattices $R\ZZ^d$ and $(R/M)\ZZ^d$ are contained in $L_{\vec{n},1}$. In practice, we will only be considering $\cF_{\vec{n}}$ where $\vec{n}$ is bounded and $(N/M)$-cubes where $M$ is bounded and $N$ can be taken to be a multiple of a sufficiently large integer.
\end{rmk}

By Lemmas~\ref{lem:projineq} and \ref{lem:sublocalld}, we always have 
$$|\pi_{l+1}(\LD_{Q}(A;\cF_{\vec{n}'}))|\leq |\pi_{l+1}(\LD_{Q}(A;\cF_{\vec{n}}))|\leq |\pi_{l+1}(\LD_{P}(A;\cF_{\vec{n}}))|,$$
so regularity says that both inequalities are close to equalities. In other words, our notion of regularity really encompasses two different types of regularity. The first is that the size of the projection $\pi_{l+1}$ does not change much when we replace $\vec{n}$ with $\vec{n}'$. The second is that the local lattice density does not change much when we shrink the local region from $P$ to $Q$. Note that in the definition of regularity, we may replace $\vec{n},\vec{n}'$ with $\vec{n}=(*,\ldots,*,n_l,\ldots,n_k)$ and $\vec{n}'=(*,\ldots,*,n_l+1,\ldots,n_k)$, where the $*$'s could be any (possibly distinct) natural numbers, since that does not change the relevant projection of the lattice density. 

Before proving our main result on regularity, we note some simple consequences of the definition.

\begin{lem} \label{lem:squeezereg}
    Let $M_1,M_2$ be positive integers and $P$ be an $(M_1M_2,\delta,n_l,\ldots,n_k)$-regular $R$-cube. Then the following hold: 
    \begin{enumerate}
        \item $P$ is $(M_1,\delta,n_l,\ldots,n_k)$-regular.
        \item For any $R/M_1$-subcube $Q$ of $P$, $Q$ is $(M_2,\delta,n_l,\ldots,n_k)$-regular.
    \end{enumerate}
\end{lem}
\begin{proof}
    Let $Q$ be any $R/M_1$-subcube of $P$ and $S$ be any $R/(M_1M_2)$-subcube of $Q$. By regularity, we have 
    $$|\pi_{l+1}(\LD_{S}(A;\cF_{\vec{n}'}))|\geq (1-\delta)|\pi_{l+1}(\LD_{P}(A;\cF_{\vec{n}}))|.$$
    By Lemma~\ref{lem:sublocalld}, we have $|\pi_{l+1}(\LD_{P}(A;\cF_{\vec{n}}))|\geq |\pi_{l+1}(\LD_{Q}(A;\cF_{\vec{n}}))|$ and $|\pi_{l+1}(\LD_{Q}(A;\cF_{\vec{n}'}))|\geq |\pi_{l+1}(\LD_{S}(A;\cF_{\vec{n}'}))|$. Therefore, 
    \begin{align*}
        |\pi_{l+1}(\LD_{Q}(A;\cF_{\vec{n}'}))| &\geq (1-\delta)|\pi_{l+1}(\LD_{P}(A;\cF_{\vec{n}}))|,\\
        |\pi_{l+1}(\LD_{S}(A;\cF_{\vec{n}'}))| &\geq (1-\delta)|\pi_{l+1}(\LD_{Q}(A;\cF_{\vec{n}}))|,
    \end{align*}
    which prove the first and second parts of the lemma, respectively.
\end{proof}

We now come to our main result on regularity, which says that, for any dense $A\subseteq [0,N)^d$, one can cut the box $[0,N)^d$ into a bounded number of subcubes, most of which are regular and where the union of the regular subcubes covers most of $A$. We first prove such a result with respect to a single projection $\pi_{l+1}$, before iterating to establish regularity with respect to all projections.

\begin{lem} \label{lem:regularity}
    Fix $\varepsilon,\delta>0$ and $l\in [k]$, a positive integer $M$ and non-negative integers $n_{l+1},\ldots,n_k$. Then there exists $R_0=R_0(M,\varepsilon,\delta)$ such that if $A\subseteq [0,N)^d$ is of size at least $\varepsilon N^d$ and $N'\mid N$, there exists a natural number $r\leq R_0$ and a collection $\cP$ of disjoint $N'/M^r$-cubes such that, for $A'=A\cap \bigcup_{P\in \cP} P$, 
    \begin{enumerate}
        \item $|A'|\geq (1-\delta)|A|$,
        \item $P$ is $(M,\delta,r,n_{l+1},\ldots,n_k)$-regular for all $P\in \cP$.
    \end{enumerate}
\end{lem}

\begin{proof}
    Let $\cP^{(r)}$ be the collection of $N'/M^r$-cubes in $[0,N)^d$,  $\cP^{(r)}_0$ be the subcollection of all $(M,\delta,r,n_{l+1},\ldots,n_k)$-regular cubes in $\cP^{(r)}$ and $A^{(r)}=A\cap \bigcup_{P\in \cP_0^{(r)}} P$. We will set $A'=A^{(r)}$ and $\cP=\cP^{(r)}_0$, so we wish to show that there is some bounded $r$ such that $|A^{(r)}|\geq (1-\delta)|A|$.

    Let $\cP^{(r)}_1$ be the collection of all cubes in $\cP^{(r)}$ which are not $(M,\delta,r,n_{l+1},\ldots,n_k)$-regular. Writing $\vec{n}^{(r)}=(0,\ldots,0,r,n_{l+1},\ldots,n_k)$, consider the quantity
    $$D_r:=\frac{(N'/N)^d}{M^{rd}}\sum_{P\in \cP^{(r)}} |\pi_{l+1}(\LD_{P}(A;\cF_{\vec{n}^{(r)}}))|\leq \frac{(N'/N)^d}{M^{rd}}|\cP^{(r)}|= 1.$$
    For any $P\in \cP^{(r)}$ and subcube $Q\in \cP^{(r+1)}$, we have the inequalities
    $$|\pi_{l+1}(\LD_{P}(A;\cF_{\vec{n}^{(r)}}))|\geq |\pi_{l+1}(\LD_{P}(A;\cF_{\vec{n}^{(r+1)}}))|\geq |\pi_{l+1}(\LD_{Q}(A;\cF_{\vec{n}^{(r+1)}}))|.$$
    Therefore, $D_r$ is decreasing in $r$. 
    
    Set $R_0:=\frac{M^d}{\varepsilon\delta^2}$. Since $D_r$ is decreasing and in $[0,1]$, there is some $r\leq R_0$ such that $D_r\geq D_{r+1}\geq D_{r}-\frac{\varepsilon\delta^2}{M^d}$. For each $P\in \cP^{(r)}_1$, since $P$ is not regular, there is some subcube $Q\in \cP^{(r+1)}$ of $P$ such that 
    $$|\pi_{l+1}(\LD_{Q}(A;\cF_{\vec{n}^{(r+1)}}))|\leq (1-\delta)|\pi_{l+1}(\LD_{P}(A;\cF_{\vec{n}^{(r)}}))|.$$
    Therefore,
    \begin{align}
        D_r-D_{r+1} &= \frac{(N'/N)^d}{M^{rd}}\sum_{P\in \cP^{(r)}}\paren{|\pi_{l+1}(\LD_{P}(A;\cF_{\vec{n}^{(r)}}))|-\frac{1}{M^d}\sum_{\substack{Q\in \cP^{(r+1)}\\ Q\subset P}}|\pi_{l+1}(\LD_{Q}(A;\cF_{\vec{n}^{(r+1)}}))|} \nonumber\\
        &\geq \frac{(N'/N)^d}{M^{rd}}\sum_{P\in \cP^{(r)}_1}\frac{\delta}{M^d}|\pi_{l+1}(\LD_{P}(A;\cF_{\vec{n}^{(r)}}))| \nonumber\\
        &= \frac{(N'/N)^d\delta}{M^{(r+1)d}}\sum_{P\in \cP^{(r)}_1} |\pi_{l+1}(\LD_{P}(A;\cF_{\vec{n}^{(r)}}))|. \label{eqn:sumP1_1}
    \end{align}
    By Lemma~\ref{lem:localldvol}, for any $P\in \cP^{(r)}_1$, we have
    $$\Vol(\LD_{P}(A;\cF_{\vec{n}^{(r)}}))=\frac{|A\cap P|}{|\ZZ^d\cap P|}=\frac{M^{rd}}{N'^d}|A\cap P|.$$
    Therefore, 
    \begin{align}
        |A\setminus A^{(r)}| &= \sum_{P\in \cP^{(r)}_1} |A\cap P| = \frac{N'^d}{M^{rd}}\sum_{P\in \cP^{(r)}_1} \Vol(\LD_{P}(A;\cF_{\vec{n}^{(r)}})) \nonumber\\
        &\leq \frac{N'^d}{M^{rd}}\sum_{P\in \cP^{(r)}_1} |\pi_{l+1}(\LD_{P}(A;\cF_{\vec{n}^{(r)}}))|. \label{eqn:sumP1_2}
    \end{align}
    Combining \eqref{eqn:sumP1_1} and \eqref{eqn:sumP1_2}, we have 
    \begin{align*}
        |A\setminus A^{(r)}| &\leq \frac{N^dM^d}{\delta}(D_r-D_{r+1})\leq \varepsilon\delta N^d\leq \delta|A|,
    \end{align*}
    as required.
\end{proof}

\begin{lem} \label{lem:bigregularity}
    Fix $\varepsilon,\delta>0$ and a positive integer $M$ and suppose that $A\subseteq [0,N)^d$ is of size at least $\varepsilon N^d$. Then there exist $n_1,\ldots,n_k$, $r\leq R_1=R_1(M,\varepsilon,\delta)$ and a collection $\cP$ of disjoint $N/M^r$-cubes such that, for $A'=A\cap \bigcup_{P\in \cP} P$,
    \begin{enumerate}
        \item $|A'|\geq (1-\delta)|A|$,
        \item $P$ is $(M,\delta,n_l,\ldots,n_k)$-regular for all $P\in \cP$ and $l\in [k]$.
    \end{enumerate}
    
\end{lem}
\begin{proof}
    Following the notation of Lemma~\ref{lem:regularity}, set  $S_1=R_0(M,\varepsilon/2,\delta/k)$ and, for $l=2,\ldots,k$,  
    $$S_l=R_0(M^{S_1+\cdots+S_{l-1}+1},\varepsilon/2, \delta/k).$$
    We then set $R_1:=S_1+\cdots+S_k$. We shall apply Lemma~\ref{lem:regularity} $k$ times in succession to obtain $n_k,n_{k-1},\ldots,n_1\leq R_1$. 
    
    First, we obtain $n_k\leq S_k$ and a collection $\cP^{(k)}$ of disjoint $N/M^{n_k}$-cubes such that, for $A^{(k)}=A\cap \bigcup_{P\in \cP^{(k)}} P$, we have
    \begin{enumerate}
        \item $|A^{(k)}|\geq \paren{1-\frac{\delta}{k}}|A|$,
        \item $P$ is $(M^{S_1+\cdots+S_{k-1}+1},\delta,n_k)$-regular for all $P\in \cP^{(k)}$.
    \end{enumerate}
    
    Suppose we have constructed $n_k,n_{k-1},\ldots,n_{l+1}$ for some $l\geq 1$. Then, using Lemma~\ref{lem:regularity}, we obtain $n_l\leq S_l$ and a collection $\cP^{(l)}$ of disjoint $N/M^{n_k+\cdots+n_l}$-cubes such that, for $A^{(l)}=A^{(l+1)}\cap \bigcup_{P\in \cP^{(l)}} P$, we have
    \begin{enumerate}
        \item $|A^{(l)}|\geq \paren{1-\frac{\delta}{k}}|A^{(l+1)}|$,
        \item $P$ is $(M^{S_1+\cdots+S_{l-1}+1},\delta,n_l,\ldots,n_k)$-regular for all $P\in \cP^{(l)}$.
    \end{enumerate}
    We may also assume that the collection $\cP^{(l)}$ is a subset of a refinement of $\cP^{(l+1)}$.

    Finally, set $\cP=\cP^{(1)}$, a collection of $N/M^r$-cubes, where $r=n_1+\cdots+n_k\leq R_1$. Then, for $A'=A\cap \bigcup_{P\in \cP} P$, we have
    \begin{enumerate}
        \item $|A'|\geq \paren{1-\frac{\delta}{k}}^k |A|\geq (1-\delta)|A|$,
        \item for each $l\in [k]$ and each $P\in \cP$, $P$ is a subcube of some $P^{(l)}\in \cP^{(l)}$, which is, by construction,  $(M^{S_1+\cdots+S_{l-1}+1},\delta,n_l,\ldots,n_k)$-regular. But then, by Lemma~\ref{lem:squeezereg}, $P$ is $(M,\delta,n_l,\ldots,n_k)$-regular. \qedhere
    \end{enumerate}
\end{proof}

\section{Proof of the dense case} \label{sec:dense}

In this section, we make use of the results of the last two sections to prove Lemma~\ref{lem:dense}, which we restate for the reader's convenience. As noted in Section~\ref{sec:reduction}, this will complete the proof of our main result. Recall, from Section~\ref{sec:mapZd}, that we have isomorphisms $\Phi':\frakD \to \ZZ^d$ and $\Phi:\cO_K\to \ZZ^d$. Multiplication of the elements of $\frakD$ by $\lambda_l$ then corresponds to the map $\cL_l:\ZZ^d\to \ZZ^d$ given by $\cL_l=\Phi\circ \cM_l\circ \Phi'^{-1}$. In particular, one may check that $|\det \cL_0|=N_{K/\QQ}(\frakD)$.

\begin{lem} \label{lem:dense2}
    For any $\varepsilon>0$, there exists $N_0$ such that if $N \ge N_0$  and $A\subset [0,N)^d$ with $|A|\geq \varepsilon N^d$, then
    $$|\cL_0 A+\cdots+\cL_k A|\geq H(\lambda_1,\ldots,\lambda_k)|A|-o_{\varepsilon}(|A|).$$
\end{lem}

\begin{proof}
Suppose $A\subseteq [0,N)^d$ with $|A|\geq \varepsilon N^d$. Let $\delta>0$ be arbitrary, $D$ be a large integer and $M$ be a sufficiently large multiple of $D$. $M$ will depend on both $\varepsilon$ and $\delta$, but not on $N$, which is assumed to be very large. By Lemma~\ref{lem:bigregularity}, there are bounded $n_1,\ldots,n_k,r$ and a collection $\cP$ of disjoint $N/M^r$-cubes such that, for $A'=A\cap \bigcup_{P\in \cP} P$, we have
\begin{enumerate}
    \item $|A'|\geq (1-\delta)|A|$,
    \item $P$ is $(M^2,\delta,n_l,\ldots,n_k)$-regular for all $P\in \cP$ and $l\in [k]$.
\end{enumerate}

Let $\cQ$ be the collection of $N/M^{r+1}$-cubes $Q$ such that $Q\subset P$ for some $P\in \cP$ and $Q$ is at least at a distance of $DN/M^{r+1}$ away from the boundary of $P$. In particular, $|\cQ|=(M-2D)^d|\cP|$. By Lemma~\ref{lem:squeezereg}, each $Q$ is $(M,\delta,n_l,\ldots,n_k)$-regular for all $l\in [k]$. Set $A''=A\cap \bigcup_{Q\in \cQ} Q$. Then $A'\setminus A''$ consists of points covered by $\cP$ but not $\cQ$, so
\begin{align*}
    |A'\setminus A''| &\leq \paren{1-\paren{\frac{M-2D}{M}}^d} N^d \leq \varepsilon^{-1}\paren{1-\paren{\frac{M-2D}{M}}^d}|A|\\
    &\leq \frac{2Dd}{M\varepsilon}|A| \leq \delta |A|
\end{align*}
for $M\geq 2Dd/\delta\varepsilon$. It follows that $|A''|\geq (1-2\delta)|A|$. Let $\cQ_0$ be the collection of all $N/M^{r+1}$-cubes, including those outside $[0,N)^d$. For $Q\in \cQ_0$, denote by $Q^+$ the slightly expanded cube $Q+[-\frac{DN}{M^{r+2}},\frac{DN}{M^{r+2}}]^d$. Then, for $M$ sufficiently large ($M\geq 4Dd/\delta$ suffices), 
\begin{equation} \label{eq:vol}
\Vol(Q^+)=\paren{1+\frac{2D}{M}}^d\Vol(Q)\leq (1+\delta)\Vol(Q).
\end{equation}

Let $\cF_{\vec{n}}^K,\cG_{\vec{n}}^K$ be the families of flags of sublattices of $\frakD$ and $\cO_K$ defined in Section~\ref{sec:algfamily}. Under the isomorphisms $\Phi,\Phi'$, these families translate to families $\cF_{\vec{n}},\cG_{\vec{n}}$ in $\ZZ^d$ given by $\cF_{\vec{n}}:=\Phi'(\cF_{\vec{n}}^K)$ and $\cG_{\vec{n}}:=\Phi(\cG_{\vec{n}}^K)$. By Lemmas~\ref{lem:pi1stable} and \ref{lem:pilstable}, we have the following two properties:
\begin{enumerate}
    \item For $d$-periodic $A\subseteq \ZZ^d$,
    \begin{equation}
        \pi_1(\LD(A;\cF_{\vec{n}}))=\pi_1(\LD(\cL_0 A;\cG_{\vec{n}+1})). \label{eqn:p1}
    \end{equation}
    \item For $d$-periodic $A\subseteq \ZZ^d$ and $l\in [k]$, 
    \begin{equation}
        |\pi_{l+1}(\LD(A;\cF_{\vec{n}}))|\leq |\pi_{l+1}(\LD(\cL_l A;\cG_{\vec{m}}))| \label{eqn:p2}
    \end{equation}
    if $m_i=n_i+1$ for $i=l+1,\ldots,k$ and $m_l=n_l$.
\end{enumerate}

Define the bodies $X,Y\subset \RR^{d+k+1}=\RR^d\times \RR^{k+1}$ by
\begin{align*}
    X &:= \bigcup_{Q\in \cQ} (Q\times \LD_Q(A;\cF_{\vec{n}})),\\
    Y &:= \bigcup_{Q\in \cQ_0} (\cL_0 Q\times (1+2\delta)\LD_{\cL_0 (Q^+)}(\cL_0 A+\cdots+\cL_k A;\cG_{\vec{n}+1})).
\end{align*}
We remark that in order for $\LD_{\cL_0(Q^+)}$ to make sense, we require that $\cL_0(Q^+)$ be tileable with respect to the sparsest lattice in $\cG_{\vec{n}+1}$. But this is possible for $N$ a multiple of a large enough number, since $\vec{n}$ is bounded. 

For each $l=0,\ldots,k$, let $\cL_l':\RR^{d+k+1}\to \RR^{d+k+1}$ be the linear map given by
$$\cL_l'(\vec{x},y_0,y_1,\ldots,y_k)=(\cL_l \vec{x},0,\ldots,0,y_l,0,\ldots,0).$$
We make the following claim. 
\begin{claim} \label{claim:ctssubset}
    $\cL_0' X+\cdots +\cL_k' X\subseteq Y.$
\end{claim}
Before proving this key claim, we first finish the proof of Lemma~\ref{lem:dense2} assuming it. Let $\cL^*:\RR^{d+k+1}\to \RR^{d+k+1}$ be given by $\cL^*(x,y)=(\cL_0^{-1}x, y)$ for $x\in \RR^d$ and $y\in \RR^{k+1}$. Note that $\cL_0^{-1}\cL_l$ is conjugate to $\cM_l$, the map corresponding to multiplication by $\lambda_l$ on $K$. By Lemma~\ref{lem:diag}, the maps $1, \cL_0^{-1}\cL_1,\ldots,\cL_0^{-1}\cL_k$ are simultaneously diagonalisable over $\CC$, where the diagonal matrix corresponding to $\cL_0^{-1}\cL_l$ has diagonal entries $(\sigma_1(\lambda_l),\ldots,\sigma_d(\lambda_l))$. Therefore, the $\cL^*\cL_l'$ are simultaneously diagonalisable with corresponding diagonal matrix entries 
$(1, \ldots, 1, 1, 0, \ldots, 0)$ for $l = 0$ and 
$(\sigma_1(\lambda_l),\ldots,\sigma_d(\lambda_l),0,\ldots,0,1,0,\ldots,0)$ otherwise. Thus, by Theorem~\ref{thm:cts}, we have
$$\mu(\cL^*\cL_0' X+\cdots +\cL^*\cL_k' X)\geq \prod_{i=1}^d (1+|\sigma_i(\lambda_1)|+|\sigma_i(\lambda_2)|+\cdots+|\sigma_i(\lambda_k)|)\mu(X).$$
Therefore, using Claim~\ref{claim:ctssubset} and the fact that $|\det \cL_0|=N_{K/\QQ}(\frakD)$, 
\begin{align*}
    \mu(Y) &\geq \mu(\cL_0' X+\cdots +\cL_k' X)\\
    &= \frac{1}{|\det(\cL^*)|}\mu(\cL^*\cL_0' X+\cdots +\cL^*\cL_k' X)\\
    &\geq |\det(\cL_0)|\prod_{i=1}^d (1+|\sigma_i(\lambda_1)|+|\sigma_i(\lambda_2)|+\cdots+|\sigma_i(\lambda_k)|)\mu(X)\\
    &= H(\lambda_1,\ldots,\lambda_k)\mu(X).
\end{align*}

By Lemma~\ref{lem:localldvol},
\begin{align*}
    \mu(X) &= \sum_{Q\in \cQ} \Vol(Q)\times \Vol(\LD_Q(A;\cF_{\vec{n}}))\\
    &= \sum_{Q\in \cQ} |A\cap Q| = |A''|\geq (1-2\delta)|A|.
\end{align*}

By the definition of $C_1$ (see Lemma~\ref{lem:consts}), since $A$ lies in the cube $(-N,N)^d$, the sum $A+\cL_0^{-1}\cL_1 A+\cdots+\cL_0^{-1}\cL_k A$ lies in the cube $(-(k+1)C_1N,(k+1)C_1N)^d\subset \RR^d$. There are at most $(4(k+1)C_1)^dM^{d(r+1)}$ different $Q\in \cQ_0$ such that $Q^+$ intersects $(-(k+1)C_1N,(k+1)C_1N)^d$. For simplicity, assume that $D>(4(k+1)C_1)^d$, so that there are at most $DM^{d(r+1)}$ such $Q$. Thus, there are at most $DM^{d(r+1)}$ different $Q\in \cQ_0$ such that $\cL_0 (Q^+)\cap (\cL_0 A+\cdots+\cL_k A)\neq\emptyset$. For each such $Q$, we have 
\begin{align*}
    |\cL_0 (Q^+)\cap (\cL_0 A+\cdots+\cL_k A)|-|\cL_0 Q\cap (\cL_0 A+\cdots+\cL_k A)| &\leq \det(\cL_0)(\Vol(Q^+)-\Vol(Q))\\
    &= O\paren{\frac{DN^d}{M^{(r+1)d+1}}}\\
    &\leq \delta (N/M^{r+1})^d.
\end{align*}
Therefore, again using Lemma~\ref{lem:localldvol}, 
\begin{align*}
    \mu(Y) &= \sum_{Q\in \cQ_0} \Vol(\cL_0 Q)\times (1+2\delta)^{k+1}\Vol(\LD_{\cL_0 (Q^+)}(\cL_0 A+\cdots+\cL_k A;\cG_{\vec{n}+1}))\\
    &= (1+2\delta)^{k+1}\sum_{Q\in \cQ_0} \frac{\Vol(\cL_0 Q)}{\Vol(\cL_0 (Q^+))}|\cL_0 (Q^+)\cap (\cL_0 A+\cdots+\cL_k A)|\\
    &\leq (1+2\delta)^{k+1}\sum_{Q\in \cQ_0} |\cL_0 (Q^+)\cap (\cL_0 A+\cdots+\cL_k A)|\\
    &\leq (1+2\delta)^{k+1}\paren{\sum_{Q\in \cQ_0} |\cL_0 Q\cap (\cL_0 A+\cdots+\cL_k A)|+DM^{d(r+1)}\cdot \delta (N/M^{r+1})^d}\\
    &= (1+2\delta)^{k+1}(|\cL_0 A+\cdots+\cL_k A|+D\delta N^d).
\end{align*}
Thus, we have
\begin{align*}
    |\cL_0 A+\cdots+\cL_k A| &\geq (1+2\delta)^{-(k+1)}\mu(Y)-O_D(\delta)N^d\\
    &= (1-O(\delta))\mu(Y)-O_D(\delta)N^d\\
    &\geq (1-O(\delta))H(\lambda_1,\ldots,\lambda_k)\mu(X)-O_D(\delta)N^d\\
    &\geq (1-O(\delta))H(\lambda_1,\ldots,\lambda_k)(1-2\delta)|A|-O_D(\delta)N^d\\
    &= H(\lambda_1,\ldots,\lambda_k)|A|-O_D(\delta)N^d.
\end{align*}
Since $\delta$ was arbitrary, this proves the lemma.
\end{proof}

In order to complete the proof, we now return to Claim~\ref{claim:ctssubset}.

\begin{proof}[Proof of Claim~\ref{claim:ctssubset}]
    Let $(x_l,y_l)\in X$ for $l=0,\ldots,k$ with $Q_l\in \cQ$ the cube containing $x_l$ and $y_l\in \LD_{Q_l}(A;\cF_{\vec{n}})$. Our aim is to show that $(\sum_l \cL_l x_l,y)\in Y$, where $y=(\pi_1(y_0),\ldots,\pi_{k+1}(y_k))$.

    Let $Q^*\in \cQ_0$ be the cube containing $x:=x_0+\cL_0^{-1}\cL_1 x_1+\cdots + \cL_0^{-1}\cL_k x_k$. Then $\cL_0 x_0+\cdots +\cL_k x_k=\cL_0 x\in \cL_0 Q^*$, so it suffices to show that $y\in (1+2\delta)\LD_{\cL_0 (Q^{*+})}(\cL_0 A+\cdots+\cL_k A;\cG_{\vec{n}+1})$.
    
    Suppose $Q^*=Q_0+t$ for some translate $t\in \frac{N}{M^{r+1}}\ZZ^d$. Then $t=\cL_0^{-1}\cL_1 x_1+\cdots + \cL_0^{-1}\cL_k x_k+t_0$ for some $t_0\in (-\frac{N}{M^{r+1}},\frac{N}{M^{r+1}})^d$. Let $x_k^*=x_k+\cL_k^{-1}\cL_0 t_0$, so that 
    $$x_k^*-x_k=\cL_k^{-1}\cL_0 t_0\in \sqb{-\frac{C_1N}{M^{r+1}},\frac{C_1N}{M^{r+1}}}^d\subseteq \sqb{-\frac{DN}{M^{r+1}},\frac{DN}{M^{r+1}}}^d.$$
    Therefore, if $P_k\in \cP$ is the cube containing $Q_k$ and $Q_k^*\in \cQ_0$ is the cube containing $x_k^*$, we must have $Q_k^*\subset P_k$, since $Q_k\in \cQ$ is at least a distance $DN/M^{r+1}$ away from the boundary of $P_k$. 
    
    Let $R_l$ be the $N/M^{r+2}$-cube containing $x_l$ for each $l=1,\ldots,k-1$ and $R_k^*$ the $N/M^{r+2}$-cube containing $x_k^*$, so that $R_k^*\subset P_k$. Define the following sets:
    \begin{itemize}
        \item $A_0=A\cap Q_0$,
        \item $A_l=A\cap R_l$ for $l=1,\ldots,k-1$,
        \item $A_k=A\cap R_k^*$.
    \end{itemize}

    We have $x_l\in A_l$ for $l=0,\ldots,k-1$, $x_k^*\in A_k$ and $t=\cL_0^{-1}\cL_1 x_1+\cdots +\cL_0^{-1}\cL_k x_k^*$. Since $A_l$ is contained in an $N/M^{r+2}$-cube for $l=1,\ldots,k$, $\cL_0^{-1}\cL_1 A_1+\cdots +\cL_0^{-1}\cL_k A_k$ is contained in a cube of side length $DN/M^{r+2}$ if $D$ is sufficiently large. Since $t\in \cL_0^{-1}\cL_1 A_1+\cdots +\cL_0^{-1}\cL_k A_k$ and $A_0\subseteq Q_0$, we have $A_0+\cL_0^{-1}\cL_1 A_1+\cdots +\cL_0^{-1}\cL_k A_k\subseteq Q_0^{+}+t=Q^{*+}$.

    Suppose $L$ is a lattice such that $Q^+$ is tiled by $L$ for every $Q\in \cQ_0$, such as $L=((N/M^{r+1}+2DN/M^{r+2})\ZZ)^d$. Then $\cL_0(Q^+)$ is tiled by $\cL_0 L$. By repeatedly applying Theorem~\ref{thm:ldsum}, we have that 
    \begin{align*}
        \prod_{l=0}^k \pi_{l+1}(\LD(\cL_l A_l+\cL_0 L;\cG_{\vec{n}+1})) &\subseteq \LD(\cL_0 A_0+\cdots+\cL_k A_k+\cL_0 L;\cG_{\vec{n}+1})\\
        &= \LD_{\cL_0(Q^{*+})}(\cL_0 A_0+\cdots+\cL_k A_k;\cG_{\vec{n}+1}).
    \end{align*}

    We will now show that $|\pi_{l+1}(\LD(\cL_l A_l+\cL_0 L;\cG_{\vec{n}+1}))|\geq (1-\delta)\pi_{l+1}(y_l)$ for all $l$, looking at each of the three cases $l=0$, $1\leq l\leq k-1$ and $l=k$ separately. For $l=0$, we have 
    \begin{align*}
        |\pi_1(\LD(\cL_0 A_0+\cL_0 L;\cG_{\vec{n}+1}))| &= |\pi_1(\LD(A_0+L;\cF_{\vec{n}}))| & \text{by (\ref{eqn:p1})}\\
        &= |\pi_1(\LD_{Q_0^+}(A_0;\cF_{\vec{n}}))|\\
        &= \frac{\Vol(Q_0)}{\Vol(Q_0^+)}|\pi_1(\LD_{Q_0}(A_0;\cF_{\vec{n}}))| & \text{by Lemma~\ref{lem:localldchange}}\\
        &\geq (1-\delta)|\pi_1(\LD_{Q_0}(A;\cF_{\vec{n}}))| & \text{by (\ref{eq:vol})}\\
        &\geq (1-\delta)\pi_1(y_0).
    \end{align*}
For $l=1,\ldots,k-1$, since $Q_l$ is $(M,\delta,n_l,\ldots,n_k)$-regular, by (\ref{eqn:regularity}), we have, for $\vec{n}^{(l)}=(n_1+1,\ldots,n_l+1,n_{l+1},\ldots,n_k)$, that 
    $$|\pi_{l+1}(\LD_{R_l}(A;\cF_{\vec{n}^{(l)}}))|\geq (1-\delta)|\pi_{l+1}(\LD_{Q_l}(A;\cF_{\vec{n}}))|.$$
    Note that $\cL_l^{-1}\cL_0(Q^+)$ is tiled by $\cL_l^{-1}\cL_0 L$ for any $Q\in \cQ_0$. Let $S_l$ be a translate of $\cL_l^{-1}\cL_0(Q^{*+})$ containing $R_l$. Such a translate exists for $M$ sufficiently large since $R_l$ is an $N/M^{r+2}$-cube and $Q^*$ is an $N/M^{r+1}$-cube. Therefore, 
    \begin{align*}
        |\pi_{l+1}(\LD(\cL_l A_l+\cL_0 L;\cG_{\vec{n}+1}))| &\geq |\pi_{l+1}(\LD(A_l+\cL_l^{-1}\cL_0 L;\cF_{\vec{n}^{(l)}}))| & \text{by (\ref{eqn:p2})}\\
        &= |\pi_{l+1}(\LD_{S_l}(A_l;\cF_{\vec{n}^{(l)}}))|\\
        &= |\pi_{l+1}(\LD_{R_l}(A_l;\cF_{\vec{n}^{(l)}}))| & \text{by Lemma~\ref{lem:localldchange}}\\
        &= |\pi_{l+1}(\LD_{R_l}(A;\cF_{\vec{n}^{(l)}}))|\\
        &\geq (1-\delta)|\pi_{l+1}(\LD_{Q_l}(A;\cF_{\vec{n}}))| & \text{by regularity}\\
        &\geq (1-\delta)\pi_{l+1}(y_l).
    \end{align*}
Finally, for $l=k$, similarly define $S_k$ to be a translate of $\cL_k^{-1}\cL_0(Q_k^{*+})$ containing $R_k^*$. Then, we have
    \begin{align*}
        |\pi_{k+1}(\LD(\cL_k A_k+\cL_0 L;\cG_{\vec{n}+1}))| &\geq |\pi_{k+1}(\LD(A_k+\cL_k^{-1}\cL_0 L;\cF_{\vec{n}^{(k)}}))| & \text{by (\ref{eqn:p2})}\\
        &= |\pi_{k+1}(\LD_{S_k}(A_k;\cF_{\vec{n}^{(k)}}))|\\ 
        &= |\pi_{k+1}(\LD_{R_k^*}(A_k;\cF_{\vec{n}^{(k)}}))| & \text{by Lemma~\ref{lem:localldchange}}\\
        &= |\pi_{k+1}(\LD_{R_k^*}(A;\cF_{\vec{n}^{(k)}}))|\\
        &\geq (1-\delta)|\pi_{k+1}(\LD_{P_k}(A;\cF_{\vec{n}}))| & \text{by regularity of $P_k$}\\
        &\geq (1-\delta)|\pi_{k+1}(\LD_{Q_k}(A;\cF_{\vec{n}}))| & \text{by Lemma~\ref{lem:sublocalld}}\\
        &\geq (1-\delta)\pi_{k+1}(y_k).
    \end{align*}
    
    Therefore, we have
    \begin{align*}
        (1-\delta)y &= ((1-\delta)\pi_1(y_0),\ldots,(1-\delta)\pi_{k+1}(y_k))\\
        &\in \LD_{\cL_0(Q^{*+})}(\cL_0 A_0+\cdots+\cL_k A_k;\cG_{\vec{n}+1})\\
        &\subseteq \LD_{\cL_0(Q^{*+})}(\cL_0 A+\cdots+\cL_k A;\cG_{\vec{n}+1}),
    \end{align*}
    which implies that $y\in (1+2\delta)\LD_{\cL_0(Q^{*+})}(\cL_0 A+\cdots+\cL_k A;\cG_{\vec{n}+1})$, as required.
\end{proof}

\section{Sums of linear transformations} \label{sec:linearsums}

In this section, we prove Theorem~\ref{thm:linearsums}, our main result about sums of pre-commuting linear transformations. As mentioned in the introduction, the idea of the proof is to show that the general case reduces to the seemingly special case of sums of algebraic dilates.

\subsection{Algebraic number theory preliminaries}

Recall that, for $\alpha_1,\ldots,\alpha_k\in K$, the denominator ideal $\frakD_{\alpha_1,\ldots,\alpha_k;K}$ is given by
$$\frakD_{\alpha_1,\ldots,\alpha_k;K}:=\setcond{x\in\cO_K}{x\alpha_i\in\cO_K \text{ for all } i=1,\ldots,k}.$$
Abbreviating this again as $\frakD$, the ideal norm $N_{K/\QQ}(\frakD)$ is the index $[\cO_K:\frakD]$. The main result of this short subsection gives an alternative way to compute the norm of the denominator ideal. In the statement, we also use the notation $N_{K/\QQ}(\alpha)$, but this now refers to the field norm of an element $\alpha$ of $K$, which is the product of the conjugates of $\alpha$.

\begin{thm} \label{thm:altdenom}
    Let $\alpha_1,\ldots,\alpha_k\in K$ and consider the polynomial
    \begin{align*}
        F(x_0,\ldots,x_k) &:= N_{K/\QQ}(x_0+x_1\alpha_1+\dots+x_k\alpha_k)\in \QQ[x_0,x_1,\ldots,x_k].
    \end{align*}
    If $D>0$ is the smallest positive integer such that $DF$ has integer coefficients, then $D=N_{K/\QQ}(\frakD)$.
\end{thm}

To prove this, we require a variant of Gauss's lemma over the ring of integers $\cO_K$. We first need a definition. 
\begin{defn}
    Let $F(x)=a_0+a_1 x+\cdots+a_n x^n \in K[x]$. Define the \emph{content} of $F$, denoted by $\cont_K(F)$, to be the fractional ideal $a_0\cO_K+a_1\cO_K+\cdots+a_n\cO_K\subseteq K$. If it is clear from context, we omit the subscript and simply write $\cont(F)$.
\end{defn}
If $F\in \ZZ[x]$, then $\cont_\QQ(F)=c\ZZ$, where $c\in \ZZ$ is the  content of $F$ as used in the usual Gauss's lemma, that is, the greatest common divisor of the coefficients of $F$. 
If $L$ is a field extension of $K$ and $F\in K[x]$, then $\cont_L(F)=\cont_K(F)\cdot \cO_L$. In particular, if $F\in \QQ[x]$, then $\cont_K(F)=\cont_\QQ(F)\cdot \cO_K$. Our variant of Gauss's lemma over $\cO_K$ is now as follows. 

\begin{lem}[Gauss's lemma over $\cO_K$] \label{lem:gauss}
    For any two polynomials $F,G\in K[x]$, $\cont(FG)=\cont(F)\cont(G)$.
\end{lem}
\begin{rmk}
    This result also follows from the Dedekind--Mertens lemma, which says that for any ring $R$ and polynomials $F,G\in R[x]$ there exists a positive integer $n$ such that $\cont(F)^n\cont(FG)=\cont(F)^{n+1}\cont(G)$. Our lemma then follows, since every non-zero fractional ideal in $\cO_K$ is invertible. We give a direct proof here for completeness.
\end{rmk}
\begin{proof}[Proof of Lemma~\ref{lem:gauss}]
    Let $F(x)=a_0+a_1x+\cdots+a_n x^n$ and $G(x)=b_0+b_1 x+\cdots+b_m x^m$. Then their product $F(x)G(x)=c_0+c_1x+\cdots+c_{n+m}x^{n+m}$ has coefficients $c_j=a_0b_j+a_1b_{j-1}+\cdots+a_jb_0$. It is clear that $\cont(FG)\subseteq \cont(F)\cont(G)$. To show that $\cont(FG)\supseteq \cont(F)\cont(G)$, it suffices to show that, for any prime ideal $\frakp\subseteq \cO_K$, $\nu_\frakp(\cont(FG))\leq \nu_\frakp(\cont(F))+\nu_\frakp(\cont(G))$.

    Suppose $\nu_\frakp(\cont(F))=s$ and $\nu_\frakp(\cont(G))=t$. Since $\nu_\frakp(\cont(F))=\min(\nu_\frakp(a_0),\ldots,\nu_\frakp(a_n))$, there exists an index $k$ such that $\nu_\frakp(a_k)=s$. Let $k$ be the smallest such index, so that $\nu_\frakp(a_j)\geq s+1$ for $j=0,\ldots,k-1$. Similarly, let $l$ be the smallest index such that $\nu_\frakp(b_l)=t$, so that $\nu_\frakp(b_j)\geq t+1$ for $j=0,\ldots,l-1$.

    Consider the coefficient $c_{k+l}=\sum_{j=0}^{k+l} a_j b_{k+l-j}$. For $j=k$, the term $a_kb_l$ satisfies $\nu_\frakp(a_kb_l)=\nu_\frakp(a_k)+\nu_\frakp(b_l)=s+t$. For every other $j\neq k$, either $j<k$ (for which $\nu_\frakp(a_j)\geq s+1$) or $j>k$ (for which $\nu_\frakp(b_{k+l-j})\geq t+1$). In either case, we have $\nu_\frakp(a_j b_{k+l-j})\geq s+t+1$. Therefore, $\nu_\frakp(c_{k+l})=s+t$, so we have $\nu_\frakp(\cont(FG))\leq s+t$, as required.
\end{proof}

Observe that we may similarly define content for multivariate polynomials $F\in K[x_0,\ldots,x_k]$ and our variant of Gauss's lemma then also holds for multivariate polynomials. Indeed, the set of coefficients for $F(x_0,\ldots,x_k)$ is the same as for $F(x,x^{N_1},\ldots,x^{N_k})$ for sufficiently large $N_k\gg N_{k-1}\gg\cdots \gg N_1\gg 1$. Thus, the content of $F$ is the same as the content of $F(x,x^{N_1},\ldots,x^{N_k})$, so we may apply the univariate case.

\begin{proof}[Proof of Theorem~\ref{thm:altdenom}]
    Let $\sigma_1,\ldots,\sigma_d:K\to \CC$ be the complex embeddings of $K$, with $\sigma_1$ being the identity. Then 
    \[F(x_0,\ldots,x_k)=N_{K/\QQ}(x_0+x_1\alpha_1+\dots+x_k\alpha_k)=\prod_{i=1}^d (x_0+x_1\sigma_i(\alpha_1)+\cdots+x_k\sigma_i(\alpha_k)).\]
    Let $K'\subseteq \CC$ be the smallest field containing $\sigma_1(K),\ldots,\sigma_d(K)$, that is, $K'$ is the normal closure of $K$ over $\QQ$. By definition, $D^{-1}\ZZ=\cont_\QQ(F)$, so we have $\cont_{K'}(F)=D^{-1}\cO_{K'}$. On the other hand, by Lemma~\ref{lem:gauss},
    \[\cont_{K'}(F)=\prod_i \cont_{K'}(x_0+x_1\sigma_i(\alpha_1)+\cdots+x_k\sigma_i(\alpha_k)).\]
    For any subset $S\subset K'$, denote by $S\cO_{K'}$ the $\cO_{K'}$-fractional ideal generated by $S$, i.e., the set of elements of the form $s_1a_1+\cdots+s_na_n$ for some non-negative integer $n$, $s_i\in S$ and $a_i\in \cO_{K'}$. Then
    \begin{align*}
        \cont_{K'}(x_0+x_1\sigma_i(\alpha_1)+\cdots+x_k\sigma_i(\alpha_k)) &= \cO_{K'}+\sigma_i(\alpha_1)\cO_{K'}+\cdots+\sigma_i(\alpha_k)\cO_{K'}\\
        &= \sigma_i(\cO_{K'}+\alpha_1\cO_{K'}+\cdots+\alpha_k\cO_{K'})\\
        &= \sigma_i(\frakD^{-1} \cO_{K'})=\sigma_i(\frakD^{-1}) \cO_{K'}.
    \end{align*}
    Multiplying over all $i$, we get $\cont_{K'}(F)=\prod_i \sigma_i(\frakD^{-1})\cO_{K'}=N_{K/\QQ}(\frakD^{-1})\cO_{K'}$, where the last equality follows from the fact that, for any $\cO_K$-fractional ideal $\fraka$, we have $\prod_i \sigma_i(\fraka)\cO_{K'}=N_{K/\QQ}(\fraka)\cO_{K'}$. Indeed, this holds when $\fraka=\alpha\cO_K$ is principal (since $N_{K/\QQ}(\alpha)=\prod_i \sigma_i(\alpha)$ and $N_{K/\QQ}(\alpha\cO_K)=N_{K/\QQ}(\alpha)\cO_K$), so the general case follows since $\fraka^m$ is always principal for some $m\geq 0$. 
    Therefore, we have $D^{-1}\cO_{K'}=N_{K/\QQ}(\frakD^{-1})\cO_{K'}$, so $D=N_{K/\QQ}(\frakD)$, as required. 
\end{proof}

\subsection{Pre-commuting matrices}

In this subsection, we prove the following result, which allows us to regard sums of pre-commuting linear transformations as sums of algebraic dilates.

\begin{thm} \label{thm:precommchar}
    Suppose $\cL_0,\ldots,\cL_k\in \Mat_d(\ZZ)$ are non-zero, pre-commuting, irreducible and coprime. Then they are invertible over $\QQ$ and there exist a number field $K$ with $\deg(K/\QQ)=d$, $\lambda_1,\ldots,\lambda_k\in K$ and a $\QQ$-isomorphism $\Phi:K\to \QQ^d$ such that $|\det(\cL_0)|=N_{K/\QQ}(\frakD_{\lambda_1,\ldots,\lambda_k;K})$ and, for all $u\in \QQ^d$ and $l=1,\ldots,k$,
    $$\cL_0^{-1}\cL_l(u)=\Phi(\lambda_l\cdot \Phi^{-1}(u)).$$
\end{thm}

Before proving this theorem, we prove a structure theorem for pairwise commuting matrices with no non-trivial common invariant subspace. 
A folklore result (e.g., \cite[Corollary 2.4.6.4]{HJ85}) says that pairwise commuting maps are simultaneously upper-triangularisable over $\CC$ and so have a common eigenvector.

\begin{lem} \label{lem:commupper}
    If $\cL_1,\ldots,\cL_k\in \Mat_d(\CC)$ are pairwise commuting matrices, then they have a common eigenvector $v\in \CC^d$.
\end{lem}

Suppose $\lambda_1,\ldots,\lambda_k$ generate the field $K$ and multiplication by these elements correspond to the matrices $\cM_1,\ldots,\cM_k\in \Mat_d(\QQ)$, as spelled out in Section~\ref{sec:mapZd}. Then $\cM_1,\ldots,\cM_k$ are pairwise commuting and have no non-trivial common invariant subspace over $\QQ$. Conversely, we now show that any such tuple of matrices $\cM_1,\ldots,\cM_k$ arise from some $\lambda_1,\ldots,\lambda_k$ in some field $K$.

\begin{lem} \label{lem:irredstruct}
    Suppose $\cM_1,\ldots,\cM_k\in \Mat_d(\QQ)$ are pairwise commuting and have no non-trivial common invariant subspace over $\QQ$. Then there is a number field $K$ of degree $d$, algebraic numbers $\lambda_1,\ldots,\lambda_k\in K$ and a $\QQ$-isomorphism $\Phi:K\to \QQ^d$ such that
    \begin{enumerate}
        \item $K=\QQ(\lambda_1,\ldots,\lambda_k)$,
        \item for $l=1,\ldots,k$, the map $\Phi^{-1}\cM_l\Phi:K\to K$ is given by multiplication by $\lambda_l$.
    \end{enumerate}
\end{lem}

\begin{proof}
    By Lemma~\ref{lem:commupper}, there is a common eigenvector $v\in \CC^d$ for $\cM_1,\ldots,\cM_k$ with eigenvalues $\lambda_1,\ldots,\lambda_k$. Let $K=\QQ(\lambda_1,\ldots,\lambda_k)$. Then we may assume without loss of generality that $v\in K^d$. Let $d'=\deg(K/\QQ)$ and $\sigma_1,\ldots,\sigma_{d'}:K\to \CC$ be the complex embeddings. Then $\sigma_i(v)$ is also a common eigenvector for $\cM_1,\ldots,\cM_k$ with eigenvalues $\sigma_i(\lambda_1),\ldots,\sigma_i(\lambda_k)$. Since the tuples $(\sigma_i(\lambda_1),\ldots,\sigma_i(\lambda_k))$ are distinct for $i=1,\ldots,d'$ (as each tuple uniquely determines the map $\sigma_i:K\to K$), the common eigenvectors $\sigma_1(v),\ldots,\sigma_{d'}(v)\in \CC^d$ are linearly independent, implying that $d'\leq d$.

    Note that $U=\ang{\sigma_1(v),\ldots,\sigma_{d'}(v)}\cap \QQ^d$ is a common invariant subspace. We will show that this subspace has dimension $d'$. Clearly it has dimension at most $d'$. Let $\lambda\in K$ be a generator, so that $1,\lambda,\ldots,\lambda^{d'-1}$ is a $\QQ$-basis for $K$. For $i=0,\ldots,d'-1$, let $u_i=\sigma_1(\lambda^iv)+\cdots+\sigma_{d'}(\lambda^i v)=\Tr_{K/\QQ} (\lambda^i v)\in \QQ^d$, so that $u_i\in U$. We claim that $u_0,\ldots,u_{d'-1}$ are linearly independent. Indeed, since $u_i=\sigma_1(\lambda)^i\sigma_1(v)+\cdots+\sigma_{d'}(\lambda)^i\sigma_{d'}(v)$ and $\sigma_1(v),\ldots,\sigma_{d'}(v)$ are linearly independent, it suffices to show that the $d'\times d'$ matrix
    \[
    \begin{pmatrix}
        1 & \sigma_1(\lambda) & \cdots & \sigma_1(\lambda)^{d'-1}\\
        1 & \sigma_2(\lambda) & \cdots & \sigma_2(\lambda)^{d'-1}\\
        \vdots & \vdots & \ddots & \vdots\\
        1 & \sigma_{d'}(\lambda) & \cdots & \sigma_{d'}(\lambda)^{d'-1}\\
    \end{pmatrix}
    \]
    is non-singular. But this is true since it is a Vandermonde matrix and $\sigma_1(\lambda),\ldots,\sigma_{d'}(\lambda)$ are distinct, which in turn follows from the fact that $\lambda$ generates $K$.

    Since $\cM_1,\ldots,\cM_k$ have no non-trivial common invariant subspace, we must have $d'=d$. We deduce that $\cM_1,\ldots,\cM_k$ are simultaneously diagonalisable with eigenvalue tuples $(\sigma_i(\lambda_1),\ldots,\sigma_i(\lambda_k))$ for $i=1,\ldots,d$. Define the linear map $\Phi:K\to \QQ^d$ as follows. First, let $e_1\in \QQ^d$ be any non-zero vector and set $\Phi(1)=e_1$. Then, for any $\alpha\in K$, express $\alpha$ as a polynomial $P(\lambda_1,\ldots,\lambda_k)$ with rational coefficients in $\lambda_1,\ldots,\lambda_k$. Such a polynomial exists since $\lambda_1,\ldots,\lambda_k$ generate $K$, though it is not unique. Set $\Phi(\alpha)=P(\cM_1,\ldots,\cM_k)e_1$. 
    Observe that this is independent of the choice of $P$. Indeed, it suffices to show that if $P$ is a polynomial with rational coefficients such that $P(\lambda_1,\ldots,\lambda_k)=0$, then $P(\cM_1,\ldots,\cM_k)=0$. The matrix $P(\cM_1,\ldots,\cM_k)$ is diagonalisable with eigenvalues $P(\sigma_i(\lambda_1),\ldots,\sigma_i(\lambda_k))$ for $i=1,\ldots,d$. Since $P$ has rational coefficients, we have $P(\sigma_i(\lambda_1),\ldots,\sigma_i(\lambda_k))=\sigma_i(P(\lambda_1,\ldots,\lambda_k))=0$ for all $i$ and thus $P(\cM_1,\ldots,\cM_k)=0$.

    Notice that if $P(\lambda_1,\ldots,\lambda_k)=\alpha\neq 0$, then $P(\cM_1,\ldots,\cM_k)$ is diagonalisable with eigenvalues $\sigma_1(\alpha),\ldots,\sigma_k(\alpha)$. Thus, $P(\cM_1,\ldots,\cM_k)$ is non-singular and so $\Phi(\alpha)\neq 0$. It follows that $\Phi$ is injective and hence an isomorphism. Since $\Phi$ also satisfies condition 2, the result follows. 
\end{proof}

We now return to Theorem~\ref{thm:precommchar}.

\begin{proof}[Proof of Theorem~\ref{thm:precommchar}]
    Let $\cP\in \GL_d(\QQ)$ be such that $\cP\cL_0,\ldots,\cP\cL_k$ are pairwise commuting. Since $\cL_0,\ldots,\cL_k$ are irreducible, $\cP\cL_0,\ldots,\cP\cL_k$ have no non-trivial common invariant subspace. Let $K,\lambda_0,\ldots,\lambda_k,\Phi$ be as in the conclusion of Lemma~\ref{lem:irredstruct} when applied to $\cP\cL_0,\ldots,\cP\cL_k$.

    Since all of $\cL_0,\ldots,\cL_k$ are non-zero, all of  $\lambda_0,\ldots,\lambda_k$ are also non-zero. In particular, $\cL_0, \ldots,\cL_k$ are invertible over $\QQ$. Since $I,(\cP\cL_0)^{-1}\cP\cL_1,\ldots,(\cP\cL_0)^{-1}\cP\cL_k$ are also pairwise commuting, we may assume without loss of generality that $\cP=\cL_0^{-1}$, so we have $\lambda_0=1$. From Lemma~\ref{lem:irredstruct}, we have that, for all $u\in \QQ^d$ and $l=1,\ldots,k$,
    $$\cL_0^{-1}\cL_l(u)=\Phi(\lambda_l\cdot \Phi^{-1}(u)).$$

    It remains to show that $|\det(\cL_0)|=N_{K/\QQ}(\frakD)$, where $\frakD=\frakD_{\lambda_1,\ldots,\lambda_k;K}$. Consider the integer polynomial
    \begin{align*}
        G(x_0,\ldots,x_k) &= \det(x_0\cL_0+\cdots+x_k\cL_k)\\
        &= \det(\cL_0)\det(x_0+x_1\cL_0^{-1}\cL_1+\cdots+x_k\cL_0^{-1}\cL_k)\\
        &= \det(\cL_0)N_{K/\QQ}(x_0+x_1\lambda_1+\cdots+x_k\lambda_k).
    \end{align*}
    By Theorem~\ref{thm:altdenom}, $N_{K/\QQ}(\frakD)$ is the smallest positive integer required to scale $N_{K/\QQ}(x_0+x_1\lambda_1+\cdots+x_k\lambda_k)$ into an integer polynomial. Thus, $N_{K/\QQ}(\frakD)$ divides $\det(\cL_0)$, so that $|\det(\cL_0)|\geq N_{K/\QQ}(\frakD)$.

    Let $\Phi_1:\cO_K\to \ZZ^d$ and $\Phi_2:\frakD\to \ZZ^d$ be linear isomorphisms of lattices, so that $\Phi_1\circ \Phi_2^{-1}:\ZZ^d\to \ZZ^d$ is a $d\times d$ integer matrix with absolute determinant $N_{K/\QQ}(\frakD)$. Since $\lambda_l\cdot \frakD\subseteq \cO_K$, we have $\Phi_1(\lambda_l\cdot \Phi_2^{-1}(u))\in \ZZ^d$ for any $u\in \ZZ^d$. Thus, the linear map $u\mapsto \Phi_1(\lambda_l\cdot \Phi_2^{-1}(u))$ is represented by an integer matrix. But this map is also equal to the composition $(\Phi_1\circ \Phi^{-1})(\cL_0^{-1}\cL_l)(\Phi\circ \Phi_2^{-1})$. Since $\cL_0,\ldots,\cL_k$ are coprime, we have $|\det((\Phi_1\circ \Phi^{-1})(\cL_0^{-1})(\Phi\circ \Phi_2^{-1}))|\geq 1$, which implies that $N_{K/\QQ}(\frakD)=|\det(\Phi_1\circ \Phi_2^{-1})|\geq |\det(\cL_0)|$. Therefore, $|\det(\cL_0)|= N_{K/\QQ}(\frakD)$, as required.
\end{proof}

\subsection{Sums of pre-commuting linear transformations}

We are now ready to prove Theorem~\ref{thm:linearsums}, our main result about sums of linear transformations, which we restate for convenience. Recall that if $\cL_0,\ldots,\cL_k\in \Mat_d(\ZZ)$ are non-zero, pre-commuting, irreducible and coprime and the polynomial $G(x_0,\ldots,x_k)=\det(x_0\cL_0+\cdots+x_k\cL_k)$ factorises as 
$$G(x_0,\ldots,x_k)=\prod_{i=1}^d (a_{0i}x_0+\cdots+a_{ki}x_k),$$
then $H(\cL_0,\ldots,\cL_k)$ is defined by
$$H(\cL_0,\ldots,\cL_k)=\prod_{i=1}^d (|a_{0i}|+\cdots+|a_{ki}|).$$
The statement that we wish to prove is then as follows.

\begin{thm} \label{thm:linearsums2}
    Suppose that $\cL_0,\ldots,\cL_k\in \Mat_d(\ZZ)$ are pre-commuting, irreducible and coprime. Then 
    $$|\cL_0 A+\cdots+\cL_k A|\geq H(\cL_0,\ldots,\cL_k)|A|-o(|A|)$$
    for all finite subsets $A$ of $\ZZ^d$.
\end{thm}
\begin{proof}
    We may assume that $\cL_0,\ldots,\cL_k$ are non-zero. Let $\lambda_1, \dots, \lambda_k$ be the algebraic numbers given by applying Theorem~\ref{thm:precommchar} to $\cL_0,\ldots,\cL_k$. Then the required estimate follows from Theorem~\ref{thm:main} provided only that $H(\cL_0,\ldots,\cL_k)=H(\lambda_1,\ldots,\lambda_k)$. To check this, note that the map corresponding to multiplication by $\lambda_l$ is similar to $\cL_0^{-1}\cL_l$, so we have
    \begin{align*}
        G(x_0,\ldots,x_k) &= \det(x_0\cL_0+\cdots+x_k\cL_k)\\
        &= \det(\cL_0)\det(x_0I+x_1\cL_0^{-1}\cL_1+\cdots+x_k\cL_0^{-1}\cL_k)\\
        &= \det(\cL_0)N_{K/\QQ}(x_0+x_1\lambda_1+\cdots+x_k\lambda_k)\\
        &= \det(\cL_0)\prod_{i=1}^d (x_0+\sigma_i(\lambda_1)x_1+\cdots+\sigma_i(\lambda_k)x_k).
    \end{align*}
    Therefore,
    \begin{align*}
        H(\cL_0,\ldots,\cL_k) &= |\det(\cL_0)|\prod_{i=1}^d (1+|\sigma_i(\lambda_1)|+\cdots+|\sigma_i(\lambda_k)|)\\
        &= N_{K/\QQ}(\frakD_{\lambda_1,\ldots,\lambda_k;K})\prod_{i=1}^d (1+|\sigma_i(\lambda_1)|+\cdots+|\sigma_i(\lambda_k)|)\\
        &= H(\lambda_1,\ldots,\lambda_k),
    \end{align*}
    completing the proof.
\end{proof}

\section{Concluding remarks} \label{sec:conc}

\noindent
{\bf Lower-order terms.}
A close inspection of our arguments shows that the $o(|A|)$ term in our bound
\[|A+\lambda_1\cdot A+\cdots+\lambda_k\cdot A|\geq H(\lambda_1,\ldots,\lambda_k)|A|-o(|A|)\]
can be taken to be $O(|A|/\sqrt{\log_{(k)} |A|})$, where $\log_{(k)}$ is the $k$-times iterated logarithm. This is clearly not best possible. The lower bound in Section~\ref{sec:low} suggests that one should be able to improve the error term to $O(|A|^{1-1/d})$, where $d = \deg(K/\QQ)$, though this is likely to be difficult. Given this, it would already be interesting to obtain $O(|A|^{1- \sigma})$ for some $\sigma$ depending only on $\lambda_1, \dots, \lambda_k$. In the particular case where $k = 1$ and $\lambda$ is of the form $(p/q)^{1/d}$, this was already achieved in our earlier paper~\cite{CL23}. However, the methods of that paper and this one are quite orthogonal, so a novel approach is likely to be necessary for the general case.

\vspace{3mm}
\noindent
{\bf An interesting example.}
The main problem left open by this paper is to prove an analogue of Theorem~\ref{thm:main} when the matrices
$\cL_0,\ldots,\cL_k\in \Mat_d(\ZZ)$ are not necessarily pre-commuting. Our own attentions in this direction have focused on the specific example where 
$$\cL_0=\begin{pmatrix}
    0 & 1 & 0\\
    -1 & 0 & 0\\
    0 & 0 & 0
\end{pmatrix},\quad \cL_1=\begin{pmatrix}
    0 & 0 & 1\\
    0 & 0 & 0\\
    -1 & 0 & 0
\end{pmatrix},\quad \cL_2=\begin{pmatrix}
    0 & 0 & 0\\
    0 & 0 & 1\\
    0 & -1 & 0
\end{pmatrix}.$$
These matrices can be shown to be irreducible and coprime, though they are not pre-commuting. We believe that 
\[|\cL_0 A+\cL_1 A+\cL_2 A|\geq 8|A|-o(|A|)\]
for all finite $A \subset \mathbb{Z}^3$, with the box $[0, N)^3$ showing that this would be asymptotically best possible. However, we were unable to even prove that there is some $C > 0$ such that 
\[|\cL_0 A+\cL_1 A+\cL_2 A|\geq C|A|\]
for all finite $A \subset \mathbb{Z}^3$. Resolving this issue would be a promising first step towards understanding the general problem.


\begin{thebibliography}{}

\bibitem{BS14}
A. Balog and G. Shakan, On the sum of dilations of a set, {\it Acta Arith.} {\bf 164} (2014), 153--162.




\bibitem{BG13}
E. Breuillard and B. Green, Contractions and expansion, {\it Eur. J. Combin.} {\bf 34} (2013), 1293--1296.

\bibitem{BJM24}
A.~L.~Bruch, Y.~Jing and A.~Mudgal, Brunn--Minkowski type estimates for certain discrete sumsets, preprint available at arXiv:2409.05638 [math.CO].

\bibitem{B08}
B. Bukh, Sums of dilates, {\it Combin. Probab. Comput.} {\bf 17} (2008), 627--639.

\bibitem{CF18}
Y.-G. Chen and J.-H. Fang, Sums of dilates in the real numbers, {\it Acta Arith.} {\bf 182} (2018), 231--241.



\bibitem{CL22}
D. Conlon and J. Lim, Sums of transcendental dilates, {\it Bull. London Math. Soc.} {\bf 55} (2023), 2400--2406.


\bibitem{CL23}
D. Conlon and J. Lim, Sums of linear transformations, to appear in {\it Trans. Amer. Math. Soc.}


\bibitem{Fr73}
G. A. Freiman, 
{\bf Foundations of a structural theory of set addition}, Translations of Mathematical Monographs, Vol. 37, American Mathematical Society, Providence, R.I., 1973. 






\bibitem{HJ85}
R. A. Horn and C. R. Johnson, {\bf Matrix Analysis}, Cambridge University Press, Cambridge, 1985.


\bibitem{J48}
F. John, Extremum problems with inequalities as subsidiary conditions, in Studies and Essays Presented to R. Courant on his 60th Birthday, 187--204,
Interscience Publishers, New York, 1948.

\bibitem{KL06}
S. Konyagin and I. \L aba, Distance sets of well-distributed planar sets for polygonal norms, {\it Israel J. Math.} {\bf 152} (2006), 157--179.

\bibitem{KP23}
D. Krachun and F. Petrov, On the size of $A+\lambda A$ for algebraic $\lambda$, {\it Mosc. J. Comb. Number Theory} {\bf 12} (2023), 117--126.

\bibitem{KP24}
D. Krachun and F. Petrov, Tight lower bound on $|A+\lambda A|$ for algebraic integer $\lambda$, preprint available at arXiv:2311.09399 [math.CO].


\bibitem{M19}
A. Mudgal, Sums of linear transformations in higher dimensions, {\it Q. J. Math.} {\bf 70} (2019), 965--984.





\bibitem{S08}
T. Sanders, Appendix to ``Roth's theorem on progressions revisited'' by J. Bourgain, {\it J. Anal. Math.} {\bf 104} (2008), 193--206.

\bibitem{S12}
T. Sanders, On the Bogolyubov--Ruzsa lemma, {\it Anal.~PDE} {\bf 5} (2012), 627--655.

\bibitem{Sch11}
T. Schoen, Near optimal bounds in Freiman's theorem, {\it Duke Math. J.} {\bf 158} (2011), 1--12.

\bibitem{S16}
G. Shakan, Sum of many dilates, {\it Combin.~Probab.~Comput.} {\bf 25} (2016), 460--469.

\bibitem{SL22}
D. Singhal and Y. Lin, Primes in denominators of algebraic numbers, {\it Int. J. Number Theory} {\bf 20} (2024), 327--348.

\bibitem{TV06}
T. Tao and V. Vu, {\bf Additive combinatorics}, Cambridge Studies in Advanced Mathematics, 105, Cambridge University Press, Cambridge, 2006.

\bibitem{TV08}
T. Tao and V. Vu, John-type theorems for generalized arithmetic progressions and iterated sumsets, {\it Adv. Math.} {\bf 219} (2008),  428--449.





\end{thebibliography}
\end{document}